\documentclass[fleqn,a4paper]{article} 

\usepackage{geometry} 
\usepackage[utf8]{inputenc}

\usepackage{fancyhdr} 
\usepackage{sectsty}
\allsectionsfont{\sffamily\mdseries\upshape} 
\usepackage[title]{appendix}

\usepackage[nottoc,notlof,notlot]{tocbibind} 
\usepackage[titles,subfigure]{tocloft} 

\usepackage{graphicx} 
\usepackage{subfig}
\usepackage[labelfont=bf,textfont=it]{caption}

\usepackage{color} 
\usepackage{floatrow}

\usepackage{amsthm}
\newtheorem{theorem}{Theorem}[section]
\newtheorem{remark}{Remark}
\newtheorem{lemma}{Lemma}[subsection]

\newtheorem{assump}{Assumption}[section]

\usepackage{amsmath}
\usepackage{amsfonts}
\usepackage{amssymb}

\usepackage{array} 

\usepackage{tabularx}
\usepackage{booktabs} 

\usepackage{todonotes}
\usepackage{verbatim} 
\usepackage{paralist} 

\numberwithin{equation}{section}

\newcommand{\sg}{\sigma}

\newcommand{\otl}[2]{\frac{d #1}{d #2}}
\newcommand{\ptl}[2]{\ensuremath{\frac{\partial #1}{\partial #2}}}
\newcommand{\ptld}[2]{\ensuremath{\frac{\partial^{2} {#1}}{\partial {#2}^{2}}}}

\newcommand{\inc}[2]{\includegraphics[width=#2]{./figures/#1}}

\newcommand{\etal}{\textit{et al. }}

\newcommand{\RR}{\ensuremath{\mathbb{R}}}

\newcommand{\bb}[1]{\ensuremath{\mathbb{#1}}}
\newcommand{\bfa}[1]{\mbox{\boldmath $ #1 $}}

\newcommand{\Qgm}{\ensuremath{Q^{\Gamma}_{T}}}
\newcommand{\Qxi}{\ensuremath{Q^{\Xi}_{T}}}
\newcommand{\Qod}{\ensuremath{Q^{\partial \mathcal{D}}_{T}}}
\newcommand{\Qom}{\ensuremath{Q^{\Omega}_{T}}}

\title{A Conservative Finite Element ALE Scheme for Mass-Conserving Reaction-Diffusion Equations on Evolving Two-Dimensional Domains}
\author{J. A. Mackenzie$^{*}$\hspace{0.25cm} C. F. Rowlatt\footnote{Department of Mathematics and Statistics, University of Strathclyde, Glasgow G1 1XH. Corresponding author JAM: email j.a.mackenzie@strath.ac.uk}\hspace{0.35cm} R. H. Insall\footnote{Cancer Research UK Beatson Institute, Garscube Estate, Switchback Road, Glasgow, G61 1BD.}}
         
\begin{document}
  \maketitle

  \begin{abstract}
  	Mass-conservative reaction-diffusion systems have recently been proposed as a general framework to describe intracellular pattern formation. These systems have been used to model the conformational switching of proteins as they cycle from an inactive state in the cell cytoplasm, to an active state at the cell membrane. The active state then acts as input to downstream effectors. The paradigm of activation by recruitment to the membrane underpins a range of biological pathways – including G-protein signalling, growth control through Ras and PI 3-kinase, and cell polarity through Rac and Rho; all activate their targets by recruiting them from the cytoplasm to the membrane. Global mass conservation lies at the heart of these models reflecting the property that the total number of active and inactive forms, and targets, remains constant. Here we present a conservative arbitrary Lagrangian Eulerian (ALE) finite element method for the approximate solution of systems of bulk-surface reaction-diffusion equations on an evolving two-dimensional domain. Fundamental to the success of the method is the robust generation of bulk and surface meshes. For this purpose, we use a moving mesh partial differential equation (MMPDE) approach. Global conservation of the fully discrete finite element solution is established independently of the ALE velocity field and the time step size. The developed method is applied to model problems with known analytical solutions; these experiments indicate that the method is second-order accurate and globally conservative. The method is further applied to a model of a single cell migrating in the presence of an external chemotactic signal.
  \end{abstract}

  \begin{center}
    \small{{\bf Keywords}}
  \end{center}
  \vspace{-0.25cm}
  \noindent
  {\small Mass-conservative reaction-diffusion systems, cell migration, chemotaxis, conservation, evolving finite elements, ALE methods, moving mesh methods.}
  \vspace{0.25cm}
  \begin{center}
  \end{center}

  \section{Introduction}
  \label{sec:intro}
  Signal processing at the cell membrane is of fundamental importance for eukaryotic cells and is essential in various biological processes, such as embryonic development, immune cell function and cancer metastasis. Signalling is frequently controlled by redistributing components from an inactive form, present in the cytoplasm, to the membrane. This spatial redistribution may of itself be sufficient to activate them (by bringing them into proximity with their targets, for example) or it may promote their activation (by bringing them within reach of an activating kinase, for example). Although a wide variety of biochemical processes, molecules and proteins are involved in signal processing, and subsequent cell polarisation (establishment of a front and a back), recent studies suggest that the Rho family of GTPases (which include RhoA, Cdc42 and Rac) and their effectors are of key importance. Signalling mediators such as the Rho GTPases are small, monomeric, proteins which behave like molecular switches, flipping between inactive and active conformations \cite{edelstein-keshet_2016}. The inactive forms, which are bound to guanosine diphosphate (GDP), predominate in resting cells. Moreover, within the cytoplasm the inactive forms are sequestered by guanine nucleotide dissociation inhibitors (GDI), which act as inhibitors of activation. The active forms, which are bound to guanosine triphosphate (GTP), bind strongly to effectors (for example the kinase PAK). As the effectors are recruited to the membrane they too become activated. Together the Rho GTPases and their effectors trigger downstream signalling events, such as actin polymerisation via the Arp2/3 complex, which lead to cell migration. Over the time scales of polarisation and cellular migration, the Rho family of GTPases are conserved quantities in that the total amount of active (membrane bound) and inactive (cytoplasmic and membrane bound) species remains constant. This suggests that the correct mathematical framework for modelling GTPase activity is mass-conservative reaction-diffusion (McRD) systems \cite{halatek_2018-1, halatek_2018}.
  
  
  Early McRD models were based on the assumption that the cell could be represented as a one-dimensional slice through the cell from front to back and that the different physical locations of the protein conformations were accounted for by the use of different diffusivities, i.e the membrane bound conformations having a smaller diffusivity compared to the cytoplasmic conformations. Examples of models of this type include Narang \etal \cite{narang_2001}, Subramanian and Narang \cite{subramanian_2004}, Otsuji \cite{otsuji_2007}, and Mori \etal \cite{mori_2008}. Recent research has tried to more faithfully account for the different spatial locations of the inactive and active GTPases.  So-called bulk-surface models consist of a combination of equations posed on the cell membrane (surface region), equations for cytoplasmic diffusion (bulk region), and Robin-type flux boundary conditions coupling the membrane bound and cytoplasmic species. For example, Cusseddu \etal \cite{cusseddu_2019} recently proposed a bulk-surface extension of the classical wave-pinning (WP) model \cite{mori_2008} to investigate cell polarisation in three dimensions.  Giese \etal \cite{giese_2015, giese_2018} constructed a two-dimensional model to study the effect of cell size and cell shape on the signalling pathway from the cell membrane to the nucleus. Diegmiller \etal \cite{diegmiller_2018} used a bulk-surface model to look at pattern formation on the surface of a spherical geometry. The effect of geometry on a pattern formation in a two-dimensional setting was also investigated using a bulk-surface model by Thalmeier \etal \cite{thalmeier_2016}.  All of the models have become amenable to numerical computation using recent developments in numerical techniques to approximate the solution of PDEs on surfaces \cite{dziuk_2007, dziuk_2013, elliott_2012, elliott_2012-1, elliott_2013, madzvamuse_2015, madzvamuse_2016}. 
  
  One area where bulk-surface McRD models have not been applied extensively is the study of cell migration. The main reason for this is the additional computational challenges which arise from the time dependence of the bulk and surface domains and the need for the numerical solution to be mass-conservative. In \cite{novak_2014} the authors presented a finite volume method using a fixed background mesh. As their method is based on a finite volume spatial discretisation, by construction it is both locally and globally conservative. However, the spatial rate of convergence was seen to be less than second order on a range of test cases with known analytical solutions. Strychalski \etal \cite{strychalski_2010-1, strychalski_2010} used a cut-cell approach, again with a fixed background mesh, and applied their method to a conservative reaction-diffusion system on an evolving domain. However, conservation was not maintained at the discrete level although the error in conservation could be reduced by grid refinement. Other popular numerical methods include: immersed boundary methods (see e.g. \cite{vanderlei_2011, campbell_2017, campbell_2018}); and phase-field methods (see e.g. \cite{levine_2006, camley_2013, moure_2017}).
  
One approach to ensure global conservation is to use domain fitted meshes which follow the evolving cell membrane and cytoplasm. In \cite{macdonald_2016} we introduced a moving finite element method for the solution of bulk-surface PDEs on evolving two-dimensional bulk domains. An adaptive moving mesh based on the solution of moving mesh PDEs  \cite{huang_1994, huang_1998} was coupled to a novel aproach for moving a mesh on an evolving curve. The moving mesh method was further extended to higher degree of temporal accuracy and applied to forced curve shortening problems in \cite{mackenzie_2019}. The main aim of the work presented here is to investigate the mass conservation property of a fully discrete finite element approximation of an ALE reformulation of a McRD bulk-surface system on an evolving domain. For efficiency and accuracy it is also important that the grid generation procedure is robust. We therefore also present a refinement of the finite element technique in \cite{macdonald_2016} to solve the MMPDEs which results in a much more robust procedure. 
  
The layout of the rest of this paper is as follows: In Section \ref{sec:rd} we introduce a class of bulk-surface reaction-diffusion systems on evolving domains and in Section \ref{sec:weakale} we present a weak ALE reformulation and show that the continuous system is globally conservative independent of the ALE velocity. A conservative FEM discretisation is introduced in Section \ref{sec:fem} and a second-order two-stage time integration scheme is given in Section \ref{sec:tdisc} for the numerical solution of the coupled ODE system obtained after spatial discretisation. The grid generation procedure is discussed in Section \ref{sec:meshgen}. In Section \ref{sec:conserv} we prove the discrete FEM scheme is globally conservative independently of the ALE velocity and time step size. In Section \ref{sec:nexp} some numerical experiments verifying the global conservation properties of the FEM scheme, as well as spatial convergence rates, are given. We finally apply the method to a bulk-surface form of the WP model \cite{cusseddu_2019} for cell polarisation and chemotaxis on a moving domain.

  \section{Reaction-diffusion systems on evolving domains}
  \label{sec:rd}
    The physical layout for the cell migration simulations considered later is shown in Fig. \ref{fig:move_domains}. We assume that the cell moves through a fixed frame of reference $\Lambda$. Let the time interval be denoted by $I = (0,T]\subset \mathbb{R}$ such that the closure is denoted $\overline{I} = [0,T]$. For each $t\in \overline{I}$, let $\mathcal{D}(t)\subset \mathbb{R}^2$ be a simply connected domain containing a smooth, closed curve $\Gamma(t)$ which separates $\mathcal{D}(t)$ into an interior region $\Omega(t)\subset \mathbb{R}^2$ and an exterior region $\Xi(t)\subset \mathbb{R}^2$ such that $\mathcal{D}(t) = \Omega(t) \cup \Xi(t)$. Note that $\partial \Omega(t) = \Gamma(t)$ and $\partial \Xi(t) = \Gamma(t) \cup \partial\mathcal{D}(t)$, where $\partial \mathcal{D}(t)$ denotes the outer boundary of the exterior domain $\Xi(t)$. To improve computational efficiency, the governing equations for the extracellular region will only be computed over $\Xi(t)$, which is centred on the centroid of $\Omega(t)$.
    
    \begin{figure}[t]
      \centering
    	  \inc{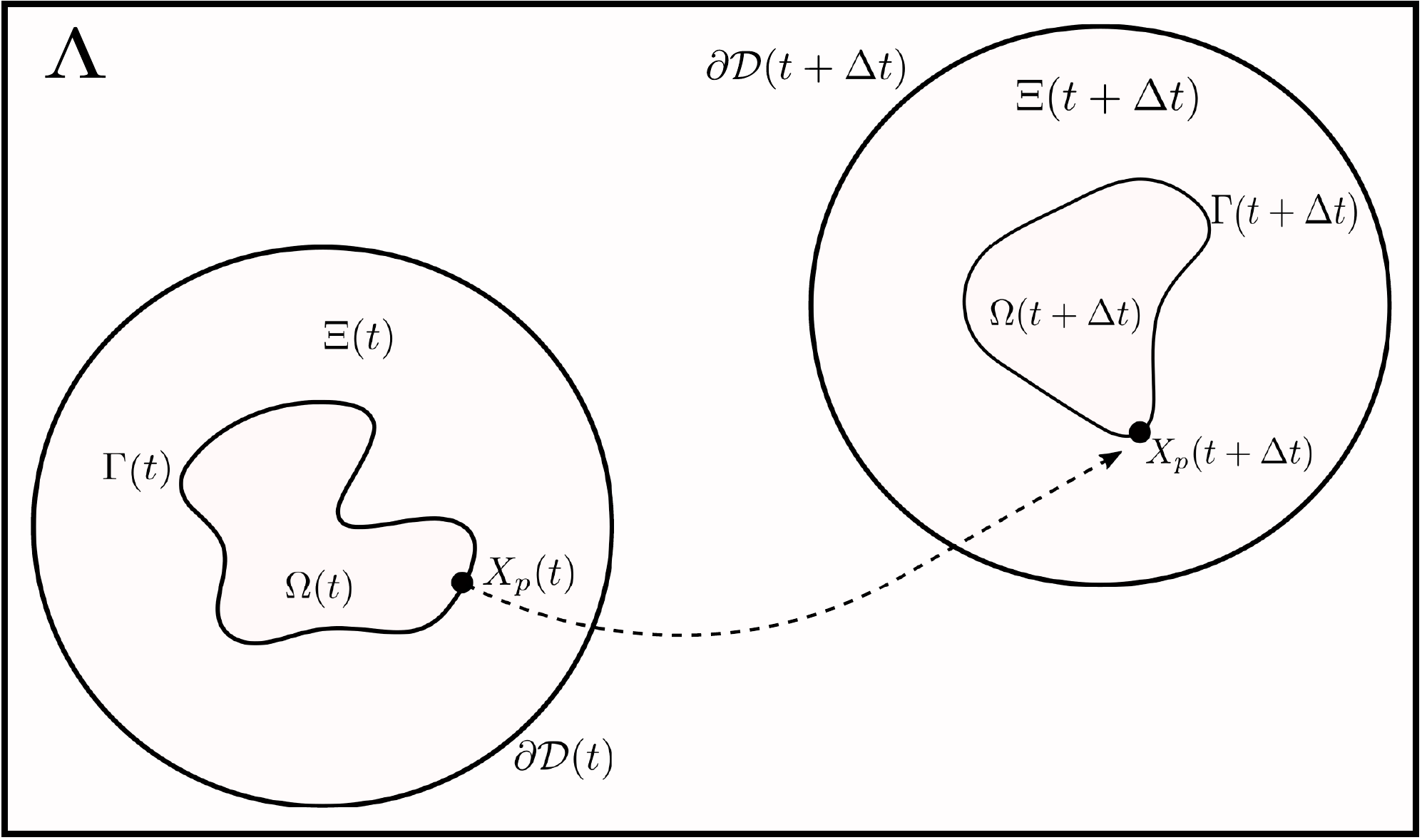}{0.7\textwidth}
    	  \caption{We consider the simulation of a motile cell through a fixed lab frame of reference $\Lambda$. The exterior region is denoted by $\Xi(t)$ and the interior region is denoted by $\Omega(t)$. The exterior and interior regions are separated by the curve $\Gamma(t)$. After a time interval of size $\Delta t$, the material point located at ${\bfa X}_{p}(t)$ on $\Gamma(t)$ evolves to the new location ${\bfa X}_{p}(t+\Delta t)$ on $\Gamma(t + \Delta t)$.}
	  \label{fig:move_domains}
	\end{figure}   

	We will assume that a material particle $P$ located at ${\bfa X}_{P}(t) \in \Gamma(t)$ has velocity $\dot{\bfa X}_{P}(t) \in \mathbb{R}^{2}$. Therefore, we assume that there exists a velocity field ${\bfa u}_{\Gamma}: \Gamma(t)\times \overline{I} \to \mathbb{R}^{2}$ so that material points on $\Gamma(t)$ evolve with a velocity field 
	\[
	  \dot{\bfa X}_{P}(t) = {\bfa u}_{\Gamma}({\bfa X}_{P}(t),t).
	\]
	Similarly, material points on the outer boundary $\partial\mathcal{D}(t)$, in the interior region and exterior region will be assumed to have material velocities ${\bfa u}_{\partial\mathcal{D}}: \partial\mathcal{D}(t)\times \overline{I} \to \mathbb{R}^{2}$, ${\bfa u}_{\Omega}: \overline{\Omega(t)}\times \overline{I} \to \mathbb{R}^{2}$ and ${\bfa u}_{\Xi}: \overline{\Xi(t)}\times \overline{I} \to \mathbb{R}^{2}$, respectively. 
	
	Let ${\bfa n}=(n_{1},n_{2})$ denote the unit outward normal to $\Gamma(t)$ and let ${\cal N}(t)$ be any open subset of $\RR^{2}$ containing $\Gamma(t)$. For any function $\zeta$ which is differentiable in ${\cal N}(t)$, we define the tangential gradient on $\Gamma(t)$ by $\nabla_{\Gamma}\zeta=\nabla \zeta -(\nabla \zeta\cdot {\bfa n}){\bfa n}$, where $\cdot$ denotes the usual scalar product and $\nabla \zeta$ denotes the usual gradient on $\RR^{2}$. For a vector function ${\bfa \zeta}=(\zeta_{1},\zeta_{2})\in \RR^{2}$, the tangential divergence is defined by 
	\[
	  \nabla_{\Gamma}\cdot {\bfa \zeta}=\nabla \cdot {\bfa \zeta}-\sum_{i=1}^{2}(\nabla\zeta_{i}\cdot {\bfa n})n_{i}.
	\]
	The Laplace-Beltrami operator on $\Gamma(t)$ is defined as the tangential divergence of the tangential gradient $\Delta_{\Gamma}\zeta=\nabla_{\Gamma}\cdot (\nabla_{\Gamma}\zeta)$. Furthermore, let us define 
    \begin{subequations}
    	  \begin{align}
      \label{eq:rd-qdefn-xi}
        \Qxi &= \left\{ ({\bfa x},t)\in \mathbb{R}^2 \colon {\bfa x}\in \Xi(t), t\in I \right\}, \\
      \label{eq:rd-qdefn-omega}
        \Qom &= \left\{ ({\bfa x,t})\in \mathbb{R}^2 \colon {\bfa x}\in \Omega(t), t\in I \right\}, \\
      \label{eq:rd-qdefn-gamma}
        \Qgm &= \left\{ ({\bfa x},t)\in \mathbb{R}^2 \colon {\bfa x}\in \Gamma(t), t\in I \right\}, \\
      \label{eq:rd-qdefn-oD}
        \Qod &= \left\{ \left( {\bfa x}, t \right) \in \mathbb{R}^2 \colon {\bfa x}\in \partial \mathcal{D}(t), t \in I \right\}.
      \end{align}
      \label{eq:rd-qdefn}
    \end{subequations}
    The closure of the sets defined in \eqref{eq:rd-qdefn} will be denoted $\overline{\Qxi}$, $\overline{\Qom}$, $\overline{\Qgm}$ and $\overline{\Qod}$, respectively, where it is understood that the closure occurs in both space and time.
    
    In the exterior domain $\Xi(t)$, we consider $N_l \in \mathbb{N}$ freely diffusing chemical species whose concentrations are given by $l^{(i)}: \overline{\Qxi} \to \mathbb{R}$, $i = 1,\ldots,N_l$. The vector of concentrations of the $N_l$ exterior bulk species is given by 
    \[
    	  {\bfa l}: \overline{\Qxi} \to \mathbb{R}^{N_l}, \qquad {\bfa l} = \left( l^{(1)}, \ldots, l^{(N_l)} \right)^{T}.
    	\]
    	Excluding cross-diffusion and coupling the reaction-diffusion systems by the reaction-kinetics only, these concentrations evolve according to the general advection-reaction-diffusion equation
    \begin{subequations}
    \begin{align}
    \label{eq:rd-ext-bulk-eq}
    & \frac{\partial l^{(i)}}{\partial t} + \nabla  \cdot \left( l^{(i)} {\bfa u}_{\Xi} \right) = D_l^{(i)} \Delta l^{(i)} + f_l^{(i)}({\bfa l}), \qquad ({\bfa x},t)\in \Qxi, \qquad i = 1,\ldots,N_l, \\
    \label{eq:rd-ext-bulk-obcs}
    & l^{(i)} = l_{\mathcal{D}}^{(i)} \quad \mbox{ or } \quad D_l^{(i)} \nabla l^{(i)} \cdot {\bfa n}_{\mathcal{D}} + l^{(i)}\left[ \left( {\bfa u}_{\partial\mathcal{D}} - {\bfa u}_{\Xi} \right) \cdot {\bfa n}_{\mathcal{D}} \right] = 0, \qquad ({\bfa x},t) \in \Qod, \\
    \label{eq:rd-ext-bulk-ibcs}
    & D_l^{(i)} \nabla l^{(i)} \cdot {\bfa n}_{\Xi} + l^{(i)} \left[ \left( {\bfa u}_{\Gamma} - {\bfa u}_{\Xi} \right) \cdot {\bfa n}_{\Xi} \right] = g^{(i)}({\bfa l}) + \hat{g}^{(i)}({\bfa l}, {\bfa l}_s), \qquad\quad\! ({\bfa x},t)\in \Qgm,
    \end{align}
    \label{eq:rd-ext-bulk}
  \end{subequations}
  where $D_{l}^{(i)}$ is the diffusion coefficient of the $i$-th exterior chemical species; the material velocities of $\Xi(t)$, $\Gamma(t)$ and $\partial\mathcal{D}(t)$ are denoted by ${\bfa u}_{\Xi}$, ${\bfa u}_{\Gamma}$ and ${\bfa u}_{\partial\mathcal{D}}$, respectively; the function $f_l^{(i)}: \Qxi\times \mathbb{R}^{N_l} \to \mathbb{R}$ describes the reaction kinetics between the interacting exterior bulk species ${\bfa l}$; the function $g^{(i)}: \Qgm \times \mathbb{R}^{N_l} \to \mathbb{R}$ describes the boundary conditions for the interacting exterior bulk species; and the function $\hat{g}^{(i)}: \Qgm \times \mathbb{R}^{N_l}\times \mathbb{R}^{N_{l_s}} \to \mathbb{R}$ describes the interaction between the exterior bulk $({\bfa l})$ and surface $({\bfa l}_{s})$ species. The normals ${\bfa n}_{\mathcal{D}}$ and ${\bfa n}_{\Xi}$ are the outward unit normals to $\Xi(t)$ from $\partial\mathcal{D}(t)$ and $\Gamma(t)$, respectively.
  
  On the surface $\Gamma(t)$, we consider $N_{l_s} \in \mathbb{N}$ chemical species that are free to diffuse tangentially along the surface $\Gamma(t)$ and whose concentrations are given by $l_s^{(j)}: \overline{\Qgm} \to \mathbb{R}$, $j = 1,\ldots,N_{l_s}$. The vector of concentrations of the surface species is given by
  \[
    {\bfa l}_s: \overline{\Qgm} \to \mathbb{R}^{N_{l_s}}, \qquad {\bfa l}_s = \left( l_s^{(1)}, \ldots, l_s^{(N_{l_s})} \right)^T.
  \]
  These concentrations evolve according to the general surface advection-reaction-diffusion equation
  \begin{align}
  \label{eq:rd-ext-surf}
  & \frac{\partial l_s^{(j)}}{\partial t} + \nabla_{\Gamma} \cdot \left( l_s^{(j)} {\bfa u}_{\Gamma} \right) = D_{l_s}^{(j)} \Delta_{\Gamma}l_s^{(j)} + f_{l_s}^{(j)}({\bfa l}_s) - \hat{g}_s^{(j)}({\bfa l},{\bfa l}_s), & ({\bfa x},t)\in \Qgm, \quad j = 1,\ldots,N_{l_s},
  \end{align}
  where $D_{l_s}^{(j)}$ is the diffusion coefficient of the $j$-th surface chemical species; the function $f_{l_s}^{(j)}: \Qgm \times \mathbb{R}^{N_{l_s}} \to \mathbb{R}$ describes the reaction kinetics between the interacting surface species; and the function $\hat{g}_s^{(j)}: \Qgm \times \mathbb{R}^{N_l} \times \mathbb{R}^{N_{l_s}} \to \mathbb{R}$ describes the interaction between the exterior bulk $({\bfa l})$ and surface $({\bfa l}_s)$ species.
  
  Similarly, in the interior domain we consider $N_c \in \mathbb{N}$ freely diffusing chemical species whose concentrations are given by $c^{(p)}: \overline{\Qom} \to \mathbb{R}$, $p = 1,\ldots,N_c$. The vector of concentrations of the $N_c$ interior bulk species is given by 
  \[
    {\bfa c}: \overline{\Qom} \to \mathbb{R}^{N_c}, \qquad {\bfa c} = \left( c^{(1)}, \ldots, c^{(N_c)} \right)^{T}. 
  \]
  These concentrations evolve according to the general advection-reaction-diffusion equation
    \begin{subequations}
    \begin{align}
    \label{eq:rd-int-bulk-eq}
    & \frac{\partial c^{(p)}}{\partial t} + \nabla  \cdot \left( c^{(p)} {\bfa u}_{\Omega} \right) = D_c^{(p)} \Delta c^{(p)} + f_c^{(p)}({\bfa c}), \qquad ({\bfa x},t)\in \Qom, \qquad p = 1,\ldots,N_c, \\
    \label{eq:rd-int-bulk-bcs}
    & D_c^{(p)} \nabla c^{(p)} \cdot {\bfa n}_{\Omega} + c^{(p)} \left[ \left( {\bfa u}_{\Gamma} - {\bfa u}_{\Omega} \right) \cdot {\bfa n}_{\Omega} \right] = r^{(p)}({\bfa c}) + \hat{r}^{(p)}({\bfa c}, {\bfa c}_s, {\bfa l}_s), \qquad ({\bfa x},t)\in \Qgm,
    \end{align}
    \label{eq:rd-int-bulk}
  \end{subequations}
  where $D_{c}^{(p)}$ is the diffusion coefficient of the $p$-th interior chemical species; the material velocity of $\Omega(t)$ is denoted by ${\bfa u}_{\Omega}$; the function $f_c^{(p)}: \Qom \times \mathbb{R}^{N_c} \to \mathbb{R}$ describes the reaction kinetics between the interacting interior bulk species ${\bfa c}$; the function $r^{(p)}: \Qgm \times \mathbb{R}^{N_c} \to \mathbb{R}$ describes the boundary conditions for the interacting interior bulk species; and the function $\hat{r}^{(p)}: \Qgm \times \mathbb{R}^{N_c}\times \mathbb{R}^{N_{c_s}} \times \mathbb{R}^{N_{l_s}} \to \mathbb{R}$ describes the interaction between the interior bulk $({\bfa c})$ and surface $({\bfa c}_{s})$ species under stimulation from the exterior surface species $({\bfa l}_s)$. The outward unit normal to $\Omega(t)$ from $\Gamma(t)$ is denoted by ${\bfa n}_{\Omega}$.
  
  Finally, on the surface $\Gamma(t)$ we consider $N_{c_s} \in \mathbb{N}$ chemical species that are free to diffuse tangentially along the surface $\Gamma(t)$ and whose concentrations are given by $c_s^{(q)}: \overline{\Qgm} \to \mathbb{R}$, $q = 1,\ldots,N_{c_s}$. The vector of concentrations of the surface species is given by 
  \[
    {\bfa c}_s: \overline{\Qgm} \to \mathbb{R}^{N_{c_s}}, \qquad {\bfa c}_s = \left( c_s^{(1)}, \ldots, c_s^{(N_{c_s})} \right)^T.
  \]
  These concentrations evolve according to the general surface advection-reaction-diffusion equation
  \begin{align}
  \label{eq:rd-int-surf}
  & \frac{\partial c_s^{(q)}}{\partial t} + \nabla_{\Gamma} \cdot \left( c_s^{(q)} {\bfa u}_{\Gamma} \right) = D_{c_s}^{(q)}\Delta_{\Gamma} c_s^{(q)} + f_{c_s}^{(q)}({\bfa c}_s) - \hat{r}_s^{(q)}({\bfa c},{\bfa c}_s, {\bfa l}_s), \qquad ({\bfa x},t)\in \Qgm,
  \end{align}
  where $D_{c_s}^{(q)}$ is the diffusion coefficient of the $q$-th surface species; the function $f_{c_s}^{(q)}: \Qgm \times \mathbb{R}^{N_{c_s}} \to \mathbb{R}$ describes the reaction kinetics between the interacting surface species; and the function $\hat{r}_s^{(q)}: \Qgm \times \mathbb{R}^{N_c} \times \mathbb{R}^{N_{c_s}} \times \mathbb{R}^{N_{l_s}} \to \mathbb{R}$ describes the interaction between the interior bulk $({\bfa c})$ and surface $({\bfa c}_{s})$ species under stimulation from the exterior surface species $({\bfa l}_s)$.

  \section{Weak ALE formulation}
  \label{sec:weakale}
    When the domain $\mathcal{D}(t)$ is moving it is conventional to adopt a common frame of reference for computational purposes. A popular choice is the Arbitrary Lagrangian Eulerian (ALE) frame. Let ${\cal A}(\cdot,t): \mathcal{D}_{c} \to \mathcal{D}(t)$ be a family of bijective mappings, which at each $t \in \overline{I}$ map points in a reference or computational configuration $\mathcal{D}_{c} \subset \RR^{2}$ with coordinates ${\bfa \xi} = (\xi,\eta)$ to points in the current physical configuration $\mathcal{D}(t) \subset \RR^{2}$ with coordinates ${\bfa x} = (x,y)$ so that ${\bfa x}({\bfa \xi}, t) = {\cal A}({\bfa \xi}, t)$. Similarly, we can write $\Xi(t) = {\cal A}(\Xi_c,t)$, $\partial{\cal D}(t) = {\cal A}(\partial{\cal D}_c,t)$, $\Omega(t) = {\cal A}(\Omega_c,t)$ and $\Gamma(t) = {\cal A}(\Gamma_{c},t)$. Here we assume that the computational configuration is the initial configuration $\mathcal{D}_{c} = \mathcal{D}(0)$. 
    
    Let $\psi: \overline{{\cal D}(t)} \times \overline{I} \to \RR$ be an arbitrary function defined on the fixed Eulerian frame and $\phi: \overline{\mathcal{D}_{c}} \times \overline{I} \to \RR$ be the corresponding function defined on the ALE frame, such that $\phi({\bfa \xi}, t) = \psi({\cal A}({\bfa \xi}, t), t)$. Then the temporal derivative of $\psi$ with respect to the ALE frame is defined as
    \begin{equation}
      \left. \ptl{\psi}{t} \right|_{{\bfa \xi}} : \overline{\mathcal{D}(t)} \times \overline{I} \to \RR, \qquad \left. \ptl{\psi}{t} \right|_{{\bfa \xi}}({\bfa x}, t) = \ptl{\phi}{t}({\bfa \xi},t), \nonumber
    \end{equation}
    where ${\bfa \xi} = {\cal A}^{-1}({\bfa x},t)$. A standard application of the chain rule yields
    \begin{equation}
      \left. \ptl{\psi}{t} \right|_{{\bfa \xi}}({\bfa x}, t) = \ptl{\psi}{t} + {\bfa w} \cdot \nabla \psi, \nonumber
    \end{equation}
    where the time derivative of the ALE mapping defines the ALE velocity ${\bfa w}$ as
    \[
      {\bfa w} = \left. \ptl{{\bfa x}}{t} \right|_{{\bfa \xi}}({\cal A}^{-1}({\bfa x}, t), t).
    \]
    Note that in general, the ALE velocity will differ from the material velocities. The reformulation of \eqref{eq:rd-ext-bulk-eq} and \eqref{eq:rd-ext-surf} in terms of the ALE reference frame takes the form
    \begin{equation}
    \label{eq:rd-ext-bulk-eq-weakale}
    \left. \frac{\partial l^{(i)}}{\partial t} \right|_{{\bfa \xi}} + \nabla  \cdot \left( l^{(i)} {\bfa u}_{\Xi} \right) - {\bfa w}\cdot \nabla l^{(i)} = D_l^{(i)} \Delta l^{(i)} + f_l^{(i)}({\bfa l}), \qquad ({\bfa \xi}, t) \in \Xi_c \times I,
    \end{equation}
    for each $i = 1,\ldots,N_l$, and
    \begin{equation}
    \label{eq:rd-ext-surf-weakale}
    \left. \frac{\partial l_s^{(j)}}{\partial t} \right|_{{\bfa \xi}} + \nabla_{\Gamma} \cdot \left( l_s^{(j)} {\bfa u}_{\Gamma} \right) - {\bfa w} \cdot \nabla_{\Gamma} l_s^{(j)} = D_{l_s}^{(j)} \Delta_{\Gamma}l_s^{(j)} + f_{l_s}^{(j)}({\bfa l}_s) - \hat{g}_s^{(j)}({\bfa l},{\bfa l}_s), \qquad ({\bfa \xi}, t) \in \Gamma_c \times I,
    \end{equation}
    for each $j = 1,\ldots,N_{l_s}$. Similarly, the reformulation of \eqref{eq:rd-int-bulk-eq} and \eqref{eq:rd-int-surf} in terms of the ALE reference frame takes the form
    \begin{equation}
    \label{eq:rd-int-bulk-eq-weakale}
      \left. \frac{\partial c^{(p)}}{\partial t} \right|_{{\bfa \xi}} + \nabla  \cdot \left( c^{(p)} {\bfa u}_{\Omega} \right) - {\bfa w}\cdot\nabla c^{(p)} = D_c^{(p)} \Delta c^{(p)} + f_c^{(p)}({\bfa c}), \qquad (\bfa{\xi}, t) \in \Omega_{c} \times I,
    \end{equation}
    for each $p = 1,\ldots,N_c$, and
    \begin{equation}
    \label{eq:rd-int-surf-weakale}
      \left. \frac{\partial c_s^{(q)}}{\partial t} \right|_{{\bfa \xi}} + \nabla_{\Gamma} \cdot \left( c_s^{(q)} {\bfa u}_{\Gamma} \right) - {\bfa w}\cdot\nabla_{\Gamma} c_s^{(q)} = D_{c_s}^{(q)}\Delta_{\Gamma} c_s^{(q)} + f_{c_s}^{(q)}({\bfa c}_s) - \hat{r}_s^{(q)}({\bfa c},{\bfa c}_s, {\bfa l}_s), \qquad ({\bfa \xi}, t) \in \Gamma_{c} \times I,
    \end{equation}
    for each $q = 1,\ldots,N_{c_s}$. For full details of the ALE formulation, the reader is referred to MacDonald \etal \cite{macdonald_2016}.
    
    To construct a weak formulation of \eqref{eq:rd-ext-bulk-eq-weakale}, we assume Dirichlet boundary conditions on the outer boundary $\partial{\cal D}(t)$ and consider the space of admissible test functions defined on the reference domain made of functions $\hat{v}: \Xi_{c} \to \RR$, $\hat{v} \in H_0^1(\Xi_{c})$, where
    \[
      H_0^1(\Xi_c) = \left\{ \hat{v} \in H^1(\Xi_c) : \hat{v} = 0, {\bfa \xi} \in \partial{\cal D}_c \right\}.
    \]
    Hence, the ALE mapping defines a set ${\cal H}(\Xi(t))$ of test functions $v: \Xi(t) \to \RR$ as follows:
    \[
      {\cal H}(\Xi(t)) = \left\{ v : v = \hat{v}\circ {\cal A}^{-1}(\cdot,t), \hat{v}\in H_0^{1}(\Xi_{c})) \right\}, \qquad t \in \overline{I}.
    \]
    
    For each $i = 1,\ldots,N_l$, multiply \eqref{eq:rd-ext-bulk-eq-weakale} by a test function $v \in {\cal H}(\Xi(t))$ and integrate over $\Xi(t)$ to give the weak form: find $l^{(i)} \in {\cal H}(\Xi(t))$ such that
    \begin{align}
      \otl{}{t} \int_{\Xi(t)} l^{(i)} v \;{\rm d}{\bfa x} 
	= & \int_{\Xi(t)} \left\{ \nabla \cdot \left[ l^{(i)}({\bfa w} - {\bfa u}_{\Xi}) \right] \right\} v \;{\rm d}{\bfa x} + \int_{\Xi(t)} \left[ \nabla\cdot (D_l^{(i)}\nabla l^{(i)}) \right] v \;{\rm d}{\bfa x} \nonumber \\
	& + \int_{\Xi(t)} f_l^{(i)}({\bfa l})v \;{\rm d}{\bfa x} \nonumber \\
	= & -\int_{\Xi(t)} l^{(i)} \left[ ({\bfa w} - {\bfa u}_{\Xi}) \cdot \nabla v \right] \;{\rm d}{\bfa x} + \int_{\Gamma(t)} l^{(i)} \left[ ({\bfa w} - {\bfa u}_{\Gamma}) \cdot {\bfa n}_{\Xi} \right] v \;{\rm d}s \nonumber \\ 
	& - \int_{\Xi(t)} D_l^{(i)} \nabla l^{(i)} \cdot \nabla v \;{\rm d}{\bfa x} + \int_{\Gamma(t)} \left[ g^{(i)}({\bfa l}) + \hat{g}^{(i)}({\bfa l}, {\bfa l}_s) \right] v \;{\rm d}s \nonumber \\
	& + \int_{\Xi(t)} f_l^{(i)}({\bfa l})v \;{\rm d}{\bfa x}, \hspace{5cm} \forall v\in{\cal H}(\Xi(t)),
    \label{eq:rd-ext-bulk-weakeq-weakale}
    \end{align}
    where we have applied the flux boundary conditions \eqref{eq:rd-ext-bulk-ibcs}. Similarly, to construct the weak formulation of \eqref{eq:rd-int-bulk-eq-weakale} we consider the space of admissible test functions defined on the reference domain made of functions $\hat{v}: \Omega_c \to \RR$, $\hat{v} \in H^{1}(\Omega_c)$. Hence, the ALE mapping defines a set ${\cal H}(\Omega(t))$ of test functions $v: \Omega(t) \to \RR$ as follows:
    \[
      {\cal H}(\Omega(t)) = \left\{ v : v = \hat{v}\circ {\cal A}^{-1}(\cdot,t), \hat{v}\in H^{1}(\Omega_{c})) \right\}, \qquad t \in \overline{I}.
    \]
    
    For each $p = 1,\ldots,N_c$, multiply \eqref{eq:rd-int-bulk-eq-weakale} by a test function $v \in {\cal H}(\Omega(t))$ and integrate over $\Omega(t)$ to give the weak form: find $c^{(p)} \in {\cal H}(\Omega(t))$ such that
    \begin{align}
    	  \otl{}{t} \int_{\Omega(t)} c^{(p)} v \;{\rm d}{\bfa x} 
	  = & \int_{\Omega(t)} \left\{ \nabla \cdot \left[ c^{(p)}({\bfa w} - {\bfa u}_{\Omega}) \right] \right\} v \;{\rm d}{\bfa x} + \int_{\Omega(t)} \left[ \nabla\cdot (D_c^{(p)}\nabla c^{(p)}) \right] v \;{\rm d}{\bfa x} \nonumber \\
	  & + \int_{\Omega(t)} f_c^{(p)}({\bfa c})v \;{\rm d}{\bfa x} \nonumber \\
	  = & -\int_{\Omega(t)} c^{(p)} \left[ ({\bfa w} - {\bfa u}_{\Omega}) \cdot \nabla v \right] \;{\rm d}{\bfa x} + \int_{\Gamma(t)} c^{(p)} \left[ ({\bfa w} - {\bfa u}_{\Gamma}) \cdot {\bfa n}_{\Omega} \right] v \;{\rm d}s \nonumber \\ 
	  & - \int_{\Omega(t)} D_c^{(p)} \nabla c^{(p)} \cdot \nabla v \;{\rm d}{\bfa x} + \int_{\Gamma(t)} \left[ r^{(p)}({\bfa c}) + \hat{r}^{(p)}({\bfa c}, {\bfa c}_s, {\bfa l}_s) \right] v \;{\rm d}s \nonumber \\
	  & + \int_{\Omega(t)} f_c^{(p)}({\bfa c})v \;{\rm d}{\bfa x}, \hspace{5cm} \forall v\in{\cal H}(\Omega(t)),
	\label{eq:rd-int-bulk-weakeq-weakale}
    \end{align}
    where we have applied the flux boundary conditions \eqref{eq:rd-int-bulk-bcs}. Finally, to construct the weak formulation of \eqref{eq:rd-ext-surf-weakale} and \eqref{eq:rd-int-surf-weakale}, consider the space of admissible test functions defined on the reference curve made of functions $\hat{v}_{s}: \Gamma_{c} \to \RR$, $\hat{v}_{s}\in H^{1}(\Gamma_{c})$. Once again, the ALE mapping defines a set ${\cal H}_{s}(\Gamma(t))$ of test functions $v_{s}: \Gamma(t) \to \RR$, as follows:
    \[
      {\cal H}_{s}(\Gamma(t)) = \left\{ v_{s} : v_{s} = \hat{v}_{s}\circ {\cal A}^{-1}(\cdot,t), \hat{v}_{s}\in H^{1}(\Gamma_{c})) \right\}, \qquad t \in \overline{I}.
    \]
    
    For each $j = 1,\ldots,N_{l_s}$ and $q = 1,\ldots,N_{c_s}$, multiply \eqref{eq:rd-ext-surf-weakale} and \eqref{eq:rd-int-surf-weakale} by a test function $v_{s} \in {\cal H}_{s}(\Gamma(t))$ and integrate over $\Gamma(t)$ to give the weak form: find $l_{s}^{(j)} \in {\cal H}_{s}(\Gamma(t))$ and $c_s^{(q)} \in {\cal H}_s(\Gamma(t))$ such that
    \begin{align}
      \otl{}{t} \int_{\Gamma(t)} l_{s}^{(j)}v_{s} \;{\rm d}s = & \int_{\Gamma(t)} \left\{\nabla_{\Gamma}\cdot\left[l_{s}^{(j)}({\bfa w} - {\bfa u}_{\Gamma})\right]\right\}v_{s} \;{\rm d}s - \int_{\Gamma(t)} D_{l_s}^{(j)} \nabla_{\Gamma}l_{s}^{(j)}\cdot\nabla_{\Gamma}v_{s} \;{\rm d}s \nonumber \\
      & + \int_{\Gamma(t)} f_{l_s}^{(j)}({\bfa l}_{s}) v_{s} \;{\rm d}s - \int_{\Gamma(t)} \hat{g}_{s}^{(j)}({\bfa l}, {\bfa l_{s}})v_{s} \;{\rm d}s, \qquad \forall v_{s}\in{\cal H}_{s}(\Gamma(t)),
    \label{eq:rd-ext-surf-weakeq-weakale}
    \end{align}
    and
    \begin{align}
      \otl{}{t} \int_{\Gamma(t)} c_{s}^{(q)}v_{s} \;{\rm d}s = & \int_{\Gamma(t)} \left\{\nabla_{\Gamma}\cdot\left[c_{s}^{(q)}({\bfa w} - {\bfa u}_{\Gamma})\right]\right\}v_{s} \;{\rm d}s - \int_{\Gamma(t)} D_{c_s}^{(q)} \nabla_{\Gamma}c_{s}^{(q)}\cdot\nabla_{\Gamma}v_{s} \;{\rm d}s \nonumber \\
      & + \int_{\Gamma(t)} f_{c_s}^{(q)}({\bfa c}_{s}) v_{s} \;{\rm d}s - \int_{\Gamma(t)} \hat{r}_s^{(q)}({\bfa c}, {\bfa c}_s, {\bfa l_{s}})v_{s} \;{\rm d}s, \qquad \forall v_{s}\in{\cal H}_{s}(\Gamma(t)),
    \label{eq:rd-int-surf-weakeq-weakale}
    \end{align}
    where we have assumed tangential surface fluxes only.
    
    \subsection{Global conservation of the continuous formulation on evolving domains}
    \label{sec:conserv-cont}
      When Dirichlet boundary conditions are used on the exterior outer boundary, global conservation cannot be expected. Global conservation can be obtained if zero flux boundary conditions are employed in \eqref{eq:rd-ext-bulk-obcs}. However, in order to maintain a homogeneous distribution of exterior chemical species under movement, one would require a non-zero exterior material velocity ${\bfa u}_{\Xi}$, which may be unphysical depending on the problem under consideration. Finally, one could employ a non-zero flux boundary condition on the exterior outer boundary. However, such a boundary condition will in general be unknown {\textit {a-priori}} and conservation would only be achieved if
      \[
        \int_{\partial{\cal D}(t)} D_l^{(i)}\nabla l^{(i)}\cdot {\bfa n}_{{\cal D}} + l^{(i)}\left[ \left({\bfa u}_{\partial{\cal D}} - {\bfa u}_{\Xi}\right)\cdot {\bfa n}_{{\cal D}} \right] \;{\rm d}s = 0, \qquad i = 1,\ldots,N_l.
      \]
      Therefore, in the results presented in \S\ref{sec:nexp-cell} we will employ Dirichlet boundary conditions on the exterior outer boundary and, in the studies of global conservation throughout this article, focus our attention on the interior bulk-surface environment only; that is, global conservation of ${\bfa c}$ and ${\bfa c}_s$. We require the following assumption:
      
      \begin{assump}
      \label{ass:cont}
		We assume that on the boundary $\Gamma(t)$, the normal components of the ALE velocity and material velocity are identical; that is, ${\bfa w}\cdot{\bfa n} = {\bfa u}_{\Gamma}\cdot{\bfa n}$.
      \end{assump}
      
      The shape of an evolving curve is determined purely by its normal velocity and therefore, Assumption \ref{ass:cont} is necessary to ensure that the evolution of the curve with respect to the ALE frame matches the evolution of the curve prescribed by the material velocity ${\bfa u}_{\Gamma}$. We are now ready to prove global conservation of the reaction-diffusion system given by \eqref{eq:rd-int-bulk} and \eqref{eq:rd-int-surf}.
      
      \begin{theorem}
      \label{thm:conserv-cont}
		Under the conditions of Assumption \ref{ass:cont} and provided
		\begin{equation}
		\label{eq:conserv-cont-sources}
		\sum_{p = 1}^{N_c} \int_{\Omega(t)} f_c^{(p)}({\bfa c}) \;{\rm d}{\bfa x} = 0, \quad \sum_{p = 1}^{N_c} \int_{\Gamma(t)} r^{(p)}({\bfa c}) \;{\rm d}s = 0, \quad \sum_{q = 1}^{N_{c_s}} \int_{\Gamma(t)} f_{c_s}^{(q)}({\bfa c}_{s}) \;{\rm d}s = 0,
		\end{equation}
		and
		\begin{equation}
		\label{eq:conserv-cont-interact}
		\sum_{p = 1}^{N_c} \int_{\Gamma(t)} \hat{r}^{(p)}({\bfa c}, {\bfa c}_s, {\bfa l}_s) \;{\rm d}s = \sum_{q = 1}^{N_{c_s}} \int_{\Gamma(t)} \hat{r}_s^{(q)}({\bfa c}, {\bfa c}_s, {\bfa l_{s}}) \;{\rm d}s,
		\end{equation}
		are satisfied, then the reaction-diffusion system given by \eqref{eq:rd-int-bulk} and \eqref{eq:rd-int-surf} is globally conservative independent of the ALE velocity; that is, the following holds true:
		\[
	  	  \otl{}{t} \left\{ \sum_{p=1}^{N_c} \int_{\Omega(t)} c^{(p)} \;{\rm d}{\bfa x} + \sum_{q=1}^{N_{c_s}} \int_{\Gamma(t)} c_{s}^{(q)} \;{\rm d}s \right\} = 0.
		\]
      \end{theorem}
      \begin{proof}
		In \eqref{eq:rd-int-bulk-weakeq-weakale} we set $v = 1$, apply the divergence theorem and sum over all $p = 1,\ldots,N_c$ to give
		\begin{align}
    	  	  \sum_{p = 1}^{N_c} \left\{ \otl{}{t} \int_{\Omega(t)} c^{(p)} \;{\rm d}{\bfa x} \right\} 
	  	  = & \sum_{p = 1}^{N_c} \left\{ \int_{\Omega(t)} \nabla \cdot \left[ c^{(p)}({\bfa w} - {\bfa u}_{\Omega}) \right] \;{\rm d}{\bfa x} + \int_{\Omega(t)} \nabla\cdot (D_c^{(p)}\nabla c^{(p)}) \;{\rm d}{\bfa x} \right. \;+ \nonumber \\
	  	  & \qquad \; \left. \int_{\Omega(t)} f_c^{(p)}({\bfa c}) \;{\rm d}{\bfa x} \right\} \nonumber \\
	  	  = & \sum_{p = 1}^{N_c} \int_{\Gamma(t)} c^{(p)} \left[ ({\bfa w} - {\bfa u}_{\Gamma}) \cdot {\bfa n}_{\Omega} \right] \;{\rm d}s + \sum_{p = 1}^{N_c} \int_{\Gamma(t)} r^{(p)}({\bfa c}) \;{\rm d}s \;+ \nonumber \\
	  	  & \sum_{p = 1}^{N_c} \int_{\Gamma(t)} \hat{r}^{(p)}({\bfa c}, {\bfa c}_s, {\bfa l}_s) \;{\rm d}s + \sum_{p = 1}^{N_c} \int_{\Omega(t)} f_c^{(p)}({\bfa c}) \;{\rm d}{\bfa x}.
		\label{eq:conserv-cont-bulk-eq}
    		\end{align}
		Similarly, in \eqref{eq:rd-int-surf-weakeq-weakale} we set $v_{s} = 1$ and sum over all $q = 1,\ldots, N_{c_s}$ to give
		\begin{align}
      	  \sum_{q = 1}^{N_{c_s}} \left\{ \otl{}{t} \int_{\Gamma(t)} c_{s}^{(q)} \;{\rm d}s \right\} = & \sum_{q = 1}^{N_{c_s}} \int_{\Gamma(t)} \nabla_{\Gamma}\cdot\left[c_{s}^{(q)}({\bfa w} - {\bfa u}_{\Gamma})\right] \;{\rm d}s + \sum_{q = 1}^{N_{c_s}} \int_{\Gamma(t)} f_{c_s}^{(q)}({\bfa c}_{s}) \;{\rm d}s \;- \nonumber \\
      	  & \sum_{q = 1}^{N_{c_s}} \int_{\Gamma(t)} \hat{r}_s^{(q)}({\bfa c}, {\bfa c}_s, {\bfa l_{s}}) \;{\rm d}s.
    		\label{eq:conserv-cont-surf-eq}
    		\end{align}
		Adding \eqref{eq:conserv-cont-bulk-eq} and \eqref{eq:conserv-cont-surf-eq} gives
		\begin{align}
	  	  \otl{}{t} \left\{ \sum_{p = 1}^{N_c} \int_{\Omega(t)} c^{(p)} \;{\rm d}{\bfa x} \right. + \left. \sum_{q = 1}^{N_{c_s}} \int_{\Gamma(t)} c_{s}^{(q)} \;{\rm d}s \right\} = & \nonumber \\
	    & \hspace{-5cm} \sum_{p = 1}^{N_c} \int_{\Gamma(t)} c^{(p)} \left[ ({\bfa w} - {\bfa u}_{\Gamma}) \cdot {\bfa n}_{\Omega} \right] \;{\rm d}s + \sum_{p = 1}^{N_c} \int_{\Gamma(t)} r^{(p)}({\bfa c}) \;{\rm d}s + \sum_{p = 1}^{N_c} \int_{\Gamma(t)} \hat{r}^{(p)}({\bfa c}, {\bfa c}_s, {\bfa l}_s) \;{\rm d}s \;+ \nonumber \\
	    & \hspace{-5cm} \sum_{p = 1}^{N_c} \int_{\Omega(t)} f_c^{(p)}({\bfa c}) \;{\rm d}{\bfa x} + \sum_{q = 1}^{N_{c_s}} \int_{\Gamma(t)} \nabla_{\Gamma}\cdot\left[c_{s}^{(q)}({\bfa w} - {\bfa u}_{\Gamma})\right] \;{\rm d}s + \sum_{q = 1}^{N_{c_s}} \int_{\Gamma(t)} f_{c_s}^{(q)}({\bfa c}_{s}) \;{\rm d}s \;- \nonumber \\
	    & \hspace{-5cm} \sum_{q = 1}^{N_{c_s}} \int_{\Gamma(t)} \hat{r}_s^{(q)}({\bfa c}, {\bfa c}_s, {\bfa l_{s}}) \;{\rm d}s.
	    \label{eq:conserv-cont-eq}
		\end{align}
		Due to Assumption \ref{ass:cont}, the normal component of the ALE and material velocites are the same. Hence, the first term on the right-hand side of the above is zero. Also, due to the curve $\Gamma(t)$ being closed there is no boundary of $\Gamma(t)$ and therefore, there cannot be any tangential flux; in other words,
		\[
	  	  \int_{\Gamma(t)} \nabla_{\Gamma}\cdot\left[c_{s}^{(q)}({\bfa w} - {\bfa u}_{\Gamma})\right] \;{\rm d}s = 0, \qquad q = 1,\ldots,N_{c_s}.
		\]
		Therefore, \eqref{eq:conserv-cont-eq} reduces to
		\begin{align*}
	  	  \otl{}{t} \left\{ \sum_{p = 1}^{N_c} \int_{\Omega(t)} c^{(p)} \;{\rm d}{\bfa x} \right. + \left. \sum_{q = 1}^{N_{c_s}} \int_{\Gamma(t)} c_{s}^{(q)} \;{\rm d}s \right\} = & \nonumber \\
	    & \hspace{-5cm} \sum_{p = 1}^{N_c} \int_{\Omega(t)} f_c^{(p)}({\bfa c}) \;{\rm d}{\bfa x} + \sum_{p = 1}^{N_c} \int_{\Gamma(t)} r^{(p)}({\bfa c}) \;{\rm d}s + \sum_{q = 1}^{N_{c_s}} \int_{\Gamma(t)} f_{c_s}^{(q)}({\bfa c}_{s}) \;{\rm d}s \;+ \nonumber \\
	    & \hspace{-5cm}  \sum_{p = 1}^{N_c} \int_{\Gamma(t)} \hat{r}^{(p)}({\bfa c}, {\bfa c}_s, {\bfa l}_s) \;{\rm d}s - \sum_{q = 1}^{N_{c_s}} \int_{\Gamma(t)} \hat{r}_s^{(q)}({\bfa c}, {\bfa c}_s, {\bfa l_{s}}) \;{\rm d}s.
		\end{align*}
		The result is then immediate following the application of \eqref{eq:conserv-cont-sources} and \eqref{eq:conserv-cont-interact}.
	
      \end{proof}

  \section{Finite element spatial discretisation}
  \label{sec:fem}
    We will assume that the reference ${\cal D}_{c}$ and physical ${\cal D}(t)$ domains are approximated by polygonal domains ${\cal D}_{c,h}$ and ${\cal D}_{h}(t)$, $\forall t\in \overline{I}$, respectively. Let ${\cal D}_{c,h}$ be covered by a fixed triangulation ${\cal T}_{h,c}^{{\cal D}}$ with straight edges, so that
    \[
      {\cal D}_{c,h} = \bigcup_{K \in {\cal T}_{h,c}^{{\cal D}}} K.
    \]
    The curve $\Gamma(t)$ separates ${\cal D}(t)$ into exterior $\Xi(t)$ and interior $\Omega(t)$ regions. Therefore, we construct the triangulation ${\cal T}_{h,c}^{{\cal D}}$ so that it is fitted to the curve $\Gamma(t)$. To do this, we let $\Xi_{c,h}$ and $\Omega_{c,h}$ be covered by fixed triangulations ${\cal T}_{h,c}^{\Xi}$ and ${\cal T}_{h,c}^{\Omega}$, respectively, with straight edges so that
    \[
      \Xi_{c,h} = \bigcup_{K \in {\cal T}_{h,c}^{\Xi}} K, \qquad \mbox{and} \qquad \Omega_{c,h} = \bigcup_{K \in {\cal T}_{h,c}^{\Omega}} K.
    \]
    Clearly, ${\cal T}_{h,c}^{\Xi} \subseteq {\cal T}_{h,c}^{{\cal D}}$ and ${\cal T}_{h,c}^{\Omega} \subseteq {\cal T}_{h,c}^{{\cal D}}$. Let ${\cal E}$ denote the set of all edges of the triangulation ${\cal T}_{h,c}^{\Omega}$ and define
    \[
      {\cal E}^{\Gamma} = \left\{ e : e \in {\cal E} \cap \partial\Omega_{c,h} \right\}.
    \]
    Thus, the curve $\Gamma_c$ is approximated using the boundary of $\Omega_{c,h}$ so that $\Gamma_{c,h} := \partial\Omega_{c,h}$. Note that due to ${\cal T}_{h,c}^{{\cal D}}$ being fitted to $\Gamma_c$, we could also define $\Gamma_{c,h}$ using the inner boundary of $\Xi_{c,h}$. Similarly, the approximation to the outer boundary $\partial{\cal D}_c$ is obtained using the outer boundary of $\Xi_{c,h}$. The number of elements of ${\cal T}_{h,c}^{{\cal D}}$, ${\cal T}_{h,c}^{\Xi}$ and ${\cal T}_{h,c}^{\Omega}$ will be denoted by $N_e^{{\cal D}}$, $N_e^{\Xi}$ and $N_e^{\Omega}$, respectively; the total number of vertices of ${\cal T}_{h,c}^{{\cal D}}$, ${\cal T}_{h,c}^{\Xi}$ and ${\cal T}_{h,c}^{\Omega}$ will be denoted by ${\cal N}^{{\cal D}}$, ${\cal N}^{\Xi}$ and ${\cal N}^{\Omega}$, respectively; and the number of vertices on the boundaries will be denoted by ${\cal N}_s^{\Gamma}$ and ${\cal N}_s^{\partial{\cal D}}$.
    
    The finite element spaces on ${\cal T}_{h,c}^{{\cal D}}$ can then be defined as
    \begin{equation}
      {\cal L}^{1}({\cal D}_{c,h}) = \left\{ \hat{v}_{h} \in H^{1}({\cal D}_{c,h}) : \left. \hat{v}_{h} \right|_{K} \in {\bb P}_{1}(K), \forall K \in {\cal T}_{h,c}^{{\cal D}} \right\}, \nonumber
    \end{equation}
    where $\hat{v}_{h}: {\cal D}_{c,h} \to \RR$ and ${\bb P}_{1}(K)$ is the space of linear polynomials on $K$. Similarly, the finite element spaces on ${\cal T}_{h,c}^{\Xi}$ and ${\cal T}_{h,c}^{\Omega}$ can be defined as
    \begin{align}
      {\cal L}^{1}(\Xi_{c,h}) = & \left\{ \hat{v}_{h} \in H^{1}(\Xi_{c,h}) : \left. \hat{v}_{h} \right|_{K} \in {\bb P}_{1}(K), \forall K \in {\cal T}_{h,c}^{\Xi} \right\}, \nonumber \\
      {\cal L}^{1}_{0}(\Xi_{c,h}) = & \left\{ \hat{v}_{h} \in {\cal L}^{1}(\Xi_{c,h}) : \hat{v}_{h} = 0, {\bfa \xi} \in \partial{\cal D}_{c,h} \right\}, \nonumber
    \end{align}
    and
    \begin{align}
      {\cal L}^{1}(\Omega_{c,h}) = & \left\{ \hat{v}_{h} \in H^{1}(\Omega_{c,h}) : \left. \hat{v}_{h} \right|_{K} \in {\bb P}_{1}(K), \forall K \in {\cal T}_{h,c}^{\Omega} \right\}, \nonumber \\
      {\cal L}^{1}_{0}(\Omega_{c,h}) = & \left\{ \hat{v}_{h} \in {\cal L}^{1}(\Omega_{c,h}) : \hat{v}_{h} = 0, {\bfa \xi} \in \partial\Omega_{c,h} \right\}, \nonumber
    \end{align}
    respectively. Furthermore, we define
    \begin{equation}
      {\cal L}^{1}(\Gamma_{c,h}) = \left\{ \hat{v}_{h} \in H^{1}(\Gamma_{c,h}) : \left. \hat{v}_{h} \right|_{e} \in {\bb P}_{1}(e), \forall e \in {\cal E}^{\Gamma} \right\}. \nonumber
    \end{equation}
    The ALE mapping is interpolated using piecewise linear finite elements giving rise to a spatially discrete mapping ${\cal A}_{h}(\cdot,t) \in \left[ {\cal L}^{1}({\cal D}_{c,h}) \right]^{2}$, for each $t \in \overline{I}$, of the form
    \[
      {\bfa x}_{h}({\bfa \xi},t) = {\cal A}_{h}({\bfa \xi},t) = \sum_{\mu = 1}^{{\cal N}^{{\cal D}}}{\bfa x}_{\mu}(t)\hat{\phi}_{\mu}({\bfa \xi}),
    \]
    where ${\bfa x}_{\mu}(t) = {\cal A}_{h}({\bfa \xi}_{\mu},t)$ denotes the position of the $\mu$-th node at time $t$, and $\{ \hat{\phi}_{\mu}({\bfa \xi}) \}_{\mu = 1}^{{\cal N}^{{\cal D}}}$ are the nodal basis functions in ${\cal L}^{1}({\cal D}_{c,h})$. The discretised ALE velocity therefore takes the form
    \[
      {\bfa w}_{h}({\bfa \xi}, t) = \sum_{\mu = 1}^{{\cal N}^{{\cal D}}}\dot{{\bfa x}}_{\mu}(t)\hat{\phi}_{\mu}({\bfa \xi}).
    \]
    The physical triangulation is defined as the image of the reference triangulation under the discrete ALE mapping: for example, ${\cal T}_h^{{\cal D}}(t) = {\cal A}_h({\cal T}_{h,c}^{{\cal D}},t)$. Since the mapping is linear, each $K(t) \in {\cal T}_h^{{\cal D}}(t)$ which is the image of a triangle $K \in {\cal T}_{h,c}^{{\cal D}}$, is also a triangle with straight edges. Similarly for triangulations ${\cal T}_h^{\Xi}(t)$ and ${\cal T}_h^{\Omega}(t)$.
    
    The finite element test space on $\Xi_h(t)$ is defined, using the ALE mapping, as a set ${\cal H}_h(\Xi_h(t))$ of test functions $v_h: \Xi_{h}(t) \to \RR$ such that
    \[
      {\cal H}_h(\Xi_h(t)) = \left\{ v_h : v_h = \hat{v}_h\circ {\cal A}^{-1}_{h}(\cdot,t), \hat{v}_h\in {\cal L}_0^1(\Xi_{c,h}) \right\}, \quad t \in \overline{I}.
    \]
    The finite element spatial discretisation of the ALE formulation \eqref{eq:rd-ext-bulk-weakeq-weakale} therefore takes the form: find $l_h^{(i)}\in {\cal H}_h(\Xi_h(t))$, for each $i = 1,\ldots,N_l$, such that
    \begin{align}
      \otl{}{t} \int_{\Xi_h(t)} l_h^{(i)} v_h \;{\rm d}{\bfa x} 
	= & -\int_{\Xi_h(t)} l_h^{(i)} \left[ ({\bfa w}_h - {\bfa u}_{\Xi}) \cdot \nabla v_h \right] \;{\rm d}{\bfa x} + \int_{\Gamma_h(t)} l_h^{(i)} \left[ ({\bfa w}_h - {\bfa u}_{\Gamma}) \cdot {\bfa n}_{\Xi} \right] v_h \;{\rm d}s \nonumber \\ 
	& - \int_{\Xi_h(t)} D_l^{(i)} \nabla l_h^{(i)} \cdot \nabla v_h \;{\rm d}{\bfa x} + \int_{\Gamma_h(t)} \left[ g^{(i)}({\bfa l}_h) + \hat{g}^{(i)}({\bfa l}_h, {\bfa l}_{s,h}) \right] v_h \;{\rm d}s \nonumber \\
	& + \int_{\Xi_h(t)} f_l^{(i)}({\bfa l}_h)v_h \;{\rm d}{\bfa x}, \hspace{4cm} \forall v_h \in {\cal H}_h(\Xi_h(t)).
    \label{eq:fem-ext-bulk-eq}
    \end{align}
    Similarly, the ALE mapping defines the finite element test space on $\Omega_h(t)$ as the set ${\cal H}_h(\Omega_h(t))$ of test functions $v_h: \Omega_h(t) \to \RR$ such that
    \[
      {\cal H}_h(\Omega_h(t)) = \left\{ v_h : v_h = \hat{v}_h\circ {\cal A}^{-1}_{h}(\cdot,t), \hat{v}_h\in {\cal L}^1(\Omega_{c,h}) \right\}, \quad t \in \overline{I}.
    \]
    The finite element spatial discretisation of the ALE formulation \eqref{eq:rd-int-bulk-weakeq-weakale} therefore takes the form: find $c_h^{(p)} \in {\cal H}_h(\Omega_h(t))$, for each $p = 1,\ldots,N_c$, such that
    \begin{align}
    	  \otl{}{t} \int_{\Omega_h(t)} c_h^{(p)} v_h \;{\rm d}{\bfa x} 
	  = & -\int_{\Omega_h(t)} c_h^{(p)} \left[ ({\bfa w}_h - {\bfa u}_{\Omega}) \cdot \nabla v_h \right] \;{\rm d}{\bfa x} + \int_{\Gamma_h(t)} c_h^{(p)} \left[ ({\bfa w}_h - {\bfa u}_{\Gamma}) \cdot {\bfa n}_{\Omega} \right] v_h \;{\rm d}s \nonumber \\ 
	  & - \int_{\Omega_h(t)} D_c^{(p)} \nabla c_h^{(p)} \cdot \nabla v_h \;{\rm d}{\bfa x} + \int_{\Gamma_h(t)} \left[ r^{(p)}({\bfa c}_h) + \hat{r}^{(p)}({\bfa c}_h, {\bfa c}_{s,h}, {\bfa l}_{s,h}) \right] v_h \;{\rm d}s \nonumber \\
	  & + \int_{\Omega_h(t)} f_c^{(p)}({\bfa c}_h) v_h \;{\rm d}{\bfa x}, \hspace{4cm} \forall v_h\in{\cal H}_h(\Omega_h(t)).
	\label{eq:fem-int-bulk-eq}
    \end{align}
    Finally, the ALE mapping defines the finite element test space on $\Gamma_h(t)$ as the set ${\cal H}_{s,h}(\Gamma_h(t))$ of test functions $v_{s,h}: \Gamma_h(t) \to \RR$ such that
    \[
      {\cal H}_{s,h}(\Gamma_h(t)) = \left\{ v_{s,h} : v_{s,h} = \hat{v}_{s,h}\circ {\cal A}^{-1}_{h}(\cdot,t), \hat{v}_{s,h}\in {\cal L}^1(\Gamma_{c,h}) \right\}, \quad t \in \overline{I}.
    \]
    The finite element spatial discretisation of the ALE formulations \eqref{eq:rd-ext-surf-weakeq-weakale} and \eqref{eq:rd-int-surf-weakeq-weakale} take the form: find $l_{s,h}^{(j)} \in {\cal H}_{s,h}(\Gamma_h(t))$, for each $j = 1,\ldots,N_{l_s}$, and $c_{s,h}^{(q)} \in {\cal H}_{s,h}(\Gamma_h(t))$, for each $q = 1,\ldots,N_{c_s}$, such that
    \begin{align}
      \otl{}{t} \int_{\Gamma_h(t)} l_{s,h}^{(j)}v_{s,h} \;{\rm d}s = & \int_{\Gamma_h(t)} \left\{\nabla_{\Gamma}\cdot\left[l_{s,h}^{(j)}({\bfa w}_h - {\bfa u}_{\Gamma})\right]\right\}v_{s,h} \;{\rm d}s - \int_{\Gamma_h(t)} D_{l_s}^{(j)} \nabla_{\Gamma}l_{s,h}^{(j)}\cdot\nabla_{\Gamma}v_{s,h} \;{\rm d}s \nonumber \\
      & + \int_{\Gamma_h(t)} f_{l_s}^{(j)}({\bfa l}_{s,h}) v_{s,h} \;{\rm d}s - \int_{\Gamma_h(t)} \hat{g}_{s}^{(j)}({\bfa l}_h, {\bfa l_{s,h}})v_{s,h} \;{\rm d}s, \nonumber \\
      & \hspace{7cm} \forall v_{s,h}\in{\cal H}_{s,h}(\Gamma_h(t)),
    \label{eq:fem-ext-surf-eq}
    \end{align}
    and
    \begin{align}
      \otl{}{t} \int_{\Gamma_h(t)} c_{s,h}^{(q)}v_{s,h} \;{\rm d}s = & \int_{\Gamma_h(t)} \left\{\nabla_{\Gamma}\cdot\left[c_{s,h}^{(q)}({\bfa w}_h - {\bfa u}_{\Gamma})\right]\right\}v_{s,h} \;{\rm d}s - \int_{\Gamma_h(t)} D_{c_s}^{(q)} \nabla_{\Gamma}c_{s,h}^{(q)}\cdot\nabla_{\Gamma}v_{s,h} \;{\rm d}s \nonumber \\
      & + \int_{\Gamma_h(t)} f_{c_s}^{(q)}({\bfa c}_{s,h}) v_{s,h} \;{\rm d}s - \int_{\Gamma_h(t)} \hat{r}_s^{(q)}({\bfa c}_h, {\bfa c}_{s,h}, {\bfa l}_{s,h})v_{s,h} \;{\rm d}s, \nonumber \\
      & \hspace{7cm} \forall v_{s,h}\in{\cal H}_{s,h}(\Gamma_h(t)).
    \label{eq:fem-int-surf-eq}
    \end{align}
    
    For each $i = 1,\ldots,N_l$, $j = 1,\ldots,N_{l_s}$, $p = 1,\ldots,N_c$ and $q = 1,\ldots,N_{c_s}$, the finite element approximation of the bulk and surface species can be expressed as
    \begin{alignat}{2}
    	  l_h^{(i)}({\bfa x},t) = & \sum_{\mu = 1}^{{\cal N}^{\Xi}} l_\mu^{(i)}(t) \phi_\mu({\bfa x},t), \quad &\mbox{and} \quad l_{s,h}^{(j)}({\bfa x},t) = & \sum_{\mu = 1}^{{\cal N}_{s}^{\Gamma}} l_{s,\mu}^{(j)}(t) \phi_{s,\mu}({\bfa x}, t), \nonumber \\
      c_{h}^{(p)}({\bfa x},t) = & \sum_{\mu = 1}^{{\cal N}^{\Omega}} c_{\mu}^{(p)}(t) \phi_{\mu}({\bfa x},t), \quad &\mbox{and} \quad c_{s,h}^{(q)}({\bfa x},t) = & \sum_{\mu = 1}^{{\cal N}_{s}^{\Gamma}} c_{s,\mu}^{(q)}(t)\phi_{s,\mu}({\bfa x},t),
    \label{eq:fem-femapprox}
    \end{alignat}
    where $\left\{ \phi_{\mu}({\bfa x},t) \right\}_{\mu=1}^{{\cal N}}$ are the time-dependent nodal basis functions for the bulk exterior $({\cal N} = {\cal N}^{\Xi})$ and interior $({\cal N} = {\cal N}^{\Omega})$ environments, respectively; and $\left\{ \phi_{s,\mu}({\bfa x},t) \right\}_{\mu=1}^{{\cal N}_{s}^{\Gamma}}$ are the time-dependent surface nodal basis functions. Let 
    \[
      {\cal L}(t) = \{ {\bfa L}^{(i)}(t) \}_{i = 1}^{N_l}, \quad {\cal L}_{s}(t) = \{ {\bfa L}_{s}^{(j)}(t) \}_{j = 1}^{N_{l_s}}, \quad {\cal C}(t) = \{ {\bfa C}^{(p)}(t) \}_{p = 1}^{N_c} \quad \mbox{and} \quad {\cal C}_s(t) = \{ {\bfa C}_s^{(q)}(t) \}_{q = 1}^{N_{c_s}},
    \]
    such that, for each $i = 1,\ldots,N_l$, $j = 1,\ldots,N_{l_s}$, $p = 1,\ldots,N_c$ and $q = 1,\ldots,N_{c_s}$, 
    \begin{alignat}{2}
      {\bfa L}^{(i)}(t) = & \left\{ l_{\mu}^{(i)}(t) \right\}_{\mu=1}^{{\cal N}^{\Xi}}, \quad & {\bfa L}_{s}^{(j)}(t) = \left\{ l_{s,\mu}^{(j)}(t) \right\}_{\mu=1}^{{\cal N}_{s}^{\Gamma}}, \nonumber \\ 
      {\bfa C}^{(p)}(t) = & \left\{ c_{\mu}^{(p)}(t) \right\}_{\mu = 1}^{{\cal N}^{\Omega}}, \quad & {\bfa C}_s^{(q)}(t) = \left\{ c_{s,\mu}^{(q)}(t) \right\}_{\mu = 1}^{{\cal N}_{s}^{\Gamma}}. \nonumber
    \end{alignat}
    Then we can express \eqref{eq:fem-ext-bulk-eq} and \eqref{eq:fem-int-bulk-eq} as systems of ordinary differential equations
    \begin{align}
      \otl{}{t} \left[ M_{\Xi}(t){\bfa L}^{(i)}(t) \right] + & \left[ B_{\Xi}(t, {\bfa w}_{h}; {\bfa u}_{\Xi}) - A_{\Gamma}(t, {\bfa w}_{h}; {\bfa u}_{\Gamma}, {\bfa n}_{\Xi}) + K_{\Xi}(t; D_l^{(i)}) \right]{\bfa L}^{(i)}(t) \;- \nonumber \\
      & {\bfa G}^{(i)}(t, {\cal L}(t), {\cal L}_{s}(t)) = {\bfa F}_{l}^{(i)}(t, {\cal L}(t)),
    \label{eq:fem-ext-bulk-eq-sd}
    \end{align}
    and
    \begin{align}
      \otl{}{t} \left[ M_{\Omega}(t){\bfa C}^{(p)}(t) \right] + & \left[ B_{\Omega}(t, {\bfa w}_{h}; {\bfa u}_{\Omega}) - A_{\Gamma}(t, {\bfa w}_{h}; {\bfa u}_{\Gamma}, {\bfa n}_{\Omega}) + K_{\Omega}(t; D_c^{(p)}) \right]{\bfa C}^{(p)}(t) \;- \nonumber \\
      & {\bfa R}^{(p)}(t, {\cal C}(t), {\cal C}_{s}(t), {\cal L}_s(t)) = {\bfa F}_{c}^{(p)}(t, {\cal C}(t)),
    \label{eq:fem-int-bulk-eq-sd}
    \end{align}
    whose time-dependent matrices are given by
    \begin{subequations}
    \begin{align}
    \label{eq:fem-matrices-M}
      [ M_{\Xi}(t) ]_{\nu,\mu} = & \int_{\Xi_{h}(t)} \phi_{\mu}({\bfa x},t)\phi_{\nu}({\bfa x},t) \;{\rm d}{\bfa x}, \\
    \label{eq:fem-matrices-B}
      [ B_{\Xi}(t, {\bfa w}_{h}; {\bfa u}_{\Xi}) ]_{\nu,\mu} = & \int_{\Xi_{h}(t)} \phi_{\mu}({\bfa x},t)\left[ ({\bfa w}_{h} - {\bfa u}_{\Xi})\cdot\nabla \phi_{\nu}({\bfa x},t) \right] \;{\rm d}{\bfa x}, \\
    \label{eq:fem-matrices-A}
      [ A_{\Gamma}(t, {\bfa w}_{h}; {\bfa u}_{\Gamma}, {\bfa n}_{\Xi}) ]_{\nu,\mu} = & \int_{\Gamma_{h}(t)} \phi_{\mu}({\bfa x},t) \left[ ({\bfa w}_{h} - {\bfa u}_{\Gamma}) \cdot {\bfa n}_{\Xi} \right] \phi_{\nu}({\bfa x},t) \;{\rm d}s, \\
    \label{eq:fem-matrices-K}
      [ K_{\Xi}(t; D_l^{(i)}) ]_{\nu,\mu} = & \int_{\Xi_{h}(t)} D_l^{(i)} \nabla\phi_{\mu}({\bfa x},t) \cdot \nabla\phi_{\nu}({\bfa x},t) \;{\rm d}{\bfa x}.
    \end{align}
    \label{eq:fem-matrices}
    \end{subequations}
    The time-dependent matrices for the interior bulk environment $\Omega_h(t)$ which are given in \eqref{eq:fem-int-bulk-eq-sd} are defined similarly to the above with the appropriate substitutions for the interior region. The time-dependent boundary and load vectors are given by
    \begin{subequations}
    \begin{align}
      [ {\bfa G}^{(i)}(t, {\cal L}(t), {\cal L}_{s}(t)) ]_{\nu} = & \int_{\Gamma_{h}(t)} \left[ g^{(i)}({\bfa l}_{h}) + \hat{g}^{(i)}({\bfa l}_{h}, {\bfa l}_{s,h}) \right]\phi_{\nu}({\bfa x},t) \;{\rm d}s, \nonumber \\
      [{\bfa F}_{l}^{(i)}(t, {\cal L}(t))]_{\nu} = & \int_{\Xi_{h}(t)} f_l^{(i)}({\bfa l}_{h})\phi_{\nu}({\bfa x},t) \;{\rm d}{\bfa x}, \nonumber \\
      [{\bfa R}^{(p)}(t, {\cal C}(t), {\cal C}_{s}(t), {\cal L}_s(t))]_{\nu} = & \int_{\Gamma_h(t)} \left[ r^{(p)}({\bfa c}_h) + \hat{r}^{(p)}({\bfa c}_h, {\bfa c}_{s,h}, {\bfa l}_{s,h}) \right]\phi_{\nu}({\bfa x},t) \;{\rm d}s, \nonumber \\
      [{\bfa F}_{c}^{(p)}(t, {\cal C}(t))]_{\nu} = & \int_{\Omega_{h}(t)} f_c^{(p)}({\bfa c}_{h})\phi_{\nu}({\bfa x},t) \;{\rm d}{\bfa x}. \nonumber
    \end{align}
    \end{subequations}
    Similarly, we can express \eqref{eq:fem-ext-surf-eq} and \eqref{eq:fem-int-surf-eq} as systems of ordinary differential equations
    \begin{align}
      \otl{}{t} [ M_{s}(t){\bfa L}_{s}^{(j)}(t)] + [ K_{s}(t; D_{l_s}^{(j)}) - B_{s}(t,{\bfa w}_{h}; {\bfa u}_{\Gamma}) ]{\bfa L}_{s}^{(j)}(t) + {\bfa G}_{s}^{(j)}(t, {\cal L}(t), {\cal L}_{s}(t)) = {\bfa F}_{l_s}^{(j)}(t, {\cal L}_{s}(t)),
    \label{eq:fem-ext-surf-eq-sd}
    \end{align}
    and
    \begin{align}
    \otl{}{t} \left[ M_{s}(t){\bfa C}_{s}^{(q)}(t) \right] + & \left[ K_{s}(t; D_{c_s}^{(q)}) - B_{s}(t,{\bfa w}_{h}; {\bfa u}_{\Gamma}) \right]{\bfa C}_{s}^{(q)}(t) \;+ \nonumber \\
    & \hspace{3cm} {\bfa R}_{s}^{(q)}(t, {\cal C}(t), {\cal C}_s(t), {\cal L}_{s}(t)) = {\bfa F}_{c_s}^{(q)}(t, {\cal C}_{s}(t)),
    \label{eq:fem-int-surf-eq-sd}
    \end{align}
    whose time-dependent matrices are given by
    \begin{subequations}
    \begin{align}
    \label{eq:fem-matsurf-M}
      [ M_{s}(t) ]_{\nu,\mu} = & \int_{\Gamma_{h}(t)} \phi_{s,\mu}({\bfa x},t)\phi_{s,\nu}({\bfa x},t) \;{\rm d}s, \\
    \label{eq:fem-matsurf-K}
      [ K_{s}(t; D) ]_{\nu,\mu} = & \int_{\Gamma_{h}(t)} D\nabla_{\Gamma}\phi_{s,\mu}({\bfa x},t) \cdot \nabla_{\Gamma}\phi_{s,\nu}({\bfa x},t) \;{\rm d}s, \\
    \label{eq:fem-matsurf-B}
      [ B_{s}(t,{\bfa w}_{h}; {\bfa u}_{\Gamma}) ]_{\nu,\mu} = & \int_{\Gamma_{h}(t)} \left\{ \nabla_{\Gamma} \cdot [\phi_{s,\mu}({\bfa x},t)({\bfa w}_{h} - {\bfa u}_{\Gamma})] \right\}\phi_{s,\nu}({\bfa x},t) \;{\rm d}s,
    \end{align}
    \label{eq:fem-matsurf}
    \end{subequations}
    and time-dependent interaction and load vectors are given by
    \begin{subequations}
    \begin{align}
      [ {\bfa G}_{s}^{(j)}(t, {\cal L}(t), {\cal L}_{s}(t)) ]_{\nu} = & \int_{\Gamma_{h}(t)} \hat{g}_{s}^{(j)}({\bfa l}_{h}, {\bfa l}_{s,h})\phi_{s,\nu}({\bfa x},t) \;{\rm d}s, \nonumber \\
      [{\bfa F}_{l_s}^{(j)}(t, {\cal L}_{s}(t))]_{\nu} = & \int_{\Gamma_{h}(t)} f_{l_s}^{(j)}({\bfa l}_{s,h})\phi_{s,\nu}({\bfa x},t) \;{\rm d}s, \nonumber \\
      [ {\bfa R}_{s}^{(q)}(t, {\cal C}(t), {\cal C}_s(t), {\cal L}_{s}(t)) ]_{\nu} = & \int_{\Gamma_{h}(t)} \hat{r}_{s}^{(q)}({\bfa c}_{h}, {\bfa c}_{s,h}, {\bfa l}_{s,h})\phi_{s,\nu}({\bfa x},t) \;{\rm d}s, \nonumber\\
      [{\bfa F}_{c_s}^{(q)}(t, {\cal C}_{s}(t))]_{\nu} = & \int_{\Gamma_{h}(t)} f_{c_s}^{(q)}({\bfa c}_{s,h})\phi_{s,\nu}({\bfa x},t) \;{\rm d}s. \nonumber
    \end{align}    
    \end{subequations}

  \section{Temporal discretisation}
  \label{sec:tdisc}
    To obtain a temporal discretisation of \eqref{eq:fem-ext-bulk-eq-sd}, \eqref{eq:fem-ext-surf-eq-sd}, \eqref{eq:fem-int-bulk-eq-sd} and \eqref{eq:fem-int-surf-eq-sd}, we subdivide the time domain $[0,T]$ into $N_{T}$ equal time intervals of size $\Delta t = T/N_{T}$ and denote $t^{n} = n\Delta t$, $n = 0,\ldots,N_{T}$. Note that the superscript $n$ will be used to denote the time level $t^{n}$. Following \cite{macdonald_2016}, we discretise the ALE mapping using linear interpolation between time levels:
    \[
      {\cal A}_{h,\Delta t}({\bfa \xi},t) = \frac{t - t^{n}}{\Delta t}{\cal A}_{h}({\bfa \xi},t^{n+1}) + \frac{t^{n+1} - t}{\Delta t}{\cal A}_{h}({\bfa \xi},t^{n}), \qquad t\in[t^{n},t^{n+1}),
    \]
    where ${\cal A}_{h}(\cdot,t)$ is the piecewise linear map at time $t$. The mesh velocity is therefore piecewise constant in time and is given by
    \begin{align}
      {\bfa w}_{h,\Delta t}({\bfa \xi}, t^{n+1}) = & \frac{{\cal A}_{h}({\bfa \xi},t^{n+1}) - {\cal A}_{h}({\bfa \xi},t^{n})}{\Delta t}, \nonumber \\
      {\bfa w}_{h,\Delta t}({\bfa x},t) = & {\bfa w}_{h,\Delta t}({\bfa \xi}, t^{n+1}) \circ {\cal A}^{-1}_{h,\Delta t}({\bfa x},t), \qquad t\in[t^{n},t^{n+1}). \nonumber
    \end{align}
    
    Following \cite{macdonald_2016}, the temporal discretisation of \eqref{eq:fem-ext-bulk-eq-sd}, \eqref{eq:fem-ext-surf-eq-sd}, \eqref{eq:fem-int-bulk-eq-sd} and \eqref{eq:fem-int-surf-eq-sd} is obtained using a modified Crank-Nicolson semi-implicit approach. First, for each $j = 1,\ldots,N_{l_s}$ and $q = 1,\ldots,N_{c_s}$, the surface approximations \eqref{eq:fem-ext-surf-eq-sd} and \eqref{eq:fem-int-surf-eq-sd} are predicted using an implicit-explicit Euler method:
    \begin{align}
    & \left\{ M_{s}^{n+1} + \Delta t \left[ K_{s}^{n+1}(D_{l_s}^{(j)}) - B_{s}^{n+1}({\bfa w}_{h}^{n+1}; {\bfa u}_{\Gamma}^{n+1}) \right] \right\}\widetilde{{\bfa L}}_{s}^{(j),n+1} = \nonumber \\
    & \hspace{3cm} M_s^n {\bfa L}_s^{(j),n} - \Delta t \left[ {\bfa G}_{s}^{(j),n}({\cal L}^n, {\cal L}_{s}^n) - {\bfa F}_{l_s}^{(j),n}({\cal L}_{s}^n) \right], \nonumber
    \end{align}
    and
    \begin{align}
    & \left\{ M_{s}^{n+1} + \Delta t \left[ K_{s}^{n+1}(D_{c_s}^{(q)}) - B_{s}^{n+1}({\bfa w}_{h}^{n+1}; {\bfa u}_{\Gamma}^{n+1}) \right] \right\} \widetilde{{\bfa C}}_{s}^{(q),n+1} = \nonumber \\
    & \hspace{3cm} M_s^n {\bfa C}_s^{(q),n} - \Delta t \left[ {\bfa R}_{s}^{(q),n}({\cal C}^n, {\cal C}_s^n, {\cal L}_{s}^n) - {\bfa F}_{c_s}^{(q),n}({\cal C}_{s}^n) \right],
    \label{eq:tdisc-int-pred}
    \end{align}
    respectively, where the diffusion and advection terms are treated implicitly whilst the reaction terms are dealt with explicitly. For each $i = 1,\ldots,N_l$ and $p = 1,\ldots,N_c$, the bulk approximations \eqref{eq:fem-ext-bulk-eq-sd} and \eqref{eq:fem-int-bulk-eq-sd} are then updated using a Crank-Nicolson step
    \begin{align}
      & \left\{ M_{\Xi}^{n+1} + \frac{\Delta t}{2}\left[ B_{\Xi}^{n+1}({\bfa w}_{h}^{n+1}; {\bfa u}_{\Xi}^{n+1}) - A_{\Gamma}^{n+1}({\bfa w}_{h}^{n+1}; {\bfa u}_{\Gamma}^{n+1}, {\bfa n}_{\Xi}^{n+1}) + K_{\Xi}^{n+1}(D_l^{(i)}) \right] \right\}{\bfa L}^{(i),n+1} \;= \nonumber \\
      & \hspace{0.5cm} \left\{ M_{\Xi}^{n} - \frac{\Delta t}{2}\left[ B_{\Xi}^{n}({\bfa w}_{h}^{n}; {\bfa u}_{\Xi}^{n}) - A_{\Gamma}^{n}({\bfa w}_{h}^{n}; {\bfa u}_{\Gamma}^{n}, {\bfa n}_{\Xi}^{n}) + K_{\Xi}^{n}(D_l^{(i)}) \right] \right\}{\bfa L}^{(i),n} \;+ \nonumber \\ 
      & \hspace{0.5cm} \frac{\Delta t}{2} \left[ {\bfa G}^{(i),n+1}({\cal L}^{n+1}, \widetilde{{\cal L}}_{s}^{n+1}) + {\bfa G}^{(i),n}({\cal L}^{n}, {\cal L}_{s}^{n}) \right] + \frac{\Delta t}{2} \left[ {\bfa F}_{l}^{(i),n+1}({\cal L}^{n+1}) + {\bfa F}_{l}^{(i),n}({\cal L}^{n}) \right],
    \label{eq:tdisc-ext-bulk}
    \end{align}
    and
    \begin{align}
      & \left\{ M_{\Omega}^{n+1} + \frac{\Delta t}{2} \left[ B_{\Omega}^{n+1}({\bfa w}_{h}^{n+1}; {\bfa u}_{\Omega}^{n+1}) - A_{\Gamma}^{n+1}({\bfa w}_{h}^{n+1}; {\bfa u}_{\Gamma}^{n+1}, {\bfa n}_{\Omega}^{n+1}) + K_{\Omega}^{n+1}(D_c^{(p)}) \right] \right\}{\bfa C}^{(p),n+1} \;= \nonumber \\
      & \hspace{0.5cm} \left\{ M_{\Omega}^{n} - \frac{\Delta t}{2} \left[ B_{\Omega}^{n}({\bfa w}_{h}^{n}; {\bfa u}_{\Omega}^{n}) - A_{\Gamma}^{n}({\bfa w}_{h}^{n}; {\bfa u}_{\Gamma}^{n}, {\bfa n}_{\Omega}^{n}) + K_{\Omega}^{n}(D_c^{(p)}) \right] \right\}{\bfa C}^{(p),n} \;+ \nonumber \\
      & \hspace{0.5cm} \frac{\Delta t}{2} \left[ {\bfa R}^{(p),n+1}({\cal C}^{n+1}, \widetilde{{\cal C}}_{s}^{n+1}, \widetilde{{\cal L}}_s^{n+1}) + {\bfa R}^{(p),n}({\cal C}^{n}, {\cal C}_{s}^{n}, {\cal L}_s^{n}) \right] + \frac{\Delta t}{2} \left[ {\bfa F}_{c}^{(p),n+1}({\cal C}^{n+1}) + {\bfa F}_{c}^{(p),n}({\cal C}^{n}) \right],
    \label{eq:tdisc-int-bulk}
    \end{align}
    respectively, where $\widetilde{{\cal L}}_{s}(t) = \{ \widetilde{{\bfa L}}_{s}^{(j)}(t) \}_{j=1}^{N_{l_s}}$ and $\widetilde{{\cal C}}_{s}(t) = \{ \widetilde{{\bfa C}}_{s}^{(q)}(t) \}_{q=1}^{N_{c_s}}$, $\forall t \in \overline{I}$. Note that in general \eqref{eq:tdisc-ext-bulk} and \eqref{eq:tdisc-int-bulk} result in nonlinear systems. In order to guarantee that the surface approximations are second-order, we perform the correction steps
    \begin{align}
    & \left\{ M_{s}^{n+1} + \frac{\Delta t}{2} \left[ K_{s}^{n+1}(D_{l_s}^{(j)}) - B_{s}^{n+1}({\bfa w}_{h}^{n+1}; {\bfa u}_{\Gamma}^{n+1}) \right] \right\}{\bfa L}_{s}^{(j),n+1} = \nonumber \\
    & \hspace{1cm} \left\{ M_{s}^{n} - \frac{\Delta t}{2} \left[ K_{s}^{n}(D_{l_s}^{(j)}) - B_{s}^{n}({\bfa w}_{h}^{n}; {\bfa u}_{\Gamma}^{n}) \right] \right\} {\bfa L}_s^{(j),n} + \frac{\Delta t}{2} \left[ {\bfa F}_{l_s}^{(j),n+1}(\widetilde{{\cal L}}_{s}^{n+1}) + {\bfa F}_{l_s}^{(j),n}({\cal L}_{s}^{n}) \right] \;- \nonumber \\
    & \hspace{1cm} \frac{\Delta t}{2} \left[ {\bfa G}_{s}^{(j),n+1}({\cal L}^{n+1}, \widetilde{{\cal L}}_{s}^{n+1}) + {\bfa G}_{s}^{(j),n}({\cal L}^{n}, {\cal L}_{s}^{n}) \right],
    \label{eq:tdisc-ext-corr}
    \end{align}
    and
    \begin{align}
    & \left\{ M_{s}^{n+1} + \frac{\Delta t}{2} \left[ K_{s}^{n+1}(D_{c_s}^{(q)}) - B_{s}^{n+1}({\bfa w}_{h}^{n+1}; {\bfa u}_{\Gamma}^{n+1}) \right] \right\} {\bfa C}_{s}^{(q),n+1} = \nonumber \\
    & \hspace{1cm} \left\{ M_{s}^{n} - \frac{\Delta t}{2} \left[ K_{s}^{n}(D_{c_s}^{(q)}) - B_{s}^{n}({\bfa w}_{h}^{n}; {\bfa u}_{\Gamma}^{n}) \right] \right\} {\bfa C}_s^{(q),n} + \frac{\Delta t}{2} \left[ {\bfa F}_{c_s}^{(q),n+1}(\widetilde{{\cal C}}_{s}^{n+1}) + {\bfa F}_{c_s}^{(q),n}({\cal C}_{s}^{n}) \right] \;- \nonumber \\
    & \hspace{1cm} \frac{\Delta t}{2} \left[ {\bfa R}_{s}^{(q),n+1}({\cal C}^{n+1}, \widetilde{{\cal C}}_s^{n+1}, \widetilde{{\cal L}}_{s}^{n+1}) + {\bfa R}_{s}^{(q),n}({\cal C}^n, {\cal C}_s^n, {\cal L}_{s}^n) \right],
    \label{eq:tdisc-int-corr}
    \end{align}
    respectively. Note that the correction steps \eqref{eq:tdisc-ext-corr} and \eqref{eq:tdisc-int-corr} require the solution of linear systems as ${\cal L}^{n+1}$, $\widetilde{{\cal L}}_s^{n+1}$, ${\cal C}^{n+1}$ and $\widetilde{{\cal C}}_{s}^{n+1}$ are known quantities.
    \begin{remark}
    \label{rem:tdisc-imex}
    		In the results presented in \S\ref{sec:nexp-cc}, only interior bulk species will be considered and thus, we choose an implicit-explicit backward Euler scheme in place of \eqref{eq:tdisc-int-bulk}, where the ALE and diffusion terms are treated implicitly whilst the reaction terms and boundary conditions are treated explicitly. In the results presented in \S\ref{sec:nexp-cell}, the bulk reaction terms are linear with respect to the bulk species and therefore, \eqref{eq:tdisc-ext-bulk} and \eqref{eq:tdisc-int-bulk} result in linear systems.
    \end{remark}

  \section{Mesh generation}
  \label{sec:meshgen}

    \subsection{Surface domain}
    \label{sec:meshgen-sf}
      Following \cite{macdonald_2016}, we consider boundary movements given by the following evolution law for the normal velocity
      \begin{equation}
	    {\cal V}({\bfa x},t) := \dot{\bfa x} \cdot {\bfa n} = \alpha({\bfa x},t)\kappa + \beta({\bfa x},t), \qquad {\bfa x} \in \Gamma(t),
      \label{eq:meshgen-sf-normalvel}
      \end{equation}
      where $\alpha$ and $\beta$ are given functions and $\kappa$ is the curvature. In this article, the curve $\Gamma(t)$ is represented by the position of a discrete set of nodal points on the curve. These mesh nodes evolve according to the equation 
      \begin{equation}
	    \dot{\bfa x} = {\cal V}{\bfa n} + {\cal B}{\bfa t}, 
      \label{eq:meshgen-sf-bndeqn}
      \end{equation}
      where ${\bfa n}$ and ${\bfa t}$ are the outward unit normal and unit tangent vectors, respectively. The tangential motion ${\cal B}$ is necessary to maintain mesh quality.
      
      Following \cite{mackenzie_2019}, the normal and tangential velocities can then be expressed in terms of the parameterisation $[0,1] \ni \sg \mapsto {\bfa x}(\sg,t) \in \Gamma(t)$:
      \begin{subequations}
        \begin{align}
        \label{eq:meshgen-sf-normalsigma}
	      & {\cal V} = \dot{\bfa x} \cdot {\bfa n} = \alpha({\bfa x},t) \left( \frac{{\bfa x}_{\sg\sg} \cdot {\bfa n}}{|{\bfa x}_{\sg}|^{2}} \right) + \beta({\bfa x},t), \\
        \label{eq:meshgen-sf-tangentsigma}
	      & {\cal B} = \dot{\bfa x} \cdot {\bfa t} = \frac{P}{\tau} (M|{\bfa x}_{\sg}|)^{-2}(M|{\bfa x}_{\sg}|)_{\sg},      
        \end{align}
        \label{eq:meshgen-sf-normtan}
      \end{subequations}
      where $F_{\sg}=\partial F/\partial \sg$, $|\bfa{a}|$ denotes the $l_{2}$ norm and $\tau$ is the mesh relaxation time. In this article, we do not consider adaptivity and therefore, the spatial balancing operator $P = 1$ and monitor function $M = 1$ but are included above for completeness. In the interests of brevity, we omit the derivation of \eqref{eq:meshgen-sf-normtan} as they are covered in detail in \cite{mackenzie_2019}. To solve \eqref{eq:meshgen-sf-bndeqn}, we follow \cite{macdonald_2016, mackenzie_2019} by using second-order central finite differences to approximate the spatial derivatives and a first-order backward Euler scheme to approximate the temporal derivatives. 
      
      As the mesh is evolved from time level $t^{n}$ to $t^{n+1}$, $n = 0,\ldots,N_T$, the surface is evolved first and then the nodal positions are used as Dirichlet boundary conditions for the bulk mesh movement which we discuss in the next section.

    \subsection{Bulk domain}
    \label{sec:meshgen-blk}
      In the interests of brevity, we discuss the motion of the interior region $\Omega_h(t)$ only, as the procedure is identical for the extracellular environment with the minor addition, that after the surface $\Gamma_h(t)$ has evolved, the outer boundary is translated so that the interior region remains in the centre. 
      
      The evolution of the bulk mesh is assumed to satisfy a Moving Mesh Partial Differential Equation (MMPDE) \cite{huang_1994, huang_1997, huang_1998, huang_2001}. The MMPDE is derived for the inverse ALE mapping ${\cal A}^{-1}({\bfa x},t) = {\bfa \xi}({\bfa x},t)$ as this prevents mesh crossings or foldings \cite{dvinsky_1991}. The mapping ${\bfa \xi}({\bfa x}) = (\xi({\bfa x}),\eta({\bfa x}))$, corresponding to a fixed $t$, is chosen so that it minimises the functional \cite{huang_1998}
      \begin{equation}
		I[{\bfa \xi}] = \frac{1}{2} \int_{\Omega(t)} \left[ (\nabla\xi)^{T}G^{-1}(\nabla\xi) + (\nabla\eta)^{T}G^{-1}(\nabla\eta) \right] \;{\rm d}{\bfa x}, \nonumber
      \end{equation}
      where $G$ is a $2\times 2$ monitor matrix. Following \cite{macdonald_2016}, we choose the monitor matrix to be
      \[
	G = \begin{bmatrix}
	      M & 0 \\
	      0 & M
	    \end{bmatrix},
      \]
      where $M$ is the monitor function. Rather than minimising the functional directly, the mapping is evolved according to the modifed gradient flow equations \cite{huang_1998, macdonald_2016}
      \[
	\ptl{\xi}{t} = \frac{\widehat{P}}{\widehat{\tau}}\nabla \cdot \left(G^{-1}\nabla\xi \right), \qquad \mbox{and} \qquad 
	\ptl{\eta}{t} = \frac{\widehat{P}}{\widehat{\tau}}\nabla \cdot \left(G^{-1}\nabla\eta \right),
      \]
      where $\widehat{\tau}$ denotes the mesh relaxation time and $\widehat{P}$ denotes the spatial balancing operator. To obtain an evolution equation for the physical mesh points $\{{\bfa x}_{i}(t)\}_{i=1}^{{\cal N}^{\Omega}}$, we interchange the dependent and independent variables to give the MMPDE \cite{macdonald_2016}
      \begin{equation}
	\widehat{\tau}\, \widehat{p}({\bfa x}, t) \ptl{{\bfa x}}{t} = \left( a{\bfa x}_{\xi\xi} + b{\bfa x}_{\xi\eta} + c{\bfa x}_{\eta\eta} + d{\bfa x}_{\xi} + e{\bfa x}_{\eta} \right), \qquad (\xi,\eta) \in \Omega_{c},
      \label{eq:meshgen-blk-mmpde}
      \end{equation}
      where $\widehat{p}({\bfa x}, t) = 1 / \widehat{P}({\bfa x},t)$. The coefficients of the MMPDE are defined by \cite{macdonald_2016}
      \begin{align}
	& a = \frac{1}{M}\left[ \frac{x_{\eta}^{2} + y_{\eta}^{2}}{J^{2}} \right], \quad 
	  b = -\frac{2}{M}\left[ \frac{ x_{\xi}x_{\eta} + y_{\xi}y_{\eta} }{J^{2}} \right], \quad
	  c = \frac{1}{M}\left[ \frac{x_{\xi}^{2} + y_{\xi}^{2}}{J^{2}} \right], \nonumber \\
	& d = \frac{1}{M^{2}}\left[ \frac{M_{\xi}\left( x_{\eta}^{2} + y_{\eta}^{2} \right) - M_{\eta}\left( x_{\xi}x_{\eta} + y_{\xi}y_{\eta} \right) }{J^{2}} \right], \nonumber \\
	& e = \frac{1}{M^{2}}\left[ \frac{M_{\eta}\left( x_{\xi}^{2} + y_{\xi}^{2} \right) - M_{\xi}\left( x_{\xi}x_{\eta} + y_{\xi}y_{\eta} \right) }{J^{2}} \right],
      \label{eq:meshgen-blk-coeffs}
      \end{align}
      where $J = x_{\xi}y_{\eta} - x_{\eta}y_{\xi}$ is the Jacobian of the mapping. To construct the initial bulk mesh, the equidistributed surface mesh is used as a set of fixed points for a Fortran implementation of the Matlab toolbox Distmesh \cite{persson_2004}. The initial computational mesh is then defined as: $\Omega_{c,h} = \Omega_{h}(0)$. As mentioned in \S\ref{sec:meshgen-sf}, the nodal positions of the surface mesh evolution are used as Dirichlet boundary conditions for the bulk MMPDE (\ref{eq:meshgen-blk-mmpde}).
      
      For the numerical solution of (\ref{eq:meshgen-blk-mmpde}), we discretise in space using piecewise linear finite elements and in time using a backward Euler method. To avoid solving nonlinear algebraic systems, the coefficients (\ref{eq:meshgen-blk-coeffs}), and spatial balancing operator $\widehat{p}$, are treated explicitly. The discretisation of the MMPDE (\ref{eq:meshgen-blk-mmpde}) therefore takes the form: find ${\bfa x}_{h}^{n+1} \in \left[ {\cal L}^{1}(\Omega_{c,h}) \right]^{2}$ such that
      \begin{align}
	    & \widehat{\tau} \int_{\Omega_{c,h}} \widehat{p}^{\,n}({\bfa x}_{h}) \left( \frac{{\bfa x}_{h}^{n+1} - {\bfa x}_{h}^{n}}{\Delta t} \right) \cdot \hat{{\bfa v}}_{h} \;{\rm d}{\bfa\xi} +
	  \int_{\Omega_{c,h}} \left( {\bfa x}_{h}^{n+1} \right)_{\xi} \cdot \left( a^{n}\hat{{\bfa v}}_{h} \right)_{\xi} \;{\rm d}{\bfa \xi} + \int_{\Omega_{c,h}} \left( {\bfa x}_{h}^{n+1} \right)_{\eta} \cdot \left( c^{n}\hat{{\bfa v}}_{h} \right)_{\eta} \;{\rm d}{\bfa \xi} \;+ \nonumber \\
	    & \hspace{0.5cm} \frac{1}{2}\int_{\Omega_{c,h}} \left( {\bfa x}_{h}^{n+1} \right)_{\xi} \cdot \left( b^{n}\hat{{\bfa v}}_{h} \right)_{\eta} + \left( {\bfa x}_{h}^{n+1} \right)_{\eta} \cdot \left( b^{n}\hat{{\bfa v}}_{h} \right)_{\xi} \;{\rm d}{\bfa \xi} 
	  = \int_{\Omega_{c,h}} \left[ d^{n}\left( {\bfa x}_{h}^{n+1} \right)_{\xi} + e^{n} \left( {\bfa x}_{h}^{n+1} \right)_{\eta} \right] \cdot \hat{{\bfa v}}_{h} \;{\rm d}{\bfa \xi}, \nonumber
      \end{align}
      for all $\hat{{\bfa v}}_{h} \in \left[ {\cal L}_{0}^{1}(\Omega_{c,h}) \right]^{2}$. The resulting linear system is solved using the UMFPACK library \cite{davis_2004} of the SuiteSparse software collection \cite{davis_2006, davis_nodate}. 
      
      As the positions $\{ {\bfa x}_{i}(t) \}_{i=1}^{{\cal N}^{\Omega}}$ are approximated using piecewise linear finite elements, the coefficients are in general piecewise constant and hence, not globally continuous. Therefore, within a standard Galerkin formulation, the coefficients require a numerical recovery procedure so that they can be differentiated accurately. Here, we use a spatial averaging technique; for example, let $a\colon \Omega_{c,h} \rightarrow \RR$ be a piecewise constant function, then one can define the function $\widetilde{a}$ by
      \begin{align}
	    \widetilde{a} = \frac{1}{|\Omega_{c,h}|} \int_{\Omega_{c,h}} a \;{\rm d}{\bfa \xi} = \frac{1}{|\Omega_{c,h}|} \sum_{K \in {\cal T}_{h,c}^{\Omega}} \int_{K} \widehat{a}_{K} \;{\rm d}{\bfa \xi} = \frac{1}{|\Omega_{c,h}|} \sum_{K \in {\cal T}_{h,c}^{\Omega}} \widehat{a}_{K} |K|, \nonumber
      \end{align}
      where $\widehat{a}_{K} = a|_{K}$ is the local restriction of $a$ to element $K$ of the triangulation ${\cal T}_{h,c}^{\Omega}$. This procedure results in $\widetilde{a} \in C^{0}(\Omega_{c,h})$ and therefore, we can approximate $\widetilde{a}$ using piecewise linear finite elements. The finite element approximation $\widetilde{a}_{h}$ can then be differentiated in the usual way. Coincidentally, we also note that the spatial balancing operator may in general depend on these coefficients and is therefore calculated using the recovered coefficients. The increased regularity of the coefficients, due to the numerical recovery technique, produces a more robust moving mesh procedure and therefore, in general, will require remeshing techniques to be used far less frequently.

  \section{Global conservation of the discrete formulation on evolving domains}
  \label{sec:conserv}

    \subsection{Spatially semi-discrete}
    \label{sec:conserv-sd}
      For the reasons mentioned at the beginning of \S\ref{sec:conserv-cont}, we consider global conservation of the interior bulk and surface species only. For notational convenience, we let ${\bfa e}^{T}$ and ${\bfa e}_s^{T}$ denote the ${\cal N}^{\Omega}$-dimensional and ${\cal N}_s^{\Gamma}$-dimensional row vectors of ones, respectively. For the proof of global conservation of the bulk-surface reaction-diffusion system given by \eqref{eq:rd-int-bulk} and \eqref{eq:rd-int-surf}, in the semi-discrete sense, we require the following observations:
      
      \begin{lemma}
      \label{lm:sd-totalamount}
	    At any time $t\in \overline{I}$, for each $p = 1,\ldots,N_c$ and $q = 1,\ldots,N_{c_s}$, the total amount of $c_{h}^{(p)}$ and $c_{s,h}^{(q)}$ can be determined in terms of the column sums of the time-dependent mass matrices; that is,
	    \[
	      \int_{\Omega_{h}(t)} c_{h}^{(p)}({\bfa x},t) \;{\rm d}{\bfa x} = {\bfa e}^{T}M_{\Omega}(t){\bfa C}^{(p)}(t), \quad \mbox{ and } \quad \int_{\Gamma_{h}(t)} c_{s,h}^{(q)}({\bfa x},t) \;{\rm d}s = {\bfa e}_{s}^{T}M_{s}(t){\bfa C}_{s}^{(q)}(t),
	    \]
	    where the time dependent mass matrices $M_{\Omega}(t)$ and $M_s(t)$ are as defined in \eqref{eq:fem-matrices-M} and \eqref{eq:fem-matsurf-M}, respectively.
      \end{lemma}
      \begin{proof}
	First, as a consequence of the finite element approximation, we note that the bulk and surface nodal basis functions form partitions of unity
	\begin{equation}
	\label{eq:sd-pum}
	  \sum_{\mu=1}^{{\cal N}^{\Omega}}\phi_{\mu}({\bfa x},t) = 1, \quad \mbox{and} \quad \sum_{\mu=1}^{{\cal N}_{s}^{\Gamma}}\phi_{s,\mu}({\bfa x},t) = 1.
	\end{equation}
	Then, for each $p = 1,\ldots,N_c$ and $q = 1,\ldots,N_{c_s}$, the proof follows straightforwardly by integrating $c_{h}^{(p)}(t)$ and $c_{s,h}^{(q)}(t)$ over $\Omega_{h}(t)$ and $\Gamma_h(t)$, respectively, and using the finite element approximation \eqref{eq:fem-femapprox}. 
	
      \end{proof}
      
      \begin{lemma}
      \label{lm:sd-stiffness}
	    For each $p = 1,\ldots,N_c$ and $q = 1,\ldots,N_{c_s}$, the column sums of the time-dependent bulk and surface stiffness matrices are zero; that is,
	    \[
	      {\bfa e}^{T}K_{\Omega}(t; D_c^{(p)}) = 0, \quad \mbox{and} \quad {\bfa e}_s^{T}K_{s}(t; D_{c_s}^{(q)}) = 0,
	    \]
	    where the time dependent stiffness matrices $K_{\Omega}(t; D_c^{(p)})$ and $K_s(t; D_{c_s}^{(q)})$ are as defined in \eqref{eq:fem-matrices-K} and \eqref{eq:fem-matsurf-K}, respectively.
      \end{lemma}
      \begin{proof}
	    Making use of the fact that both the gradient and tangential gradient operators are linear and that the finite element basis functions form a partition of unity \eqref{eq:sd-pum}, the result follows immediately:
	    \begin{align}
	      {\bfa e}^{T}K_{\Omega}(t; D_c^{(p)}) = \sum_{\nu=1}^{{\cal N}^{\Omega}} \left[ K_{\Omega}(t; D_c^{(p)}) \right]_{\nu,\mu} = & \sum_{\nu=1}^{{\cal N}^{\Omega}} \int_{\Omega_{h}(t)} D_c^{(p)} \nabla\phi_{\mu}({\bfa x},t) \cdot \nabla\phi_{\nu}({\bfa x},t) \;{\rm d}{\bfa x} \nonumber \\
	      = & \int_{\Omega_{h}(t)} D_c^{(p)} \nabla\phi_{\mu}({\bfa x},t) \cdot \nabla\left(\sum_{\nu=1}^{{\cal N}^{\Omega}} \phi_{\nu}({\bfa x},t) \right) \;{\rm d}{\bfa x} = 0, \nonumber
	    \end{align}
	    for each $p = 1,\ldots,N_c$. Similarly for the surface stiffness matrix $K_s(t; D_{c_s}^{(q)})$.
	
      \end{proof}
      
      \begin{lemma}
      \label{lm:sd-advection}
	    The column sums of the time-dependent bulk and surface advection matrices are zero; that is
	    \[
	      {\bfa e}^{T}B_{\Omega}(t, {\bfa w}_h; {\bfa u}_{\Omega}) = 0, \quad \mbox{and} \quad {\bfa e}_s^{T}B_s(t, {\bfa w}_h; {\bfa u}_{\Gamma}) = 0,
	    \]
	    where the time-dependent advection matrices $B_{\Omega}(t, {\bfa w}_h; {\bfa u}_{\Omega})$ and $B_s(t, {\bfa w}_h; {\bfa u}_{\Gamma})$ are as defined in \eqref{eq:fem-matrices-B} and \eqref{eq:fem-matsurf-B}, respectively.
      \end{lemma}
      \begin{proof}
	    Once again, making use of the linearity of the gradient operator as well as the partition of unity of the finite element basis functions \eqref{eq:sd-pum}, we have for the bulk advection matrix
	    \begin{align}
	      {\bfa e}^{T}B_{\Omega}(t, {\bfa w}_h; {\bfa u}_{\Omega}) = & \sum_{\nu=1}^{{\cal N}^{\Omega}} \left[ B_{\Omega}(t, {\bfa w}_{h}; {\bfa u}_{\Omega}) \right]_{\nu,\mu} = \sum_{\nu=1}^{{\cal N}^{\Omega}}\int_{\Omega_{h}(t)} \phi_{\mu}({\bfa x},t) \left[ ({\bfa w}_{h} - {\bfa u}_{\Omega})\cdot\nabla \phi_{\nu}({\bfa x},t) \right] \;{\rm d}{\bfa x} \nonumber \\
	      = & \int_{\Omega_{h}(t)} \phi_{\mu}({\bfa x},t) \left[ ({\bfa w}_{h} - {\bfa u}_{\Omega})\cdot\nabla \left(\sum_{\nu=1}^{{\cal N}^{\Omega}}\phi_{\nu}({\bfa x},t)\right)\right] \;{\rm d}{\bfa x} = 0. \nonumber
	    \end{align}
	    Using the fact that the curve $\Gamma_{h}(t)$ is closed and the partition of unity of the surface nodal basis functions, the result for the surface advection matrix $B_s(t, {\bfa w}_h; {\bfa u}_{\Gamma})$ follows similarly.
	
      \end{proof}
      
      Similar to the continuous case \S\ref{sec:conserv-cont}, in order to demonstrate global conservation we require the following assumption:
      
      \begin{assump}
      \label{ass:sd}
		We assume that on the boundary $\Gamma_h(t)$, the normal components of the ALE velocity and material velocity are identical; that is, ${\bfa w}_h \cdot {\bfa n}_h = {\bfa u}_{\Gamma} \cdot {\bfa n}_h$ where ${\bfa n}_h$ denotes the outward unit normal to $\Gamma_h(t)$.
      \end{assump}
      
      We are now ready to prove global conservation of the semi-discrete system given by \eqref{eq:fem-int-bulk-eq-sd} and \eqref{eq:fem-int-surf-eq-sd}.
      
      \begin{theorem}
      \label{thm:conserv-sd}
		Under the conditions of Assumption \ref{ass:sd} and provided
		\begin{equation}
		\label{eq:conserv-sd-sources}
		\sum_{p = 1}^{N_c} \int_{\Omega_h(t)} f_c^{(p)}({\bfa c}_h) \;{\rm d}{\bfa x} = 0, \quad \sum_{p = 1}^{N_c} \int_{\Gamma_h(t)} r^{(p)}({\bfa c}_h) \;{\rm d}s = 0, \quad \sum_{q = 1}^{N_{c_s}} \int_{\Gamma_h(t)} f_{c_s}^{(q)}({\bfa c}_{s,h}) \;{\rm d}s = 0,
		\end{equation}
		and
		\begin{equation}
		\label{eq:conserv-sd-interact}
		\sum_{p = 1}^{N_c} \int_{\Gamma_h(t)} \hat{r}^{(p)}({\bfa c}_h, {\bfa c}_{s,h}, {\bfa l}_{s,h}) \;{\rm d}s = \sum_{q = 1}^{N_{c_s}} \int_{\Gamma_h(t)} \hat{r}_s^{(q)}({\bfa c}_h, {\bfa c}_{s,h}, {\bfa l}_{s,h}) \;{\rm d}s,
		\end{equation}
		are satisfied, then the semi-discrete system given by \eqref{eq:fem-int-bulk-eq-sd} and \eqref{eq:fem-int-surf-eq-sd} is globally conservative independent of the ALE velocity; that is, given Lemma \ref{lm:sd-totalamount}, the following holds true:
		\[
	  	  \otl{}{t} \left\{ \sum_{p=1}^{N_c} {\bfa e}^{T}M_{\Omega}(t){\bfa C}^{(p)}(t) + \sum_{q=1}^{N_{c_s}} {\bfa e}_{s}^{T}M_{s}(t){\bfa C}_{s}^{(q)}(t) \right\} = 0,
		\]
		where the time-dependent mass matrices are given by \eqref{eq:fem-matrices-M} and \eqref{eq:fem-matsurf-M}.
      \end{theorem}
      \begin{proof}
        Following from Assumption \ref{ass:sd}, the matrix $A_{\Gamma}(t,{\bfa w}_{h}; {\bfa u}_{\Gamma}, {\bfa n}_{\Omega})$ in \eqref{eq:fem-matrices-A} is zero. We take the column sum of \eqref{eq:fem-int-bulk-eq-sd}, apply Lemmas \ref{lm:sd-stiffness} and \ref{lm:sd-advection}, and sum over $p = 1,\ldots,N_c$ to give
        \begin{equation}
          \otl{}{t} \left[ \sum_{p=1}^{N_c}{\bfa e}^{T}M_{\Omega}(t){\bfa C}^{(p)}(t) \right] - \sum_{p=1}^{N_c}{\bfa e}^{T}{\bfa R}^{(p)}(t, {\cal C}(t), {\cal C}_{s}(t), {\cal L}_s(t)) = \sum_{p=1}^{N_c}{\bfa e}^{T}{\bfa F}_{c}^{(p)}(t, {\cal C}(t)).
          \label{eq:conserv-sd-bulk-eq1}
        \end{equation}
        Similarly, taking the column sum of \eqref{eq:fem-int-surf-eq-sd}, applying Lemmas \ref{lm:sd-stiffness} and \ref{lm:sd-advection}, and summing over $q = 1,\ldots,N_{c_s}$ yields
        \begin{equation}
    	      \otl{}{t} \left[ \sum_{q=1}^{N_{c_s}}{\bfa e}_{s}^{T}M_{s}(t){\bfa C}_{s}^{(q)}(t) \right] + \sum_{q=1}^{N_{c_s}}{\bfa e}_{s}^{T}{\bfa R}_{s}^{(q)}(t, {\cal C}(t), {\cal C}_s(t), {\cal L}_{s}(t)) = \sum_{q=1}^{N_{c_s}}{\bfa e}_{s}^{T}{\bfa F}_{c_s}^{(q)}(t, {\cal C}_{s}(t)).
        \label{eq:conserv-sd-surf-eq1}
        \end{equation}
        The second term on the left hand side of \eqref{eq:conserv-sd-bulk-eq1} is given by
        \begin{align}
          \sum_{p=1}^{N_c}{\bfa e}^{T}{\bfa R}^{(p)}(t, {\cal C}(t), {\cal C}_{s}(t), {\cal L}_s(t)) = & \sum_{p=1}^{N_c} \sum_{\nu=1}^{{\cal N}^{\Omega}} \left[ {\bfa R}^{(p)}(t, {\cal C}(t), {\cal C}_{s}(t), {\cal L}_s(t)) \right]_{\nu} \nonumber \\
          = & \sum_{p=1}^{N_c} \int_{\Gamma_h(t)} \left[ r^{(p)}({\bfa c}_h) + \hat{r}^{(p)}({\bfa c}_h, {\bfa c}_{s,h}, {\bfa l}_{s,h}) \right]\left( \sum_{\nu=1}^{{\cal N}^{\Omega}} \phi_{\nu}({\bfa x},t) \right) \;{\rm d}s \nonumber \\
          = & \sum_{p=1}^{N_c} \int_{\Gamma_h(t)}  \hat{r}^{(p)}({\bfa c}_h, {\bfa c}_{s,h}, {\bfa l}_{s,h}) \;{\rm d}s,
        \label{eq:conserv-sd-bulk-eq2}
        \end{align}
        where we have applied \eqref{eq:sd-pum} and \eqref{eq:conserv-sd-sources}. Using similar arguments, it can be shown that
        \begin{subequations}
        \begin{align}
        \label{eq:conserv-sd-surf-eq2}
          & \sum_{q=1}^{N_{c_s}} {\bfa e}_{s}^{T}{\bfa R}_{s}^{(q)}(t, {\cal C}(t), {\cal C}_s(t), {\cal L}_{s}(t)) = \sum_{q=1}^{N_{c_s}} \int_{\Gamma_h(t)} \hat{r}_s^{(q)}({\bfa c}_h, {\bfa c}_{s,h}, {\bfa l}_{s,h}) \;{\rm d}s, \\
        \label{eq:conserv-sd-bulk-eq3}          
          & \sum_{p=1}^{N_c} {\bfa e}^{T}{\bfa F}_{c}^{(p)}(t, {\cal C}(t)) = \sum_{p=1}^{N_c} \int_{\Omega_{h}(t)} f_c^{(p)}({\bfa c}_{h}) \;{\rm d}{\bfa x} = 0, \\
        \label{eq:conserv-sd-surf-eq3}
          & \sum_{q=1}^{N_{c_s}} {\bfa e}_{s}^{T}{\bfa F}_{c_s}^{(q)}(t, {\cal C}_{s}(t)) = \sum_{q=1}^{N_{c_s}} \int_{\Gamma_{h}(t)} f_{c_s}^{(q)}({\bfa c}_{s,h}) \;{\rm d}s = 0,
        \end{align}
        \label{eq:conserv-sd-sourcebcs}
        \end{subequations}
        where, once again, we have applied \eqref{eq:sd-pum} and \eqref{eq:conserv-sd-sources}. The result then follows immediately from adding \eqref{eq:conserv-sd-bulk-eq1} and \eqref{eq:conserv-sd-surf-eq1}, then applying \eqref{eq:conserv-sd-bulk-eq2}, \eqref{eq:conserv-sd-sourcebcs} and \eqref{eq:conserv-sd-interact}.

      \end{proof}

    \subsection{Fully discrete}
    \label{sec:conserv-fd}
      Following the semi-discrete case \S\ref{sec:conserv-sd}, in order to establish global conservation for the fully discrete system, we require the following assumption:
      
      \begin{assump}
      \label{ass:fd}
		For any time level $t^{n}$, $n = 0, \ldots, N_T$, we assume that on the boundary $\Gamma_h(t^{n}) = \Gamma_h^n$, the normal components of the ALE velocity and material velocity are identical; that is, ${\bfa w}_h^n \cdot {\bfa n}_h^n = {\bfa u}_{\Gamma}^n \cdot {\bfa n}_h^n$ where ${\bfa n}_h^n$ denotes the outward unit normal to $\Gamma_h^n$.
      \end{assump}
      
      Observe that Lemma \ref{lm:sd-stiffness} and Lemma \ref{lm:sd-advection} hold for any fixed time $t$ and therefore, hold at any time level $t^{n}$, $n = 0,\ldots,N_T$. We are now in a position to prove global conservation for the fully discrete system \eqref{eq:tdisc-int-pred}, \eqref{eq:tdisc-int-bulk} and \eqref{eq:tdisc-int-corr}.
      
      \begin{theorem}
      \label{thm:conserv-fd}
		Under the conditions of Assumption \ref{ass:fd} and provided
		\begin{equation}
		\label{eq:conserv-fd-sources}
		\sum_{p = 1}^{N_c} \int_{\Omega_h^n} f_c^{(p),n}({\bfa c}_h^n) \;{\rm d}{\bfa x} = 0, \quad \sum_{p = 1}^{N_c} \int_{\Gamma_h^n} r^{(p),n}({\bfa c}_h^n) \;{\rm d}s = 0, \quad \sum_{q = 1}^{N_{c_s}} \int_{\Gamma_h^n} f_{c_s}^{(q),n}({\bfa c}_{s,h}^n) \;{\rm d}s = 0,
		\end{equation}
		and
		\begin{equation}
		\label{eq:conserv-fd-interact}
		\sum_{p = 1}^{N_c} \int_{\Gamma_h^n} \hat{r}^{(p),n}({\bfa c}_h^n, {\bfa c}_{s,h}^n, {\bfa l}_{s,h}^n) \;{\rm d}s = \sum_{q = 1}^{N_{c_s}} \int_{\Gamma_h^n} \hat{r}_s^{(q),n}({\bfa c}_h^n, {\bfa c}_{s,h}^n, {\bfa l}_{s,h}^n) \;{\rm d}s,
		\end{equation}
		are satisfied for any time level $t^{n}$ $(n = 0,\ldots,N_T)$, then the fully discrete system given by \eqref{eq:tdisc-int-pred}, \eqref{eq:tdisc-int-bulk} and \eqref{eq:tdisc-int-corr} is globally conservative independent of the ALE velocity, the predicted solution $\widetilde{{\cal C}}_s(t)$ and time step size $\Delta t$; that is, for any time level $t^{n}$ $(n = 0,\ldots,N_T)$, we have
		\[
	  	  \sum_{p=1}^{N_c} {\bfa e}^{T}M_{\Omega}^{n+1}{\bfa C}^{(p),n+1} + \sum_{q=1}^{N_{c_s}} {\bfa e}_{s}^{T}M_{s}^{n+1}{\bfa C}_{s}^{(q),n+1} = \sum_{p=1}^{N_c} {\bfa e}^{T}M_{\Omega}^{n}{\bfa C}^{(p),n} + \sum_{q=1}^{N_{c_s}} {\bfa e}_{s}^{T}M_{s}^{n}{\bfa C}_{s}^{(q),n}.
		\]
      \end{theorem}
      \begin{proof}
        Similar to the proof of Theorem \ref{thm:conserv-sd}, from Assumption \ref{ass:fd} the matrix $A_{\Gamma}^{n}({\bfa w}_{h}^{n}; {\bfa u}_{\Gamma}^{n}, {\bfa n}_{\Omega}^n) = 0$ at all time levels $t^{n}$, $n = 0,\ldots,N_{T}$. We take column sums of \eqref{eq:tdisc-int-pred}, \eqref{eq:tdisc-int-bulk} and \eqref{eq:tdisc-int-corr}; apply Lemmas \ref{lm:sd-stiffness} and \ref{lm:sd-advection}; and sum over all $p = 1,\ldots,N_c$ and $q = 1,\ldots,N_{c_s}$, to give
        \begin{align}
        \label{eq:conserv-fd-pred1}
          \sum_{q=1}^{N_{c_s}} {\bfa e}_{s}^{T}M_{s}^{n+1} \widetilde{{\bfa C}}_{s}^{(q),n+1} = & \; \sum_{q=1}^{N_{c_s}} {\bfa e}_{s}^{T}M_s^n {\bfa C}_s^{(q),n} - \Delta t \sum_{q=1}^{N_{c_s}} \left[ {\bfa e}_{s}^{T}{\bfa R}_{s}^{(q),n}({\cal C}^n, {\cal C}_s^n, {\cal L}_{s}^n) - {\bfa e}_{s}^{T}{\bfa F}_{c_s}^{(q),n}({\cal C}_{s}^n) \right], \\
        \label{eq:conserv-fd-bulk1}
          \sum_{p=1}^{N_c} {\bfa e}^{T}M_{\Omega}^{n+1} {\bfa C}^{(p),n+1} = & \; \sum_{p=1}^{N_c} {\bfa e}^{T}M_{\Omega}^{n} {\bfa C}^{(p),n} + \frac{\Delta t}{2} \sum_{p=1}^{N_c} \left[ {\bfa e}^{T}{\bfa F}_{c}^{(p),n+1}({\cal C}^{n+1}) + {\bfa e}^{T}{\bfa F}_{c}^{(p),n}({\cal C}^{n}) \right] \;+ \nonumber \\
          & \;\frac{\Delta t}{2} \sum_{p=1}^{N_c} \left[ {\bfa e}^{T}{\bfa R}^{(p),n+1}({\cal C}^{n+1}, \widetilde{{\cal C}}_{s}^{n+1}, \widetilde{{\cal L}}_s^{n+1}) + {\bfa e}^{T}{\bfa R}^{(p),n}({\cal C}^{n}, {\cal C}_{s}^{n}, {\cal L}_s^{n}) \right], \\
        \label{eq:conserv-fd-corr1}
          \sum_{q=1}^{N_{c_s}} {\bfa e}_{s}^{T}M_{s}^{n+1} {\bfa C}_{s}^{(q),n+1} = & \; \sum_{q=1}^{N_{c_s}} {\bfa e}_{s}^{T}M_{s}^{n} {\bfa C}_s^{(q),n} + \frac{\Delta t}{2} \sum_{q=1}^{N_{c_s}} \left[ {\bfa e}_{s}^{T}{\bfa F}_{c_s}^{(q),n+1}(\widetilde{{\cal C}}_{s}^{n+1}) + {\bfa e}_{s}^{T}{\bfa F}_{c_s}^{(q),n}({\cal C}_{s}^{n}) \right] \;- \nonumber \\
          & \;\frac{\Delta t}{2} \sum_{q=1}^{N_{c_s}} \left[ {\bfa e}_{s}^{T}{\bfa R}_{s}^{(q),n+1}({\cal C}^{n+1}, \widetilde{{\cal C}}_s^{n+1}, \widetilde{{\cal L}}_{s}^{n+1}) + {\bfa e}_{s}^{T}{\bfa R}_{s}^{(q),n}({\cal C}^n, {\cal C}_s^n, {\cal L}_{s}^n) \right].
        \end{align}
        Similar to the proof of Theorem \ref{thm:conserv-sd}, it can be shown that for any fixed time level $t^n$, $n = 0,\ldots,N_T$, we have
        \begin{subequations}
        \begin{align}
        \label{eq:conserv-fd-bulk-R}
          & \sum_{p=1}^{N_c}{\bfa e}^{T}{\bfa R}^{(p),n}({\cal C}^n, {\cal C}_{s}^n, {\cal L}_s^n) = \sum_{p=1}^{N_c} \int_{\Gamma_h^n}  \hat{r}^{(p),n}({\bfa c}_h^n, {\bfa c}_{s,h}^n, {\bfa l}_{s,h}^n) \;{\rm d}s, \\
        \label{eq:conserv-fd-surf-Rs}
          & \sum_{q=1}^{N_{c_s}} {\bfa e}_{s}^{T}{\bfa R}_{s}^{(q),n}({\cal C}^n, {\cal C}_s^n, {\cal L}_{s}^n) = \sum_{q=1}^{N_{c_s}} \int_{\Gamma_h^n} \hat{r}_s^{(q),n}({\bfa c}_h^n, {\bfa c}_{s,h}^n, {\bfa l}_{s,h}^n) \;{\rm d}s, \\
        \label{eq:conserv-fd-bulk-F}          
          & \sum_{p=1}^{N_c} {\bfa e}^{T}{\bfa F}_{c}^{(p),n}({\cal C}^n) = \sum_{p=1}^{N_c} \int_{\Omega_{h}^n} f_c^{(p),n}({\bfa c}_{h}^n) \;{\rm d}{\bfa x} = 0, \\
        \label{eq:conserv-fd-surf-Fs}
          & \sum_{q=1}^{N_{c_s}} {\bfa e}_{s}^{T}{\bfa F}_{c_s}^{(q),n}({\cal C}_{s}^n) = \sum_{q=1}^{N_{c_s}} \int_{\Gamma_{h}^n} f_{c_s}^{(q),n}({\bfa c}_{s,h}^n) \;{\rm d}s = 0,
        \end{align}
        \label{eq:conserv-fd-sourcebcs}
        \end{subequations}
        where \eqref{eq:conserv-fd-sources} has been used. Note that as a consequence of \eqref{eq:conserv-fd-interact}, we have, for any fixed time level $t^{n}$, $n = 0,\ldots,N_T$, that $\sum_{p} {\bfa e}^{T}{\bfa R}^{(p),n}({\cal C}^n, {\cal C}_{s}^n, {\cal L}_s^n) = \sum_{q} {\bfa e}_{s}^{T}{\bfa R}_{s}^{(q),n}({\cal C}^n, {\cal C}_s^n, {\cal L}_{s}^n)$. Adding \eqref{eq:conserv-fd-pred1} and \eqref{eq:conserv-fd-bulk1}, then applying \eqref{eq:conserv-fd-sourcebcs} and \eqref{eq:conserv-fd-interact} gives
        \begin{align}
          & \sum_{q=1}^{N_{c_s}} {\bfa e}_{s}^{T}M_{s}^{n+1} \widetilde{{\bfa C}}_{s}^{(q),n+1} + \sum_{p=1}^{N_c} {\bfa e}^{T}M_{\Omega}^{n+1} {\bfa C}^{(p),n+1} = \sum_{q=1}^{N_{c_s}} {\bfa e}_{s}^{T}M_s^n {\bfa C}_s^{(q),n} + \sum_{p=1}^{N_c} {\bfa e}^{T}M_{\Omega}^{n} {\bfa C}^{(p),n} \;+ \nonumber \\
          & \hspace{1cm} \frac{\Delta t}{2} \sum_{p=1}^{N_c} {\bfa e}^{T}{\bfa R}^{(p),n+1}({\cal C}^{n+1}, \widetilde{{\cal C}}_{s}^{n+1}, \widetilde{{\cal L}}_s^{n+1}) - \frac{\Delta t}{2} \sum_{q=1}^{N_{c_s}} {\bfa e}_{s}^{T}{\bfa R}_{s}^{(q),n}({\cal C}^n, {\cal C}_s^n, {\cal L}_{s}^n).
        \label{eq:conserv-fd-predbulk-sum}
        \end{align}
        Adding \eqref{eq:conserv-fd-corr1} and \eqref{eq:conserv-fd-predbulk-sum}, then applying \eqref{eq:conserv-fd-sourcebcs} and \eqref{eq:conserv-fd-interact} gives
        \begin{align}
          & \sum_{q=1}^{N_{c_s}} {\bfa e}_{s}^{T}M_{s}^{n+1} {\bfa C}_{s}^{(q),n+1} + \sum_{p=1}^{N_c} {\bfa e}^{T}M_{\Omega}^{n+1} {\bfa C}^{(p),n+1} = \sum_{q=1}^{N_{c_s}} {\bfa e}_{s}^{T}M_{s}^{n} {\bfa C}_s^{(q),n} + \sum_{p=1}^{N_c} {\bfa e}^{T}M_{\Omega}^{n} {\bfa C}^{(p),n} \;- \nonumber \\
          & \hspace{1cm} \underbrace{ \left[ \sum_{q=1}^{N_{c_s}} {\bfa e}_{s}^{T}M_{s}^{n+1} \widetilde{{\bfa C}}_{s}^{(q),n+1} - \sum_{q=1}^{N_{c_s}} {\bfa e}_{s}^{T}M_s^n {\bfa C}_s^{(q),n} + \Delta t \sum_{q=1}^{N_{c_s}} {\bfa e}_{s}^{T}{\bfa R}_{s}^{(q),n}({\cal C}^n, {\cal C}_s^n, {\cal L}_{s}^n) \right] }_{\displaystyle = 0},
        \label{eq:conserv-fd-predbulkcorr-sum}
        \end{align}
        where, from \eqref{eq:conserv-fd-pred1} and \eqref{eq:conserv-fd-surf-Fs}, the term in the square brackets in (\ref{eq:conserv-fd-predbulkcorr-sum}) is zero. Therefore, we can deduce that the predicted solution $\widetilde{{\cal C}}_{s}(t)$, ALE velocity and time step size $\Delta t$ play no role in global conservation. It is clear that \eqref{eq:conserv-fd-predbulkcorr-sum} holds for any $n = 0,\ldots,N_{T}$ and thus, the fully discrete system is globally conservative, as required.
        
      \end{proof}
      
      \begin{remark}
      \label{rem:conserv-fd-imex}
      	As mentioned in Remark \ref{rem:tdisc-imex}, the results presented in \S\ref{sec:nexp-cc} consider interior bulk species only where we employ an implicit-explicit backward Euler scheme for the temporal discretisation. Therefore, \eqref{eq:tdisc-int-bulk} can be written
      	\begin{align}
      		& \left\{ M_{\Omega}^{n+1} + \Delta t \left[ B_{\Omega}^{n+1}({\bfa w}_{h}^{n+1}; {\bfa u}_{\Omega}^{n+1}) - A_{\Gamma}^{n+1}({\bfa w}_{h}^{n+1}; {\bfa u}_{\Gamma}^{n+1}, {\bfa n}_{\Omega}^{n+1}) + K_{\Omega}^{n+1}(D_c^{(p)}) \right] \right\}{\bfa C}^{(p),n+1} \;= \nonumber \\
      		& \hspace{1cm} M_{\Omega}^{n}{\bfa C}^{(p),n} + \Delta t \left[ {\bfa R}^{(p),n}({\cal C}^{n}, {\cal C}_{s}^{n}, {\cal L}_s^{n}) + {\bfa F}_{c}^{(p),n}({\cal C}^{n}) \right].
    		\label{eq:conserv-fd-imex-bulk}
    		\end{align}
    		Global conservation of \eqref{eq:conserv-fd-imex-bulk} can then be obtained by following the steps in the proof of Theorem \ref{thm:conserv-fd}: apply Assumption \ref{ass:fd}; take column sums of \eqref{eq:conserv-fd-imex-bulk}; apply Lemmas \ref{lm:sd-stiffness} and \ref{lm:sd-advection}; sum over $p = 1,\ldots,N_c$; and, finally, apply \eqref{eq:conserv-fd-sources}.
      \end{remark}

  \section{Numerical Experiments}
  \label{sec:nexp}
  
  In the examples that follow, the conservation error is defined as the absolute difference between the total concentration at time levels $t^{n+1}$ and $t^{n}$; that is, following Theorem \ref{thm:conserv-fd}, the conservation error is given by
  \begin{align}
  	\left| \left(\sum_{p=1}^{N_c} {\bfa e}^{T}M_{\Omega}^{n+1}{\bfa C}^{(p),n+1} + \sum_{q=1}^{N_{c_s}} {\bfa e}_{s}^{T}M_{s}^{n+1}{\bfa C}_{s}^{(q),n+1}\right) - \left(\sum_{p=1}^{N_c} {\bfa e}^{T}M_{\Omega}^{n}{\bfa C}^{(p),n} + \sum_{q=1}^{N_{c_s}} {\bfa e}_{s}^{T}M_{s}^n{\bfa C}_{s}^{(q),n}\right) \right|. \nonumber
  \end{align}
  For the convergence studies presented in \S\ref{sec:nexp-cc-diffonly} and \S\ref{sec:nexp-cc-advecdiff}, the $L^{2}$ and $L^{\infty}$ norms are calculated as follows: for each $K \in {\cal T}_h^{{\cal D}}(t)$, we calculate the $N_g$ Gaussian quadrature points $\{{\bfa x}_a\}_{a = 1}^{N_g}$; their associated weights $\{w_a\}_{a=1}^{N_g}$; and the error between the computational and exact solutions $\{e_a\}_{a=1}^{N_g}$, such that 
  \[
  	e_a := c_h^{(1)}({\bfa x}_a,t) - c^{(1)}({\bfa x}_a,t), \qquad a = 1,\ldots,N_g.
  \]
  The $L^{2}$ norm is then approximated using Gaussian quadrature
  \begin{align}
  	\| e \|_{L^{2}} \approx \left( \sum_{K \in {\cal T}_h^{{\cal D}}(t)} \sum_{a = 1}^{N_g} w_a | e_a |^2 \right)^{1/2}, \nonumber
  \end{align}
  and the $L^{\infty}$ norm is approximated by
  \begin{align}
  	\| e \|_{L^{\infty}} \approx \max_{K \in {\cal T}_h^{{\cal D}}(t)} \left( \max_{a = 1,\ldots,N_g} | e_a | \right). \nonumber
  \end{align}

    \subsection{Convergence and conservation}
    \label{sec:nexp-cc}
    
      \subsubsection{Diffusion in a moving unit circle}
      \label{sec:nexp-cc-diffonly}
        In this example, we do not consider an exterior environment. Therefore, $N_l = 0$ and $N_{l_s} = 0$. We consider diffusion of a single interior bulk species $(N_c = 1)$ and no surface species $(N_{c_s} = 0)$ in a moving circular domain. We also assume that the bulk reaction term in \eqref{eq:rd-int-bulk-eq} is zero and employ zero-flux boundary conditions in \eqref{eq:rd-int-bulk-bcs}; that is,
        \begin{equation}
        		f_c^{(1)}({\bfa c}) = f_c^{(1)}(c^{(1)}) = 0, \quad r^{(1)}({\bfa c}) = r^{(1)}(c^{(1)}) = 0, \quad \mbox{and} \quad \hat{r}^{(1)}({\bfa c}, {\bfa c}_s, {\bfa l}_s) = 0. \nonumber
        \end{equation}
        Following Novak and Slepchenko \cite{novak_2014}, we assume that the bulk domain $\Omega(t)$ is a unit circle moving with a constant velocity ${\bfa u}_{b} = (1,0)$. We also assume that the boundary $\Gamma(t)$ evolves with the same constant velocity ${\bfa u}_{\Gamma} = {\bfa u}_{b}$. However, the material velocity of the bulk domain ${\bfa u}_{\Omega} = {\bfa 0}$. With the diffusion coefficient $D_c^{(1)} = 1$,  this problem has an exact solution given by 
	\begin{equation}
	  c^{(1)}(x,y,t) = c_{0}\exp{ \left[ -(x - t) \right] }, \quad\quad ({\bfa x}, t) \in \Omega(t)\times \overline{I}, \nonumber
	\end{equation}
	where $c_{0} \in \RR$. In \cite{novak_2014}, the constant value $c_{0}$ is calculated on the stationary box
	\[
	  \Omega = \left\{ (x,y) \colon |x| \leq 1, |y| \leq 1 \right\},  
	\]
	and is given by
	\[
	  c_{0} = 4\left[ \iint_{\Omega} \exp{(-x)}\; {\rm d}x {\rm d}y \right]^{-1} = \left[ \sinh(1) \right]^{-1}.
	\]
	For the moving circular domain, a value of $c_{0}$ is not given in \cite{novak_2014} and therefore, here we assume the same value for $c_{0}$ as for the stationary box. In the results which follow, we set the final time $T = 0.5$, the time step size $\Delta t = 5\times 10^{-4}$ (so that $N_T = 1000$) and initial condition $c^{(1)}(x,y,0) = c_{0}\exp{(-x)}$. The initial mesh contains $N_e^{\Omega} = 1499$ triangles and ${\cal N}_s^{\Gamma} = 87$ boundary nodes. The boundary nodes are translated with the velocity ${\bfa u}_{b}$ and these are then employed as fixed Dirichlet boundary conditions for the MMPDE system with monitor matrix $G = {\bb I}$, where ${\bb I}$ is the identity matrix. Figures \ref{fig:nexp-cc-diff}\subref{fig:nexp-cc-diff-sol} and \ref{fig:nexp-cc-diff}\subref{fig:nexp-cc-diff-cut} show the computed solution at $T=0.5$ and a comparison between the computed and exact solution along the cut $y = 0$, respectively. The computed solution is clearly in very good agreement with the exact solution. Figure \ref{fig:nexp-cc-diff}\subref{fig:nexp-cc-diff-conserv} confirms that the fully discrete numerical solution is globally conservative to machine zero. The rates of spatial convergence in the $L^{2}$ and $L^{\infty}$ norms are shown in Fig. \ref{fig:nexp-cc-diff}\subref{fig:nexp-cc-diff-L2}. Clearly second-order convergence is demonstrated in both norms. The rate of convergence is an improvement over the convergence rates obtained using the method proposed in \cite{novak_2014}, where the convergence rates were between $1$ and $2$. In \cite{novak_2014}, the impaired rates of convergence are attributed to a reinitialisation procedure of the unknowns at the boundary. Due to the fitted nature of the evolving mesh considered in this article, we do not require such a procedure and therefore, we obtain the expected rates of convergence.
	
	\begin{figure}[ht!]
        \centering
	
	  \subfloat[Computed solution and initial position of domain.]{%
	    \label{fig:nexp-cc-diff-sol}%
	    \inc{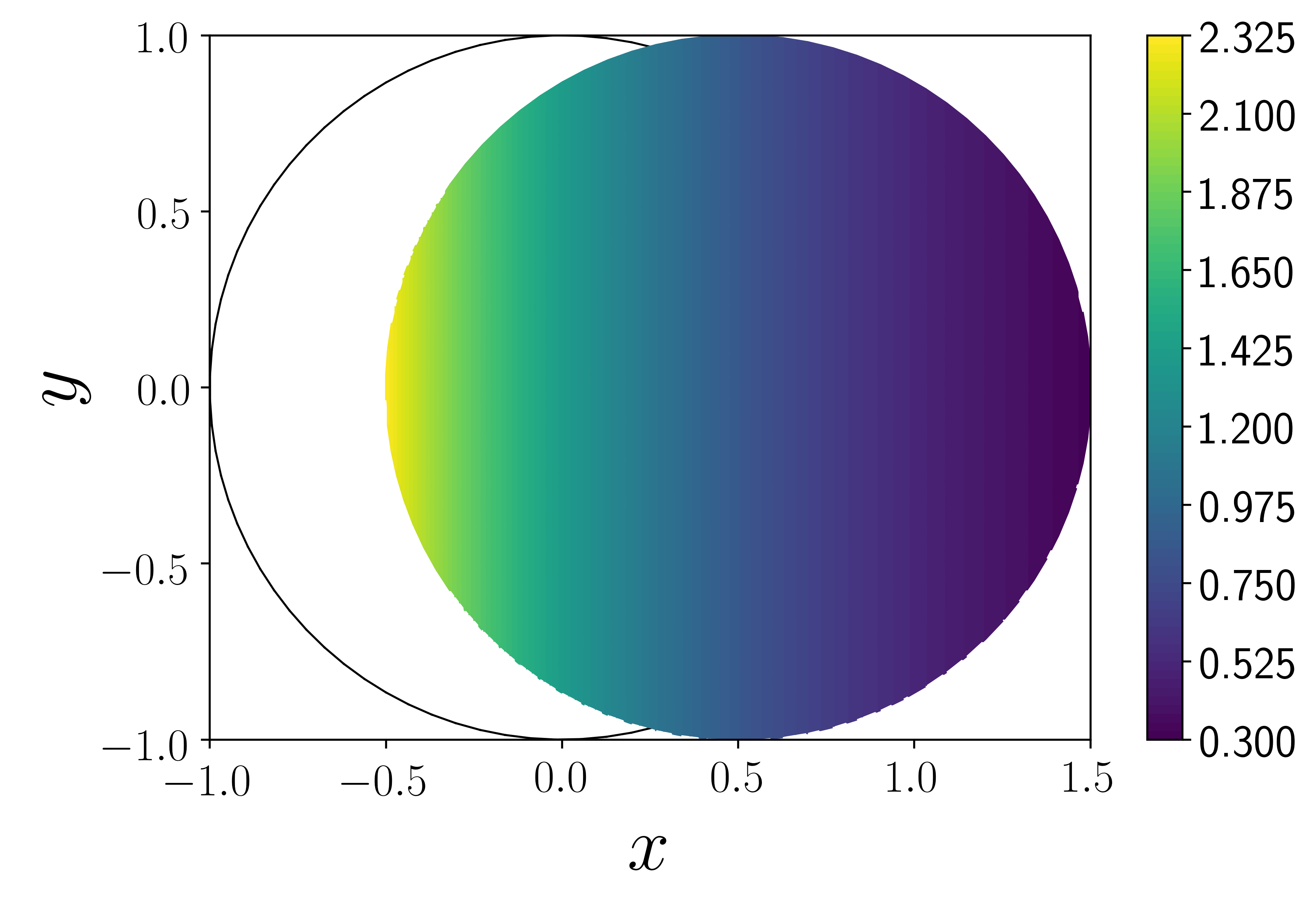}{0.52\textwidth}}\quad	    
	  \subfloat[Numerical and exact solutions along $y = 0$.]{%
	    \label{fig:nexp-cc-diff-cut}%
	    \inc{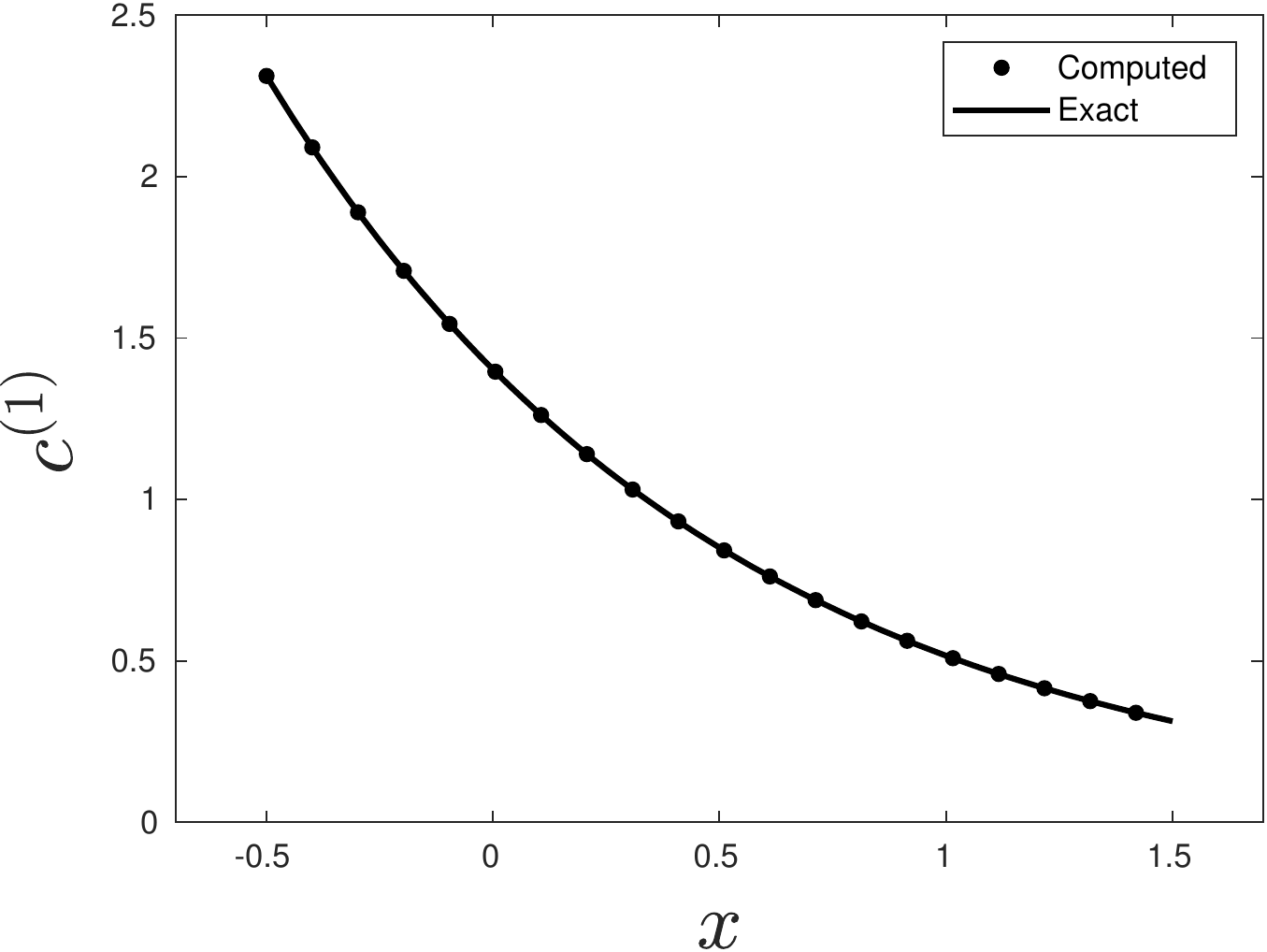}{0.44\textwidth}}\\
	  \subfloat[Global fully discrete conservation.]{%
	    \label{fig:nexp-cc-diff-conserv}%
	    \inc{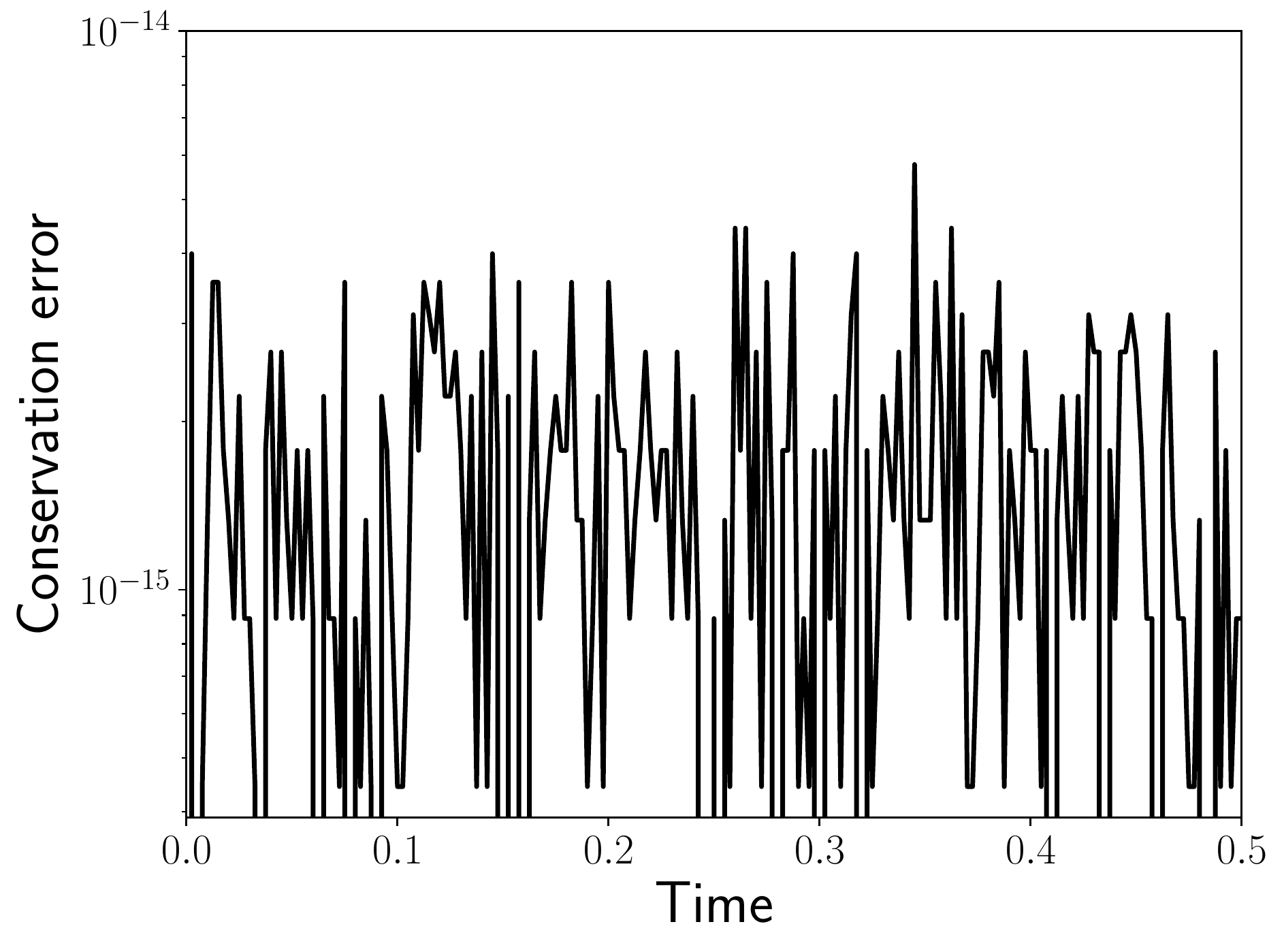}{0.53\textwidth}}\quad
	  \subfloat[$L^{2}$ and $L^{\infty}$ convergence.]{%
	    \label{fig:nexp-cc-diff-L2}%
	    \inc{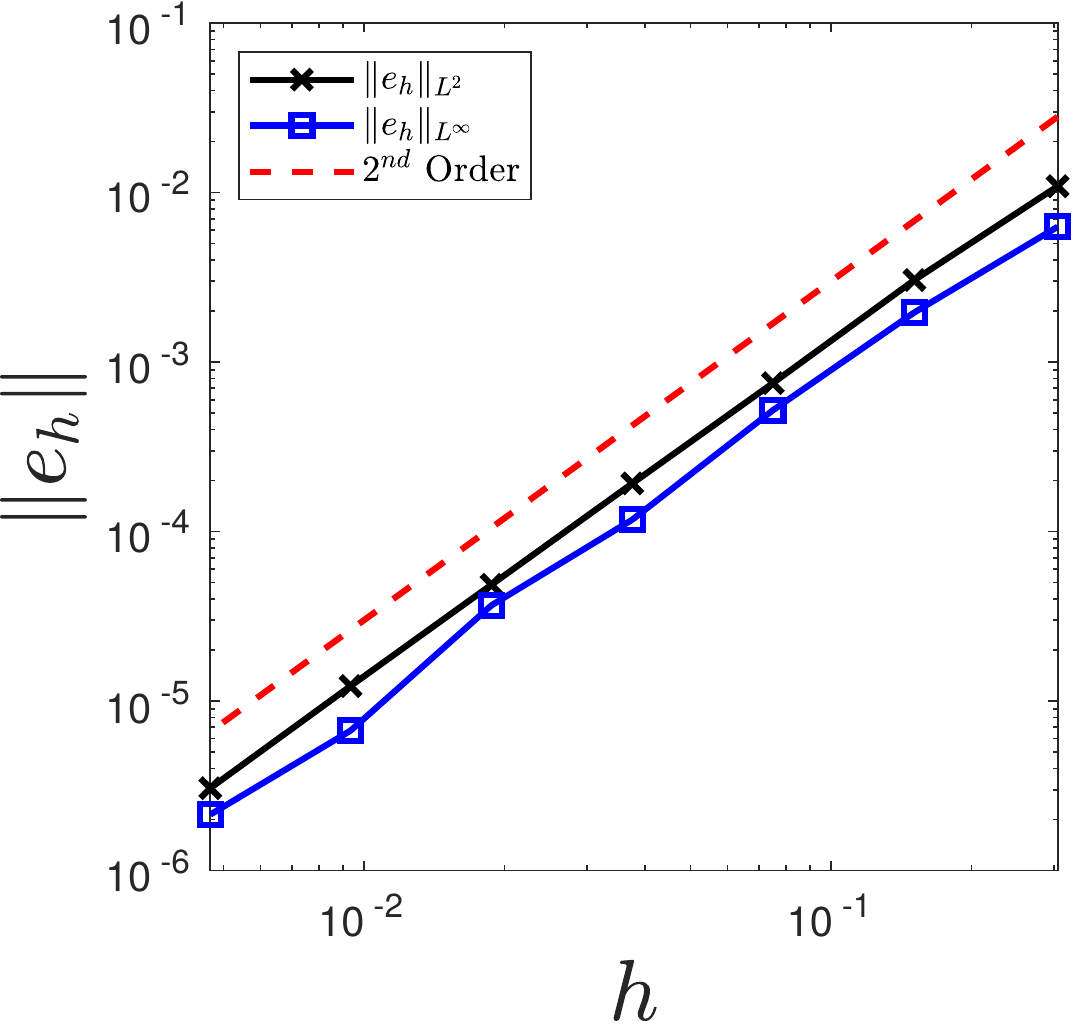}{0.43\textwidth}}
	  \caption[]{%
	    {\em (Diffusion in a moving unit circle)} \subref{fig:nexp-cc-diff-sol} Contour plot of the computed solution at $T = 0.5$ (initial configuration is given by the solid black line); \subref{fig:nexp-cc-diff-cut} Comparison of the computed and exact solution along a cut at $y = 0$; \subref{fig:nexp-cc-diff-conserv} Conservation error as a function of time; \subref{fig:nexp-cc-diff-L2} Spatial convergence in the $L^{2}$ and $L^{\infty}$ norms.}
	  \label{fig:nexp-cc-diff}
	\end{figure}

      \subsubsection{Advection-diffusion in a moving unit circle}
      \label{sec:nexp-cc-advecdiff}
        As an extension to the previous example, we consider the advection-diffusion of a chemical species in a moving circular domain. Once again, this problem has been considered by Novak and Slepchenko \cite{novak_2014} and has an exact solution. We consider only a single interior bulk species $(N_c = 1)$, no surface species $(N_{c_s} = 0)$ and no exterior environment so that $N_l = 0$ and $N_{l_s} = 0$. We also assume that the bulk reaction term in \eqref{eq:rd-int-bulk-eq} is zero and employ zero-flux boundary conditions in \eqref{eq:rd-int-bulk-bcs}, as before.
        
        Once again, we assume that the bulk domain $\Omega(t)$ is a unit circle moving with a constant velocity ${\bfa u}_{b} = (1,0)$ and that the boundary $\Gamma(t)$ evolves with the same constant velocity ${\bfa u}_{\Gamma} = {\bfa u}_{b}$. However, contrary to the previous example, the material velocity of the bulk domain ${\bfa u}_{\Omega} = {\bfa u}_{b}$. The exact solution is then given by \cite{novak_2014}
        \begin{equation}
          c^{(1)}(x,y,t) = \exp{ \left(-\lambda^{2}D_c^{(1)}t\right) }\left[ \frac{x - t}{ |{\bfa x} - {\bfa u}_{b}t| } \right] J_{1}\left(\lambda |{\bfa x} - {\bfa u}_{b}t| \right) + J_{1}(\lambda), \nonumber
        \end{equation}
        where $\lambda = 1.841183781340659$, $J_{1}(z)$ is the first-order Bessel function of the first kind and $|{\bfa x}|$ denotes the $l_{2}$ norm. In the results which follow, we set the diffusion coefficient $D_c^{(1)} = 1/4$, the final time $T = 0.2$, time step size $\Delta t = 2\times 10^{-5}$ (so that $N_T = 10000$) and initial condition
        \[
          c^{(1)}(x,y,0) = \frac{x}{|{\bfa x}|} J_{1}\left(\lambda |{\bfa x}| \right) + J_{1}(\lambda).
        \]
        The same meshes were used as in the previous example. Figures \ref{fig:nexp-cc-advecdiff}\subref{fig:nexp-cc-advecdiff-sol} and \ref{fig:nexp-cc-advecdiff}\subref{fig:nexp-cc-advecdiff-cut} illustrate the computed solution at the end of the simulation and a comparison between the computed and exact solution along the cut $y = 0$, respectively. The computed solution clearly is in very good agreement with the exact solution. Figure \ref{fig:nexp-cc-advecdiff}\subref{fig:nexp-cc-advecdiff-L2} illustrates the spatial convergence in the $L^{2}$ and $L^{\infty}$ norms. Clearly second-order convergence can be seen in both norms. This is an improvement over the convergence rates shown in \cite{novak_2014}, where less than second-order convergence was seen for the $L^{\infty}$ norm. Finally, Fig. \ref{fig:nexp-cc-advecdiff}\subref{fig:nexp-cc-advecdiff-conserv} confirms that the fully discrete numerical solution is globally conservative to machine zero.
	
	\begin{figure}[ht!]	
	  \subfloat[Computed solution and initial position of domain.]{%
	    \label{fig:nexp-cc-advecdiff-sol}%
	    \inc{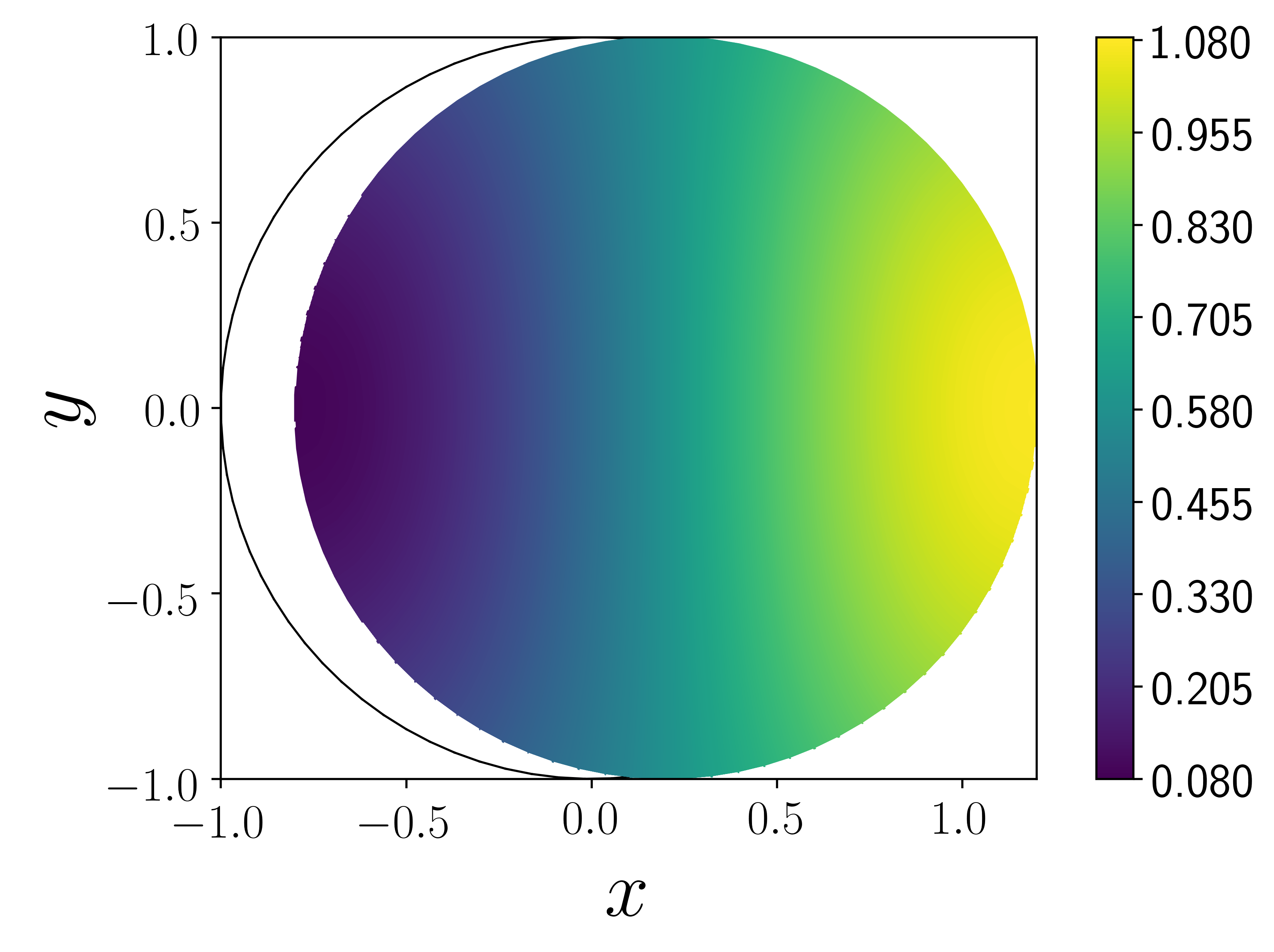}{0.49\textwidth}}\quad
	  \subfloat[Numerical and exact solution along $y = 0$.]{%
	    \label{fig:nexp-cc-advecdiff-cut}%
	    \inc{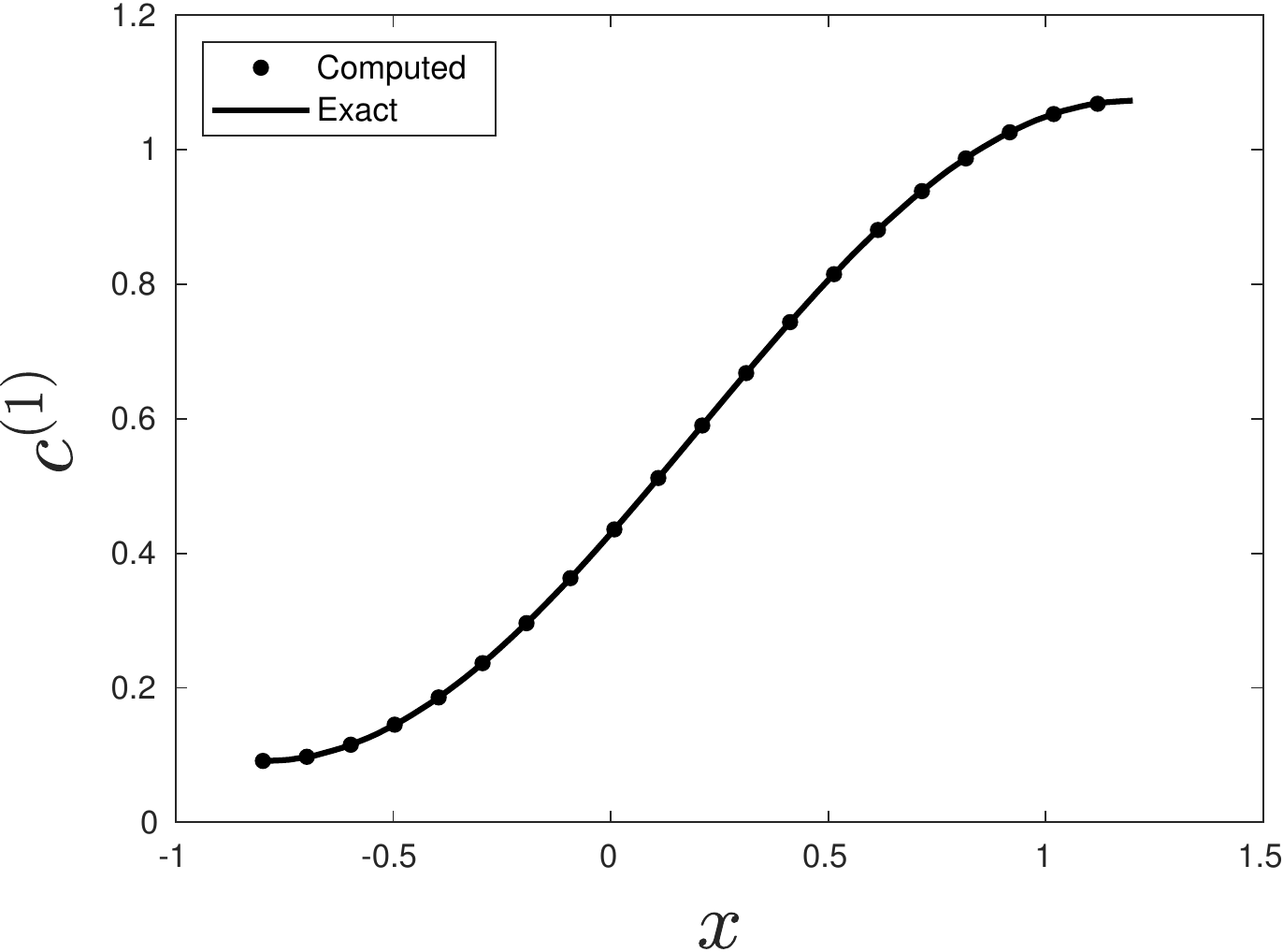}{0.46\textwidth}}\\
	  \subfloat[Global fully discrete conservation.]{%
	    \label{fig:nexp-cc-advecdiff-conserv}%
	    \inc{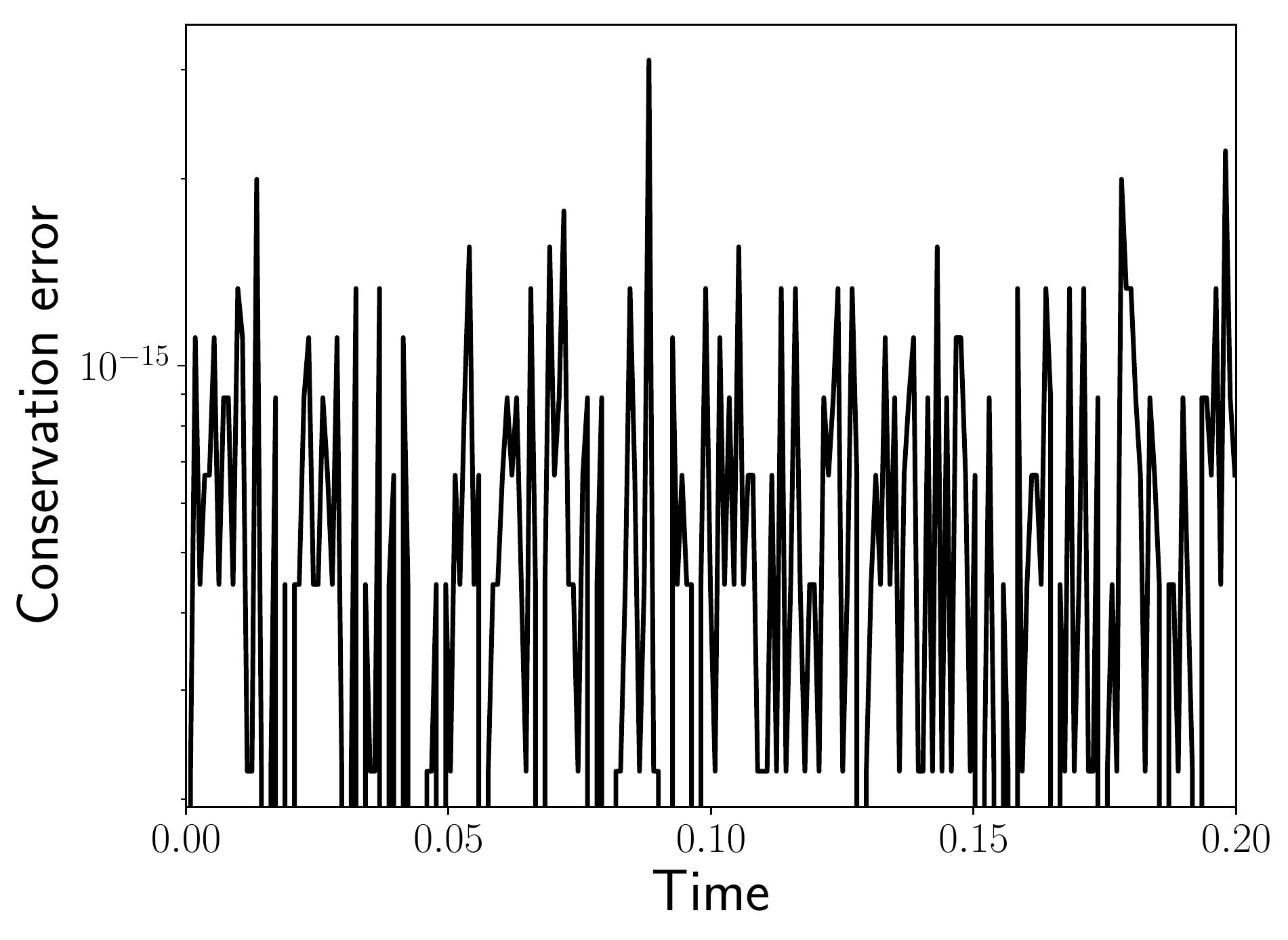}{0.53\textwidth}}\quad
	  \subfloat[$L^{2}$ and $L^{\infty}$ convergence.]{%
	    \label{fig:nexp-cc-advecdiff-L2}%
	    \inc{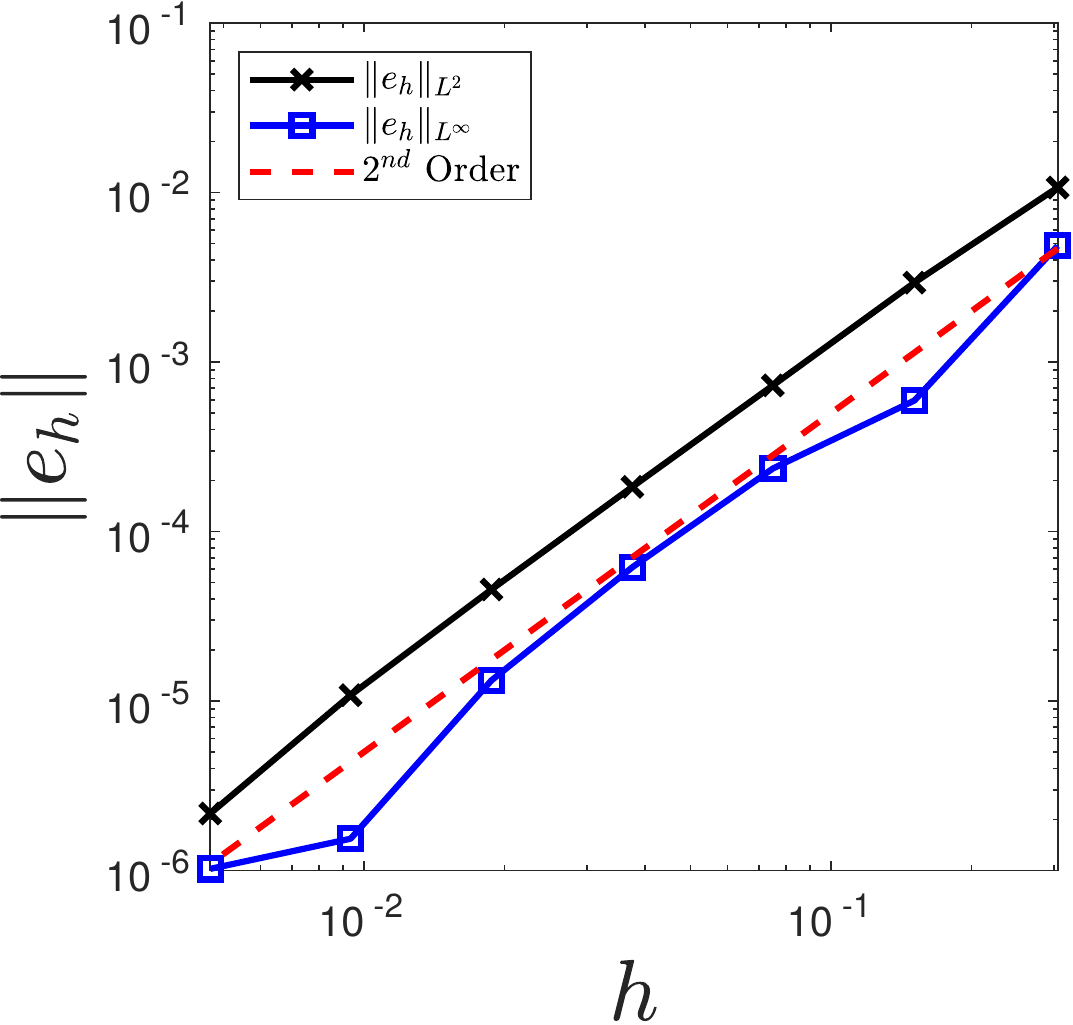}{0.43\textwidth}}
	  \caption[]{%
	    {\em (Advection-diffusion in a moving unit circle)} \subref{fig:nexp-cc-advecdiff-sol} Contour plot of the computed solution at $T = 0.2$ (initial configuration is given by the solid black line); \subref{fig:nexp-cc-advecdiff-cut} Comparison of the computed and exact solution along a cut at $y = 0$; \subref{fig:nexp-cc-advecdiff-conserv} Conservation error as a function of time; \subref{fig:nexp-cc-advecdiff-L2} Spatial convergence in the $L^{2}$ and $L^{\infty}$ norms.}
	  \label{fig:nexp-cc-advecdiff}
	\end{figure}

      \subsubsection{Advection-reaction-diffusion in a moving complex geometry}
      \label{sec:nexp-cc-stry}
        In this section, we consider an example presented by Strychalski \etal \cite{strychalski_2010-1, strychalski_2010} where the initial domain is given by
        \begin{equation}
          \Omega_{0} := \left \{ {\bfa x}=(x,y) : \left| {\bfa x} - {\bfa x}_{c} \right| \leq 0.234 - 0.0702\sin{\left ( 4 \tan^{-1}{\left( y / x \right)} \right )} \right \}, \nonumber
        \end{equation}
        and ${\bfa x}_{c} = (0.5,0.5)$. This domain is advected with a constant velocity ${\bfa u}_{b} = (0.1, 0.1)$ and therefore, the boundary $\Gamma(t)$ evolves with the same constant velocity ${\bfa u}_{\Gamma} = {\bfa u}_{b}$. Once again, we do not consider an exterior environment so that $N_l = 0$ and $N_{l_s} = 0$ and assume that there are no surface species ($N_{c_s} = 0$). However, contrary to the previous examples, we consider two interacting interior bulk species ($N_c = 2$). We also assume that the material velocity of the bulk domain ${\bfa u}_{\Omega} = {\bfa u}_{b}$. Following \cite{strychalski_2010}, the bulk reaction in \eqref{eq:rd-int-bulk-eq} and flux boundary conditions in \eqref{eq:rd-int-bulk-bcs} are given by
        \begin{alignat}{3}
        \label{eq:nexp-cca-stry1}
          & f_c^{(1)}({\bfa c}) = f_c^{(1)}(c^{(2)}) = \frac{k_{2} c^{(2)}}{K_{m2} + c^{(2)}}, & \qquad \mbox{and} & \qquad
	      r^{(1)}({\bfa c}) = r^{(1)}(c^{(1)}) = \frac{k_{1} S c^{(1)}}{K_{m1} + c^{(1)}}, \\
	    \label{eq:nexp-cca-stry2}
	      & f_c^{(2)}({\bfa c}) = f_c^{(2)}(c^{(2)}) = -\frac{k_{2} c^{(2)}}{K_{m2} + c^{(2)}}, & \qquad \mbox{and} & \qquad
	      r^{(2)}({\bfa c}) = r^{(2)}(c^{(1)}) = -\frac{k_{1} S c^{(1)}}{K_{m1} + c^{(1)}},
	    \end{alignat}
	    where $k_{1}$ and $k_{2}$ are maximum activation and deactivation rates, respectively; and $K_{m1}$ and $K_{m2}$ are Michaelis-Menten constants. It is clear that (\ref{eq:nexp-cca-stry1}) and (\ref{eq:nexp-cca-stry2}) satisfy the conditions \eqref{eq:conserv-cont-sources} of Theorem \ref{thm:conserv-cont} and therefore, in light of Remarks \ref{rem:tdisc-imex} and \ref{rem:conserv-fd-imex}, we expect this system to be globally conservative. In this example, $c^{(1)}$ and $c^{(2)}$ denote the concentrations of inactive and active protein conformations, respectively. Initially, the inactive and active concentrations are $c^{(1)}({\bfa x}, 0) = 1$ and $c^{(2)}({\bfa x}, 0) = 0$, respectively. Due to the complex geometry, this example does not possess a known exact solution. Following \cite{strychalski_2010} we set the final time $T = 0.4$, the time step size $\Delta t = 1.25\times 10^{-3}$ (so that $N_T = 320$) and choose a uniform mesh such that
	    \[
	    		h_K = 4 \Delta t, \quad \forall K \in {\cal T}_{h,c}^{{\cal D}},
	    \]
	    where $h_K$ denotes the mesh width of element $K$ of the triangulation ${\cal T}_{h,c}^{{\cal D}}$. Thus, the initial mesh contains $N_e^{\Omega} = 18659$ triangles and ${\cal N}_s^{\Gamma} = 419$ boundary nodes. Analagous to \S\ref{sec:nexp-cc-diffonly} and \S\ref{sec:nexp-cc-advecdiff}, the boundary nodes are translated with the velocity ${\bfa u}_{b}$ and then employed as fixed Dirichlet boundary conditions for the MMPDE system with monitor matrix $G = {\bb I}$, where ${\bb I}$ is the identity matrix. For illustrative purposes, a coarser mesh is depicted in Fig. \ref{fig:nexp-cc-stry}\subref{fig:nexp-cc-stry-a}. We set the parameters $D_c^{(1)} = D_c^{(2)} = 0.1$, $k_{1} = k_{2} = 1.0$, $K_{m1} = K_{m2} = 0.2$ and $S = 1.0$, as in \cite{strychalski_2010}.
	    
	    Figures \ref{fig:nexp-cc-stry}\subref{fig:nexp-cc-stry-b}-\subref{fig:nexp-cc-stry-d} illustrate the concentration of the active species at times $t = 0.15, 0.3, 0.4$, respectively. Once again, the solid black line indicates the initial configuration. To plotting accuracy our results agree with those presented in \cite{strychalski_2010}. Figure \ref{fig:nexp-cc-stry}\subref{fig:nexp-cc-stry-conserv} illustrates that the error in global conservation of the fully discrete numerical solution is machine zero. This is an improvement over the results presented in \cite{strychalski_2010}, where the conservation error was of the order $10^{-5}$ to $10^{-6}$.
	
	\begin{figure}
	  \centering
	  \subfloat[Illustration of the initial mesh.]{%
	    \label{fig:nexp-cc-stry-a}%
	    \inc{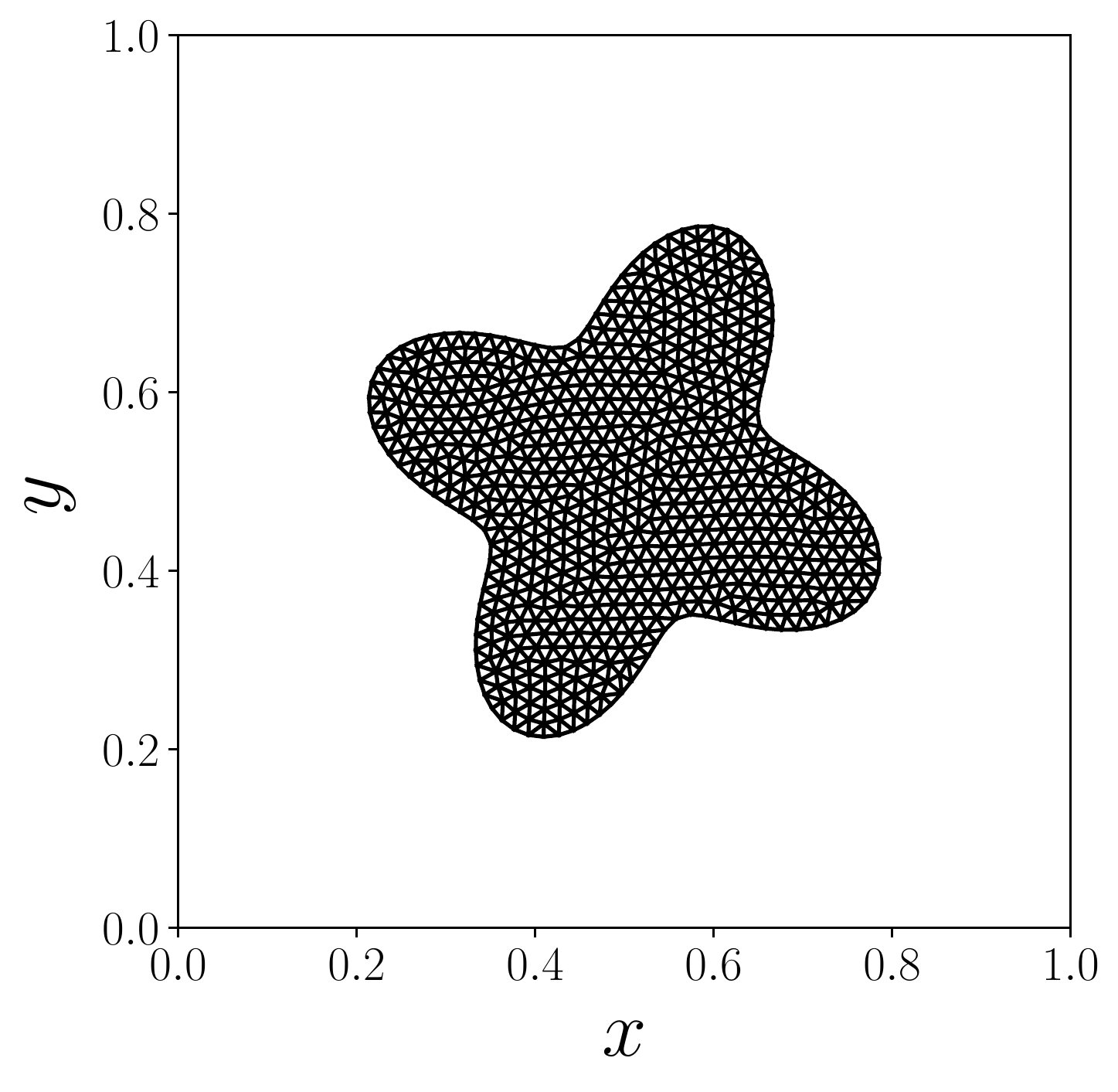}{0.42\textwidth}}\qquad\quad
	  \subfloat[Time $= 0.15$.]{%
	    \label{fig:nexp-cc-stry-b}%
	    \inc{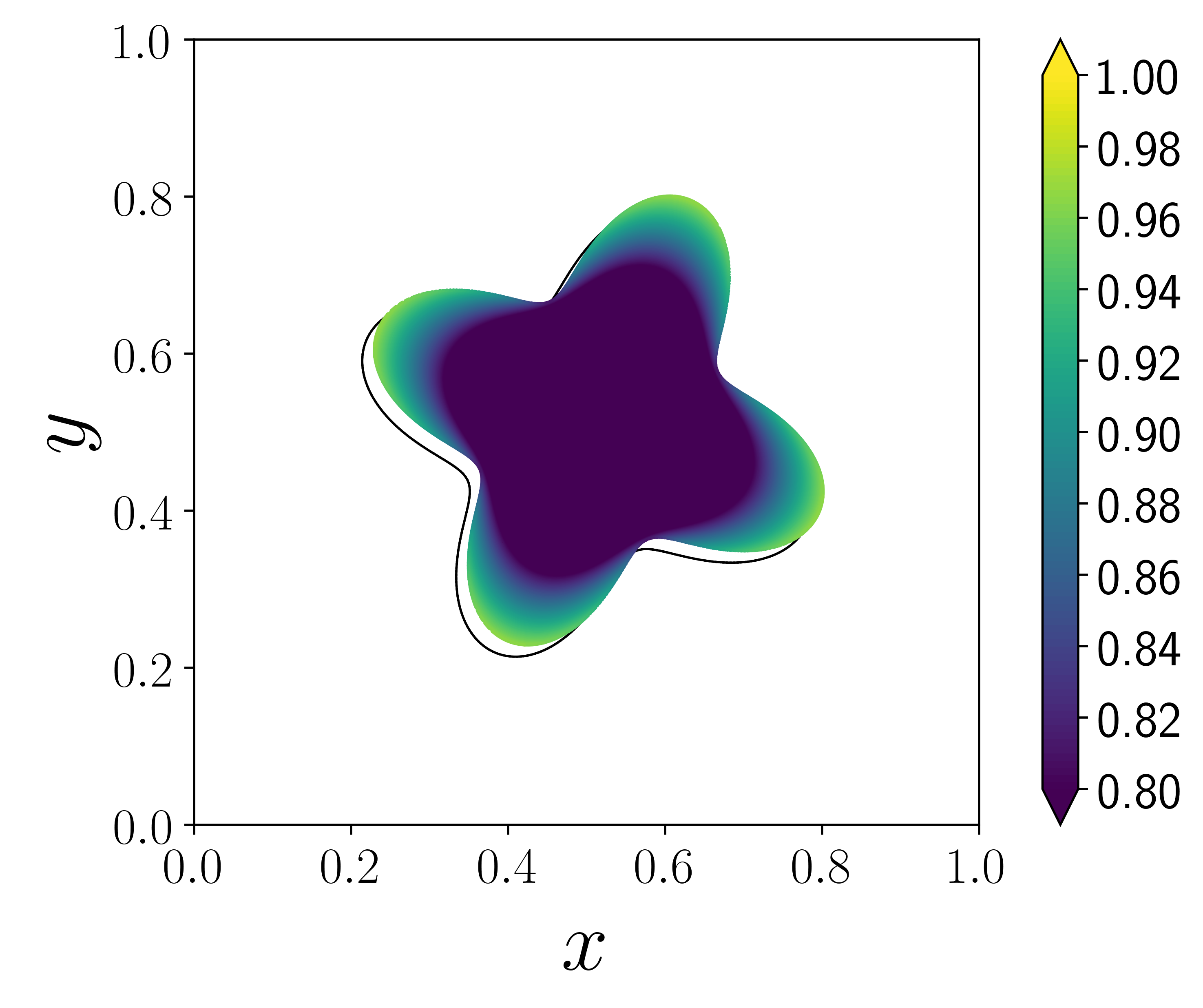}{0.48\textwidth}}\\	    
	  \subfloat[Time $= 0.3$.]{%
	    \label{fig:nexp-cc-stry-c}%
	    \inc{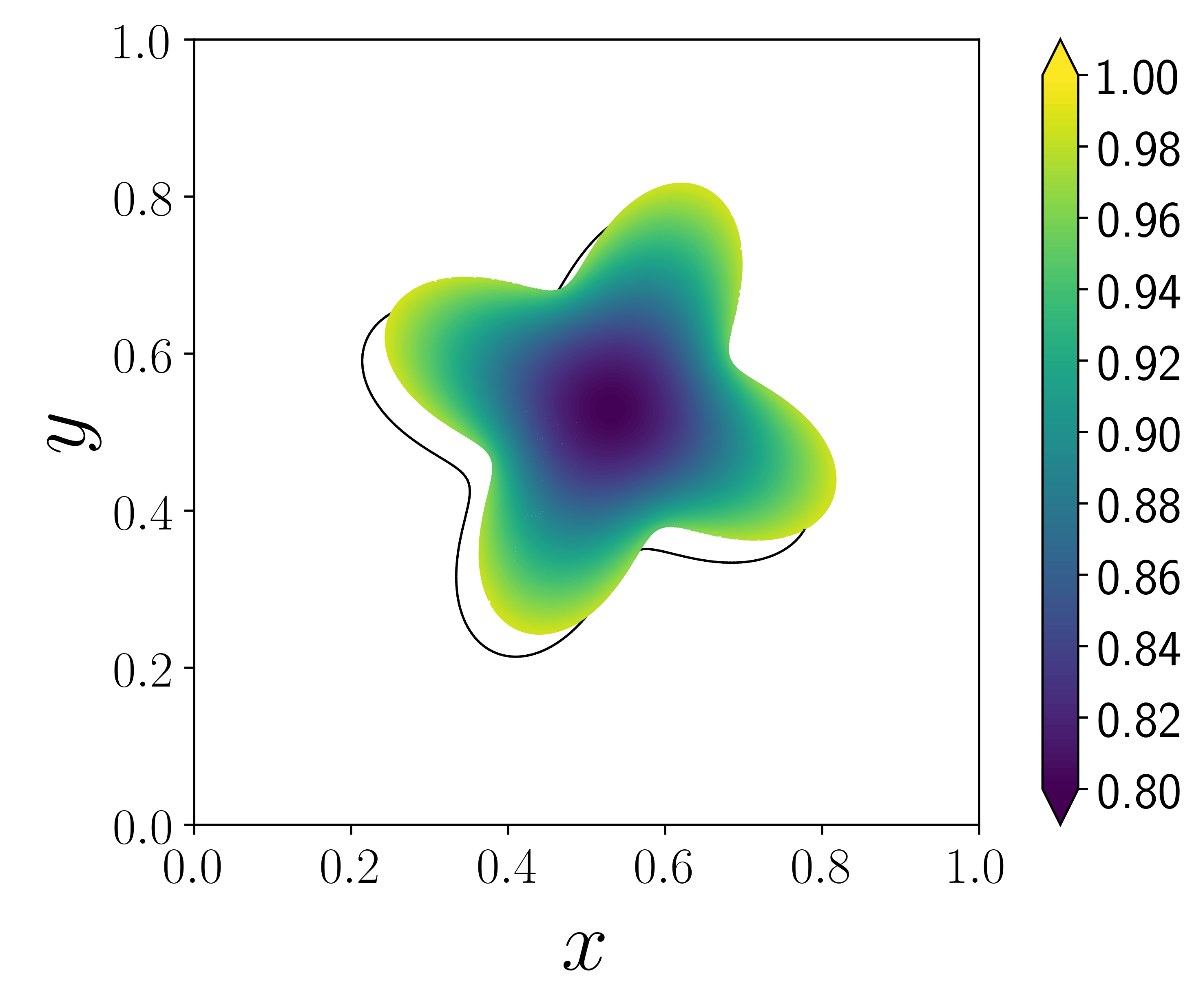}{0.48\textwidth}}		    
	  \subfloat[Time $= 0.4$.]{%
	    \label{fig:nexp-cc-stry-d}%
	    \inc{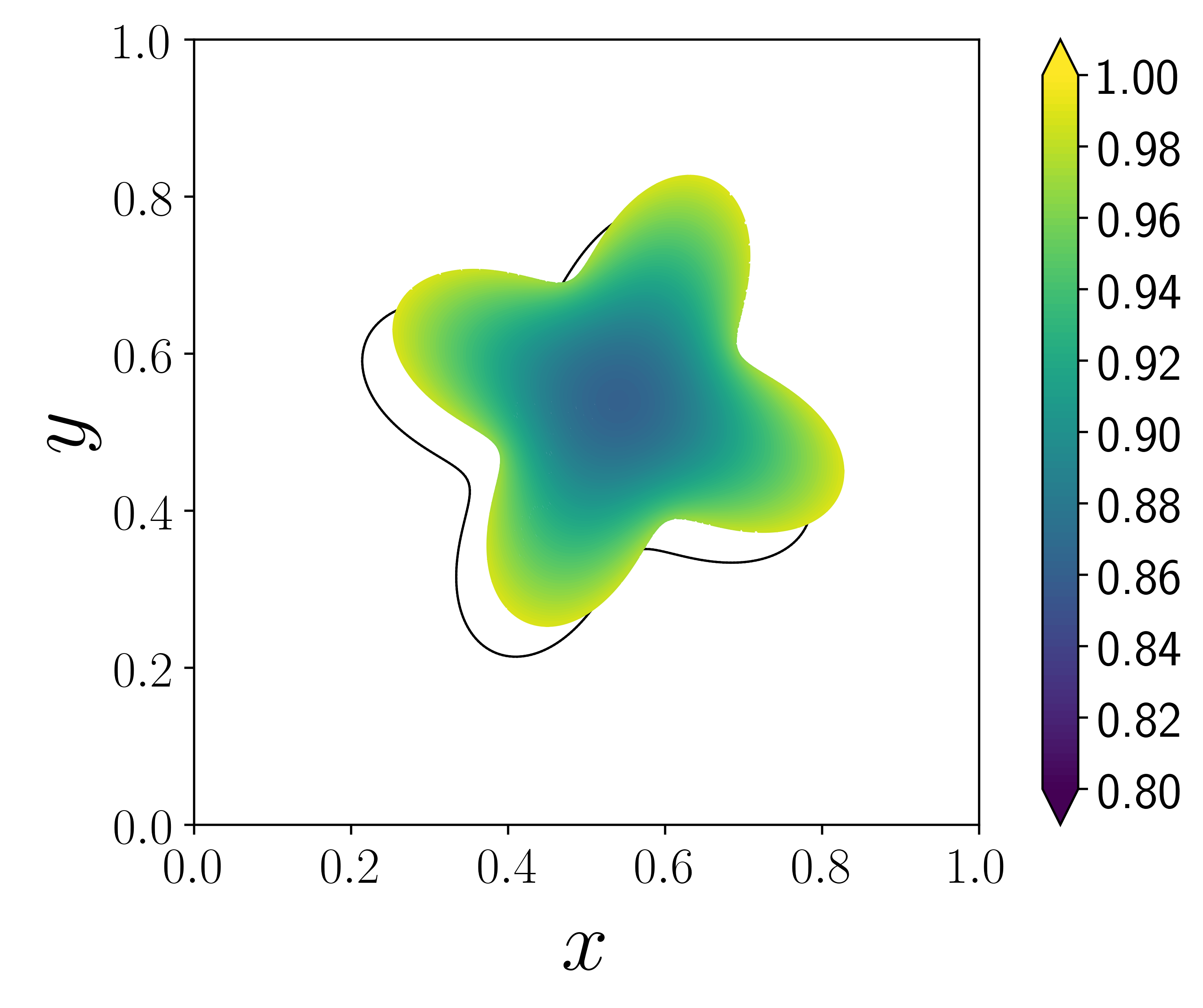}{0.48\textwidth}}\\		    
	  \subfloat[Global fully discrete conservation.]{%
	    \label{fig:nexp-cc-stry-conserv}%
	    \inc{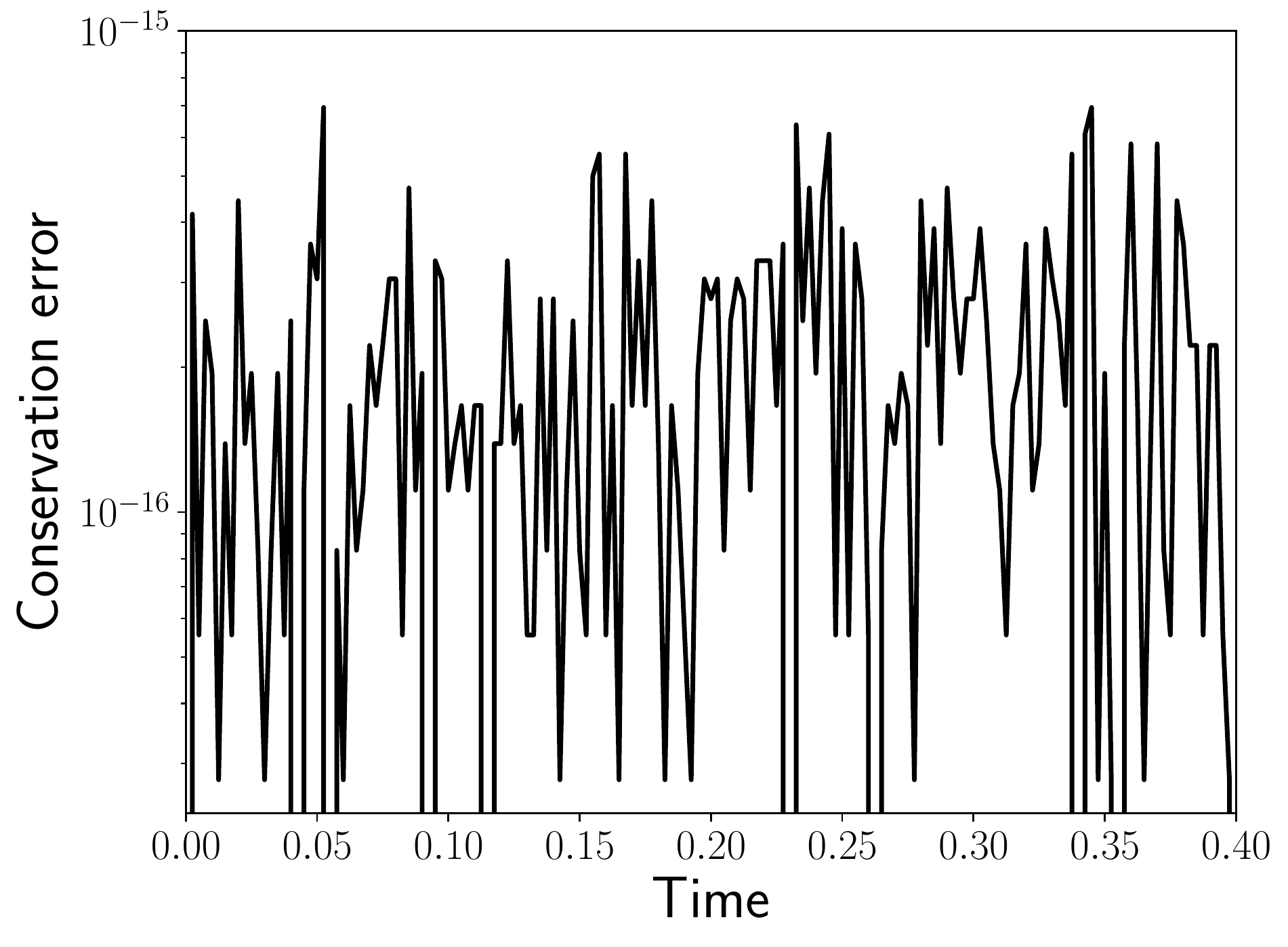}{0.48\textwidth}}	    
	  \caption[]{%
	    {\em (Advection-reaction-diffusion in a moving complex geometry)} \subref{fig:nexp-cc-stry-a} Coarser illustration of the initial mesh with $N_e^{\Omega} = 1225$ triangles and ${\cal N}_s^{\Gamma} = 111$ boundary nodes; Concentration of the active species at various points in time: \subref{fig:nexp-cc-stry-b} Time $= 0.15$, \subref{fig:nexp-cc-stry-c} Time $= 0.3$, \subref{fig:nexp-cc-stry-d} Time $= 0.4$; \subref{fig:nexp-cc-stry-conserv} Conservation error as a function of time.}
	  \label{fig:nexp-cc-stry}
	\end{figure}
      
    \subsection{Application to directed cell migration}
    \label{sec:nexp-cell}    	
    	  As discussed in \S\ref{sec:intro}, Rho GTPase activity is of key importance in directed cell migration and mass-conservative reaction-diffusion (McRD) equations are an important mathematical framework for their study. Therefore, unsurprisingly, the modelling of Rho GTPase activity has been the focus of many studies; see e.g. \cite{sakumura_2005, maree_2006, jilkine_2007, otsuji_2007, goryachev_2008, rappel_2017}. In resting cells the inactive form of the Rho GTPase is predominate. Some of the inactive form will be cytoplasmic, sequestered by GDI, and some will be membrane bound. The activation of the inactive form (in response to an upstream signal, for example) is facilitated by GDP/GTP exchange factors (GEFs), which stimulate the dissociation of GDP from the inactive form. This process cannot take place within the cytoplasm due to the inactive form being sequestered by GDI. Consequently, activation can only take place on the membrane and therefore, in order for sustainable activation to be seen, the dissociation of GDI is essential. Additionally, the exchange factors (GEFs) are cytoplasmic and therefore, must be recruited to the membrane so that activation can take place. Hydrolysis of GTP to GDP is stimulated by GTPase activating proteins (GAPs) and is responsible for the inactivation of the active form.
    	  
    	  Mar\'{e}e \textit{et al.} \cite{maree_2006} considered a multiscale model of a motile cell, where the Rho GTPase activity was coupled to the cytoskeletal dynamics and motion of the cell using the Cellular Potts Model \cite{glazier_1993}. The GTPase activity demonstrated spatial bistability and, crucially, stable polarisation could be seen outside of the bistable region, meaning that a stable resting cell would be possible. The onset of polarisation was then studied in greater detail in \cite{jilkine_2007}, where cooperativity between the interacting Rho GTPases and the fast diffusion of the GDI bound inactive form were seen to be essential in developing robust polarity. The models proposed in \cite{maree_2006, jilkine_2007} were analytically intractable due to their complexity. Therefore, building upon the previous work of \cite{maree_2006, jilkine_2007}, Mori \textit{et al.} \cite{mori_2008, mori_2011} introduced a minimalistic model for Rho GTPase activity by considering only a single Rho GTPase which cycles between an inactive (cytoplasmic) and active (membrane bound) state. The model, also known as the wave-pinning (WP) model, consists of two time-dependent reaction-diffusion equations in one spatial dimension and is given by
    	\begin{subequations}
    		\begin{align}
    			\ptl{a}{t} &= D_{a} \ptld{a}{x} + f(a,b), \qquad (x,t) \in [0,L]\times \overline{I}, \nonumber\\
    			\ptl{b}{t} &= D_{b} \ptld{b}{x} - f(a,b), \nonumber
    		\end{align}
    	\end{subequations}
    	where $a$, $b$ denote the active and inactive states, respectively and $L \in \RR$. The reaction kinetics are given by
    	\begin{equation}
    		f(a,b) = \left( k_{0} + \frac{\gamma a^{2}}{K^{2} + a^{2}} \right)b - \delta a + f_s, \nonumber
    	\end{equation}
    	where $k_{0}, \gamma, K, \delta \in \RR$ are model parameters and $f_s$ is an external stimulus. The external stimulus is assumed to increase the rate of activation. Under the influence of an external stimulus, this model sets up a travelling wave in the active state which is later halted (pinned) due to the interaction with the inactive state. Fundamental assumptions at the heart of this model are the properties of mass conservation and also the large differences in the active and inactive diffusivities. Vanderlei \etal \cite{vanderlei_2011} later extended this model to two spatial dimensions, incorporated a mechanical model for the membrane and simulated its movement using the immersed boundary method \cite{peskin_1972}. Note that following the discussion above, a consequence of this model is that the dissociation of GDI and the action of GEFs to produce the active form, are incorporated into a single step. Also note that the GEFs are assumed to be freely available.
    	
    	In \cite{maree_2006, jilkine_2007, mori_2008, vanderlei_2011}, the active and inactive forms are intermixed and therefore, occupy the same physical domain. One could arrive at this scenario by considering a top-down view of a cell, essentially flattening a three-dimensional cell into two dimensions. However, as discussed above, the active and inactive (without GDI) forms are membrane bound and therefore, occupy a physically distinct region to the inactive GDI bound state. This compartmentalisation is a key biological component of cell polarisation. Bulk-surface implementations of the wave-pinning model can be found in the literature (see e.g. \cite{ramirez_2015, giese_2015, diegmiller_2018, cusseddu_2019}), where the domain is assumed to be stationary. Although the original formulation of the wave-pinning model has been applied on moving domains (see e.g. \cite{maree_2006, vanderlei_2011, holmes_2012-1}), we are not currently aware of any studies involving the bulk-surface wave-pinning model on moving domains. 
    	
    	Section \ref{sec:bswp-stationary} illustrates the mass-conservative properties of the bulk-surface wave-pinning model on a stationary circular domain. Section \ref{sec:bswp-moving} then considers the bulk-surface wave-pinning model on a moving domain, where the motion of the cell membrane is driven by the level of activation and the stimulation is dependent on an extracellular ligand field and membrane bound ligand-receptor concentration.
    	
    	\subsubsection{Bulk-surface wave-pinning on a stationary circular domain}
    	\label{sec:bswp-stationary}
    		Before we consider the bulk-surface wave-pinning model on a moving domain, we illustrate the mass conserving properties of our algorithm on a stationary, circular domain \cite{cusseddu_2019}. Similar to the previous examples, we assume that there is no exterior environment so that $N_l = 0$ and $N_{l_s} = 0$. The initial domain $\Omega(0)$ is assumed to be a circle of radius $R_{0}$ centred at the origin. As the domain is stationary, $\Omega(t) = \Omega(0)$, $\forall t \in \overline{I}$, and the material velocities ${\bfa u}_{\Omega} = {\bfa u}_{\Gamma} = {\bfa 0}$. In this example, there is a single bulk species $(N_c = 1)$ corresponding to the inactive state and a single surface species $(N_{c_s} = 1)$ corresponding to the active state. Therefore, in \eqref{eq:rd-int-bulk} $f_c^{(1)}({\bfa c}) = 0$ and $r^{(1)}({\bfa c}) = 0$, and in \eqref{eq:rd-int-surf} $f_{c_s}^{(1)}({\bfa c}_{s}) = 0$. The interaction between the bulk and surface species is given by the reaction-kinetics
	    	\[
	    	  \hat{r}_{s}^{(1)}({\bfa c}, {\bfa c}_{s}) = \hat{r}^{(1)}({\bfa c}, {\bfa c}_{s}) = \omega\left( k_{0} + \frac{\gamma \left(c_s^{(1)}\right)^{2}}{K^{2} + \left(c_s^{(1)}\right)^{2}} \right)c^{(1)} - \delta c_s^{(1)},
	   	\]
	   	where $\omega \in \RR$ is a length scaling parameter. It is clear that $\hat{r}^{(1)}$ and $\hat{r}_s^{(1)}$ given above satisfy the conditions \eqref{eq:conserv-cont-interact} of Theorem \ref{thm:conserv-cont} and therefore, we expect the system to be globally conservative. Note that in this example, we do not consider an external stimulus; that is, $f_s = 0$. Following \cite{cusseddu_2019, ratz_2012} we assume that $\omega$  is defined as the ratio between the bulk volume/area and surface area/length; that is, $\omega = |\Omega(0)| / |\Gamma(0)|$. Following \cite{cusseddu_2019}, the initial conditions are given by
	   		\begin{align}
	   			c^{(1)}_{s}({\bfa x},0) &= s_{l} + (s_{h} - s_{l})\exp{\left[-\left( \frac{\theta^{2}}{2\sg} \right)\right]}, \nonumber \\
	   			c^{(1)}({\bfa x},0) &= c^{(1)}_{0}, \nonumber
	   		\end{align}
	   		where the (dimensional) parameters are detailed in Table \ref{tab:bswp-stationary-params} and $\theta$ denotes the angle measured anti-clockwise from the positive $x$-axis. The parameters are the same as used in \cite{mori_2008,cusseddu_2019} and the concentrations $s_{l}$, $s_h$ and $c^{(1)}_0$ are chosen so that wave-pinning occurs \cite{mori_2008,cusseddu_2019}. We set the final time $T = 6000$ and time step size $\Delta t = 0.1$ (so that $N_T = 60000$). The initial mesh is uniform and contains $N_e^{\Omega} = 2097$ triangles and ${\cal N}_s^{\Gamma} = 103$ boundary nodes.
	   		
	   		\begin{table}[ht]
	   			\centering
	   			
	   			\renewcommand{\arraystretch}{1.5}
	   			\begin{tabular}{| c | c | c || c | c | c || c | c | c |}
	   				\hline
	   				Parameter & Value & Units & Parameter & Value & Units & Parameter & Value & Units \\ \hline\hline
					$R_0$ & $5.0$ & $\mu$m & $k_{0}$ & 0.067 & $s^{-1}$ & $\gamma$ & $1.0$ & $s^{-1}$ \\ \hline
					$K$ & $1.0$ & mol $\mu$m$^{-2}$ & $\delta$ & $1.0$ & $s^{-1}$ & $s_l$ & $0.2805$ & mol $\mu$m$^{-2}$ \\ \hline
					$s_h$ & $1.5491$  & mol $\mu$m$^{-2}$ & $\sigma$ & 0.005 & - & $c^{(1)}_0$ & 0.4009 & mol $\mu$m$^{-3}$ \\ \hline
					$D_{c_s}^{(1)}$	& 0.01 & $\mu$m$^{2}$ $s^{-1}$ & $D_c^{(1)}$ & $10.0$ & $\mu$m$^{2}$ $s^{-1}$ \\ \cline{1-6}
	   			\end{tabular}
	   			\caption{%
	   				Dimensional bulk-surface wave-pinning parameters on a stationary domain.}
				\label{tab:bswp-stationary-params}
	   		\end{table}
	   
	   		The concentration of the active surface state as a function of the scaled arc-length around the circular cell and the inactive cytoplasmic state are illustrated at times $t = 5, 45, 105, 145$ in Fig. \ref{fig:nexp-cell-wp-wpwp}. Clearly when $t = 5$, the surface active state displays a sharp Gaussian-like shape centred around $\theta = 0$ (top image of Fig. \ref{fig:nexp-cell-wp-wpwp}\subref{fig:nexp-cell-wp-wpwp-sf-a}) which is inherited from the initial condition. Correspondingly, the cytoplasmic inactive state locally depletes around the highest membane-bound active state concentration (bottom image of Fig. \ref{fig:nexp-cell-wp-wpwp}\subref{fig:nexp-cell-wp-wpwp-sf-a}). The peak in the membrane bound activated state flattens, and therefore broadens, as time progresses (see e.g. top image of Fig. \ref{fig:nexp-cell-wp-wpwp}\subref{fig:nexp-cell-wp-wpwp-sf-b}) due to the wave of activation travelling around the membrane. Correspondingly, two locally depleted regions just inside the active wave-front can be seen in the cytoplasmic inactive state (see e.g. bottom image of Fig. \ref{fig:nexp-cell-wp-wpwp}\subref{fig:nexp-cell-wp-wpwp-sf-b}). This process continues to later times (see e.g. Fig. \ref{fig:nexp-cell-wp-wpwp}\subref{fig:nexp-cell-wp-wpwp-sf-d}) until eventually the wave is pinned, the process terminates and (assuming no external stimulus) a stable pattern is formed. The results depicted in Fig. \ref{fig:nexp-cell-wp-wpwp} are in excellent qualitative agreement with those presented in \cite{cusseddu_2019}. Note that in Figs. \ref{fig:nexp-cell-wp-wpwp}\subref{fig:nexp-cell-wp-wpwp-sf-a}-\subref{fig:nexp-cell-wp-wpwp-sf-d} the colourbar axis changes as time progresses. Due to the high diffusivity inside the cell, the inactive cytoplasmic species is approximately constant and therefore, changing the colourbar axis as time progresses emphasizes the local pattern of the cytoplasmic species. Figure \ref{fig:nexp-cell-wp-wpwp-csv}\subref{fig:nexp-cell-wp-wpwp-conserv} confirms that the fully discrete numerical solution is globally conservative to machine-zero. The drop off in conservation error as time progresses is most likely caused by the travelling wave of activation halting as it becomes pinned.
	    	
	    	\begin{figure}
		  		\centering
				\subfloat{%
				    \label{fig:nexp-cell-wp-wpwp-sf-a}%
    				    \inc{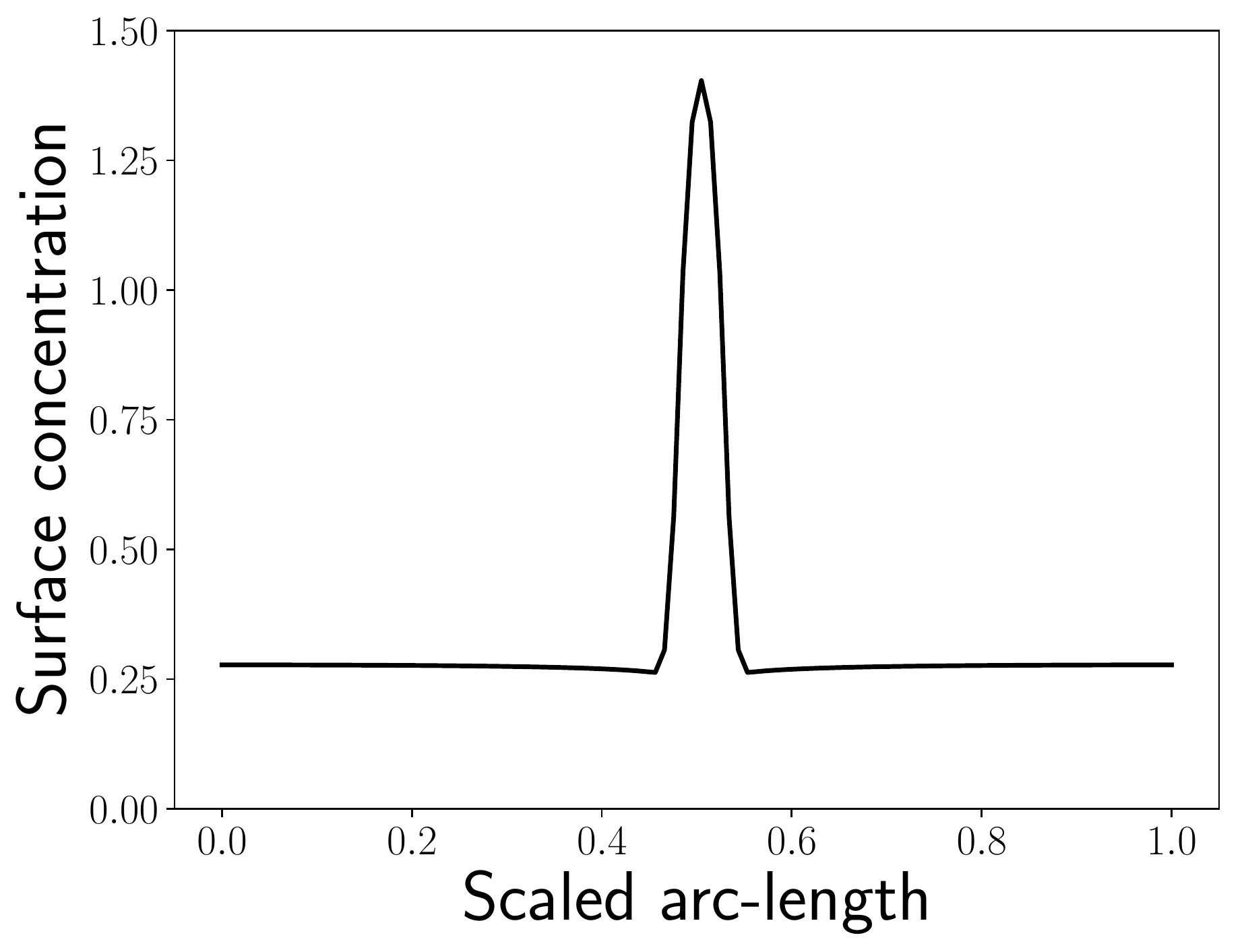}{0.24\textwidth}}
				\subfloat{%
				    \label{fig:nexp-cell-wp-wpwp-sf-b}%
				    \inc{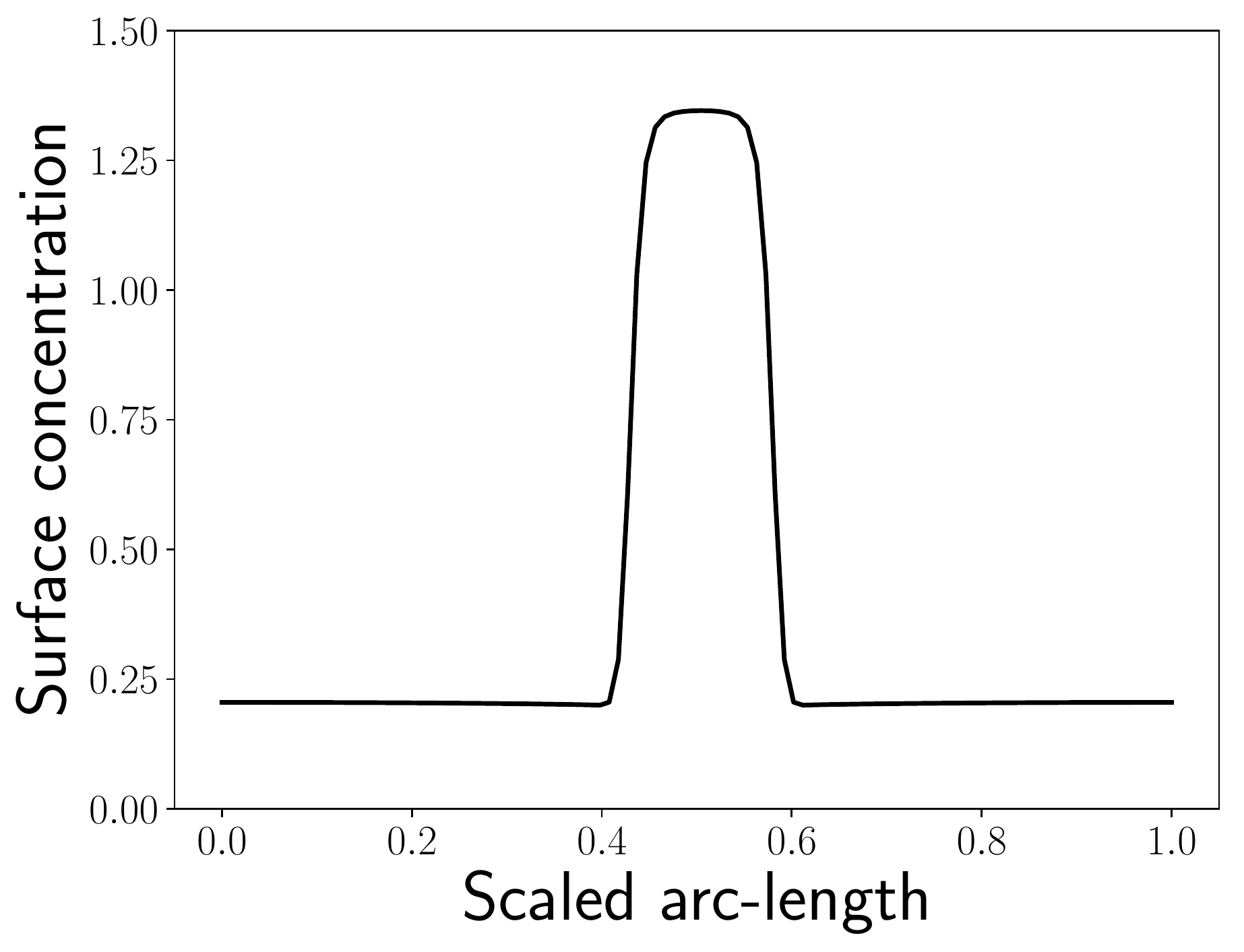}{0.24\textwidth}}
				\subfloat{%
				    \label{fig:nexp-cell-wp-wpwp-sf-c}%
				    \inc{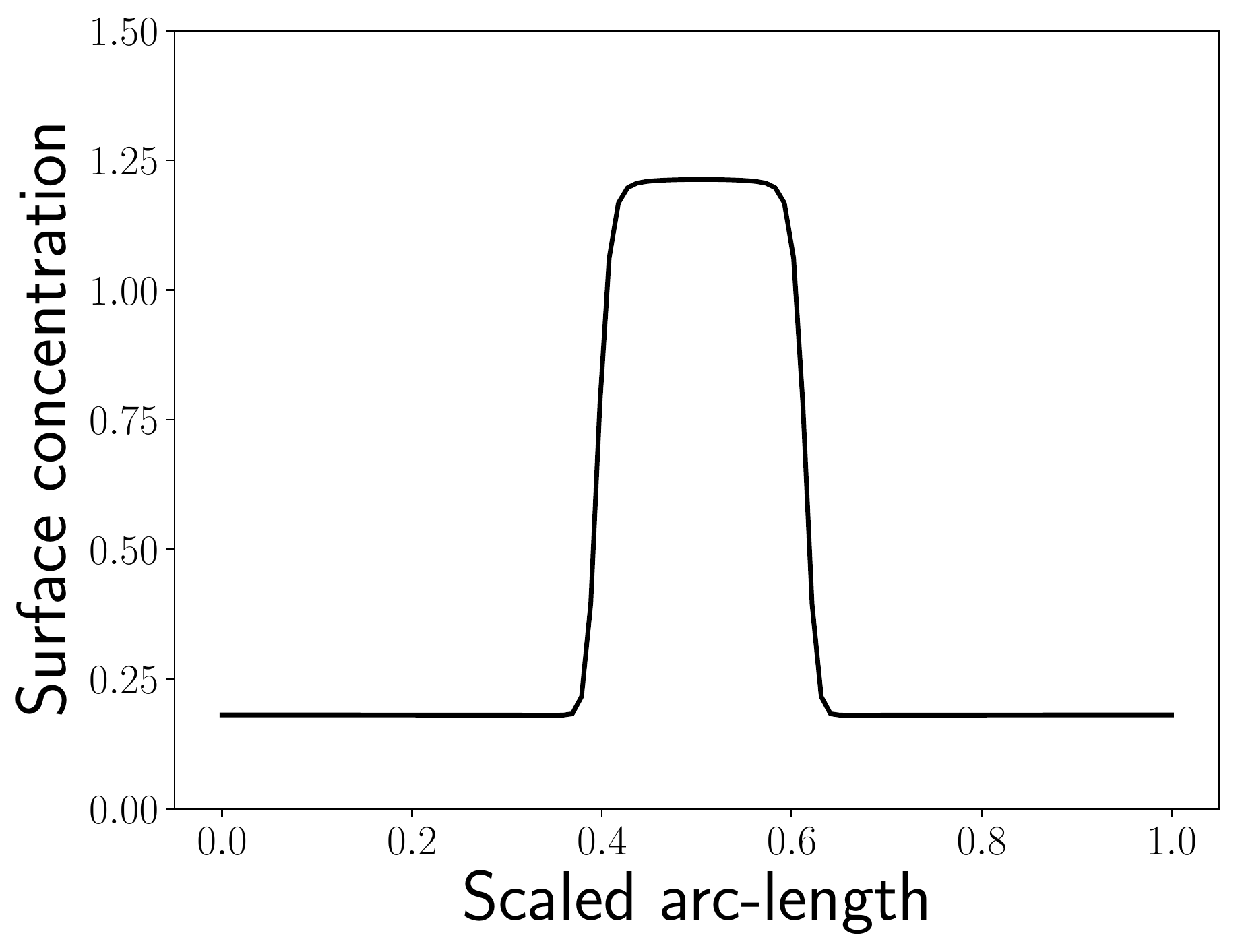}{0.24\textwidth}}
				\subfloat{%
				    \label{fig:nexp-cell-wp-wpwp-sf-d}%
				    \inc{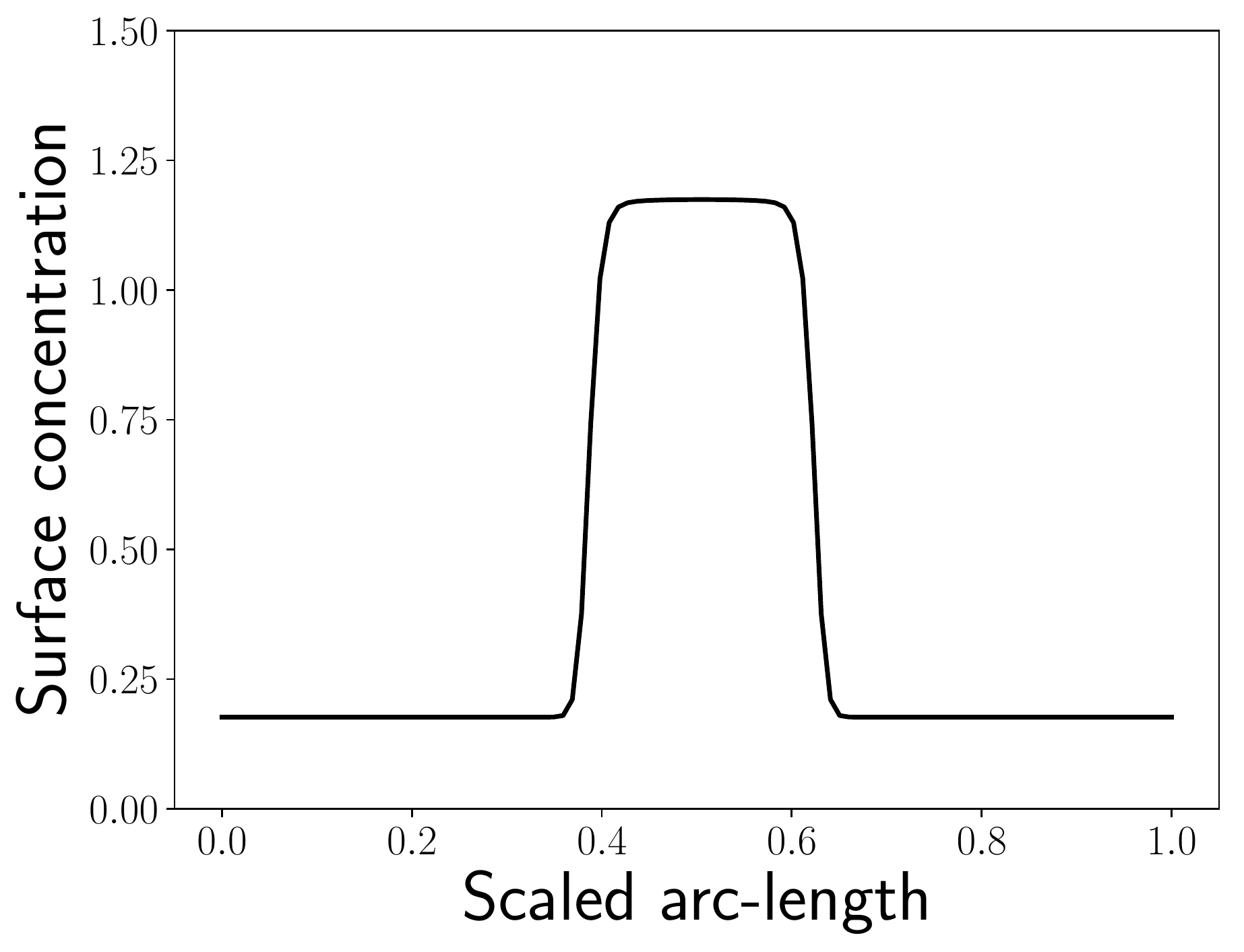}{0.24\textwidth}}
				\setcounter{subfigure}{0}
				\subfloat[Time = $5$.]{%
				    \label{fig:nexp-cell-wp-wpwp-a}%
				    \inc{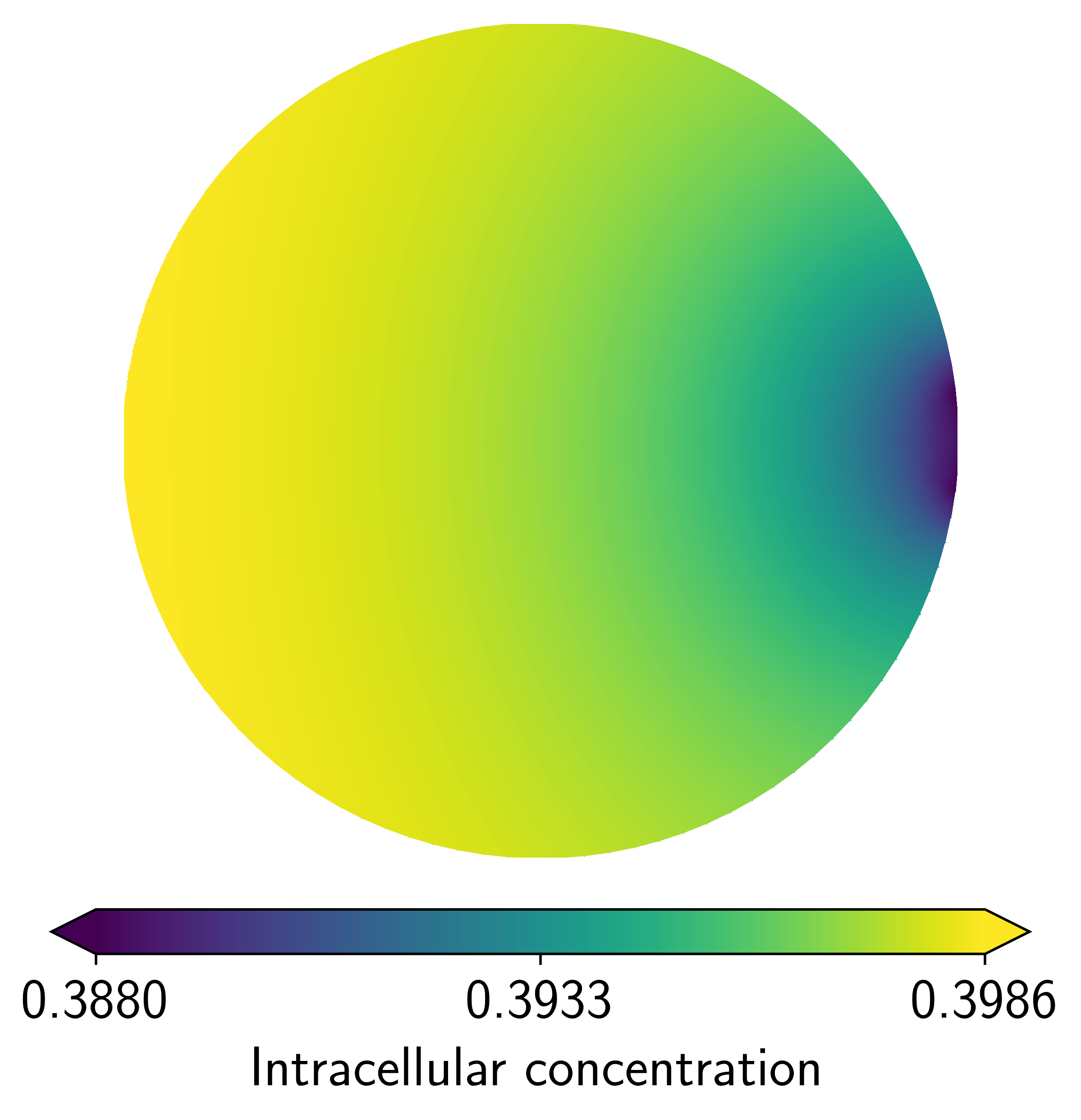}{0.24\textwidth}}
				\subfloat[Time = $45$.]{%
				    \label{fig:nexp-cell-wp-wpwp-b}%
				    \inc{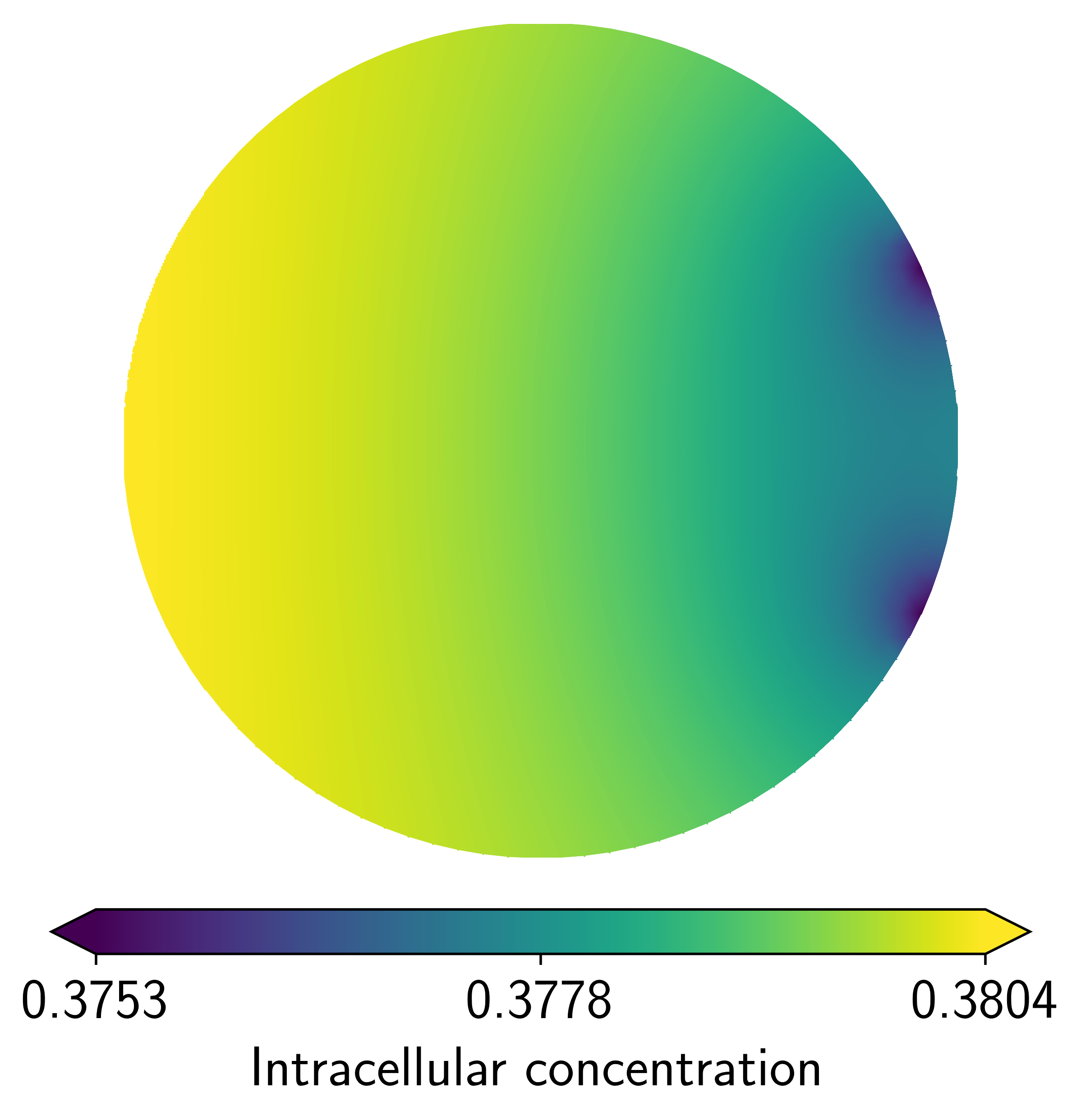}{0.24\textwidth}}
				\subfloat[Time = $105$.]{%
				    \label{fig:nexp-cell-wp-wpwp-2-c}%
				    \inc{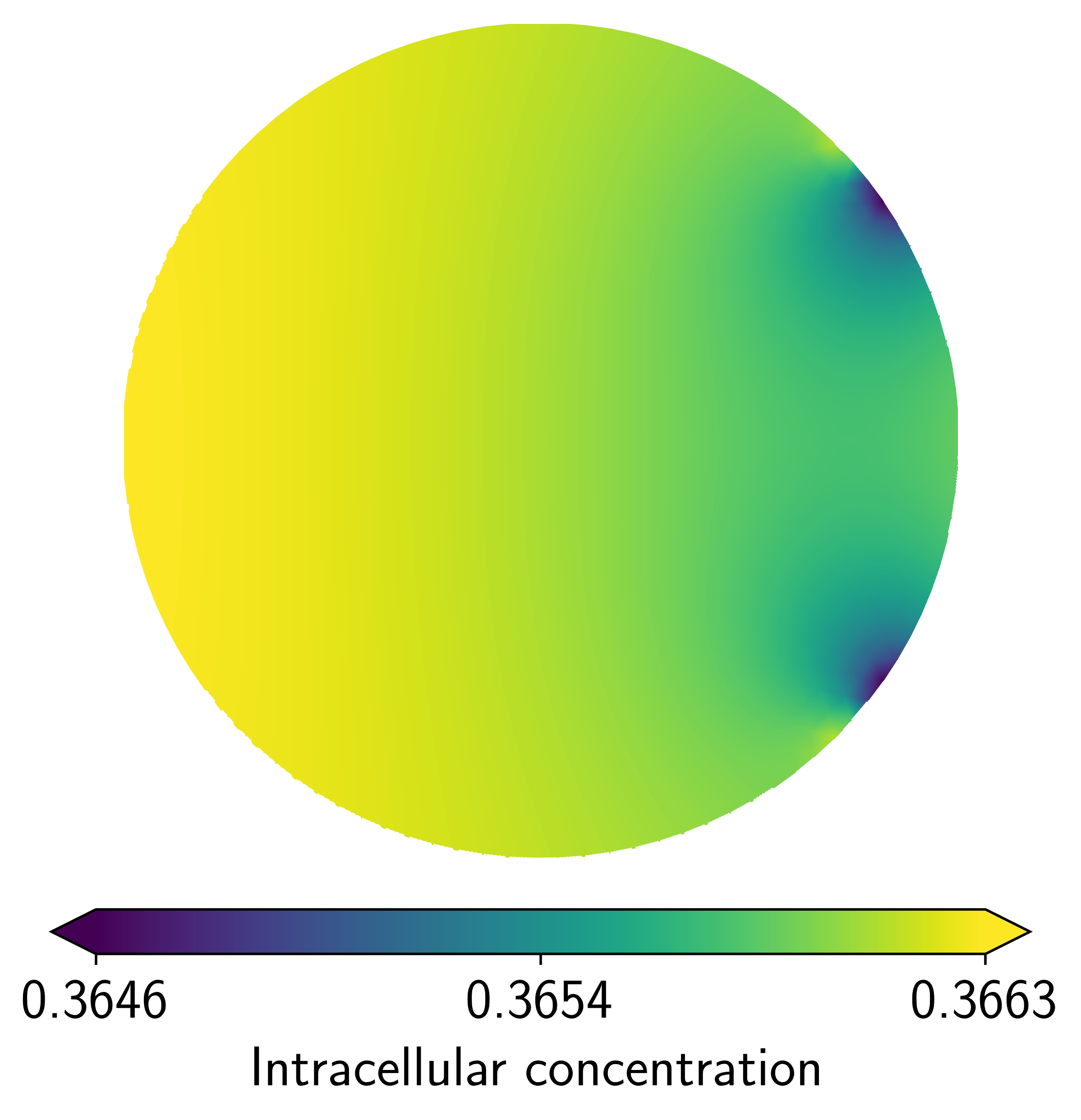}{0.24\textwidth}}
				\subfloat[Time = $145$.]{%
				    \label{fig:nexp-cell-wp-wpwp-2-d}%
				    \inc{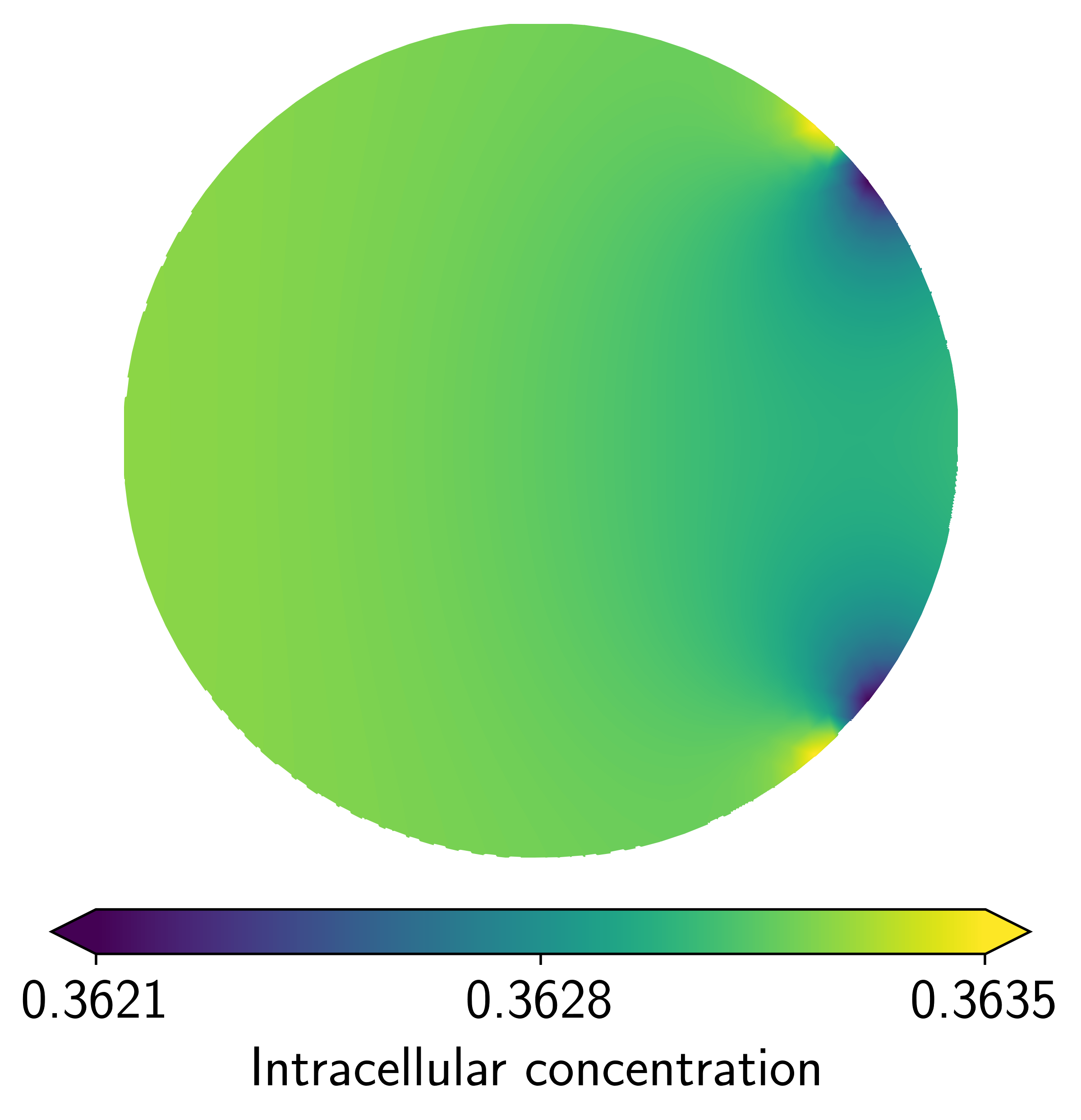}{0.24\textwidth}}
				\caption[]{%
					{\em (Bulk-surface wave-pinning on a stationary domain)} Concentrations of membrane bound active state (top row) and cytoplasmic inactive state (bottom row) on a stationary, circular domain at times \subref{fig:nexp-cell-wp-wpwp-a} $t = 5$, \subref{fig:nexp-cell-wp-wpwp-b} $t = 45$, \subref{fig:nexp-cell-wp-wpwp-2-c} $t = 105$, \subref{fig:nexp-cell-wp-wpwp-2-d} $t = 145$.}
				\label{fig:nexp-cell-wp-wpwp}
			\end{figure}
			
			\begin{figure}
	  			\centering
	  			\subfloat[Stationary domain.]{%
		    			\label{fig:nexp-cell-wp-wpwp-conserv}%
		    			\inc{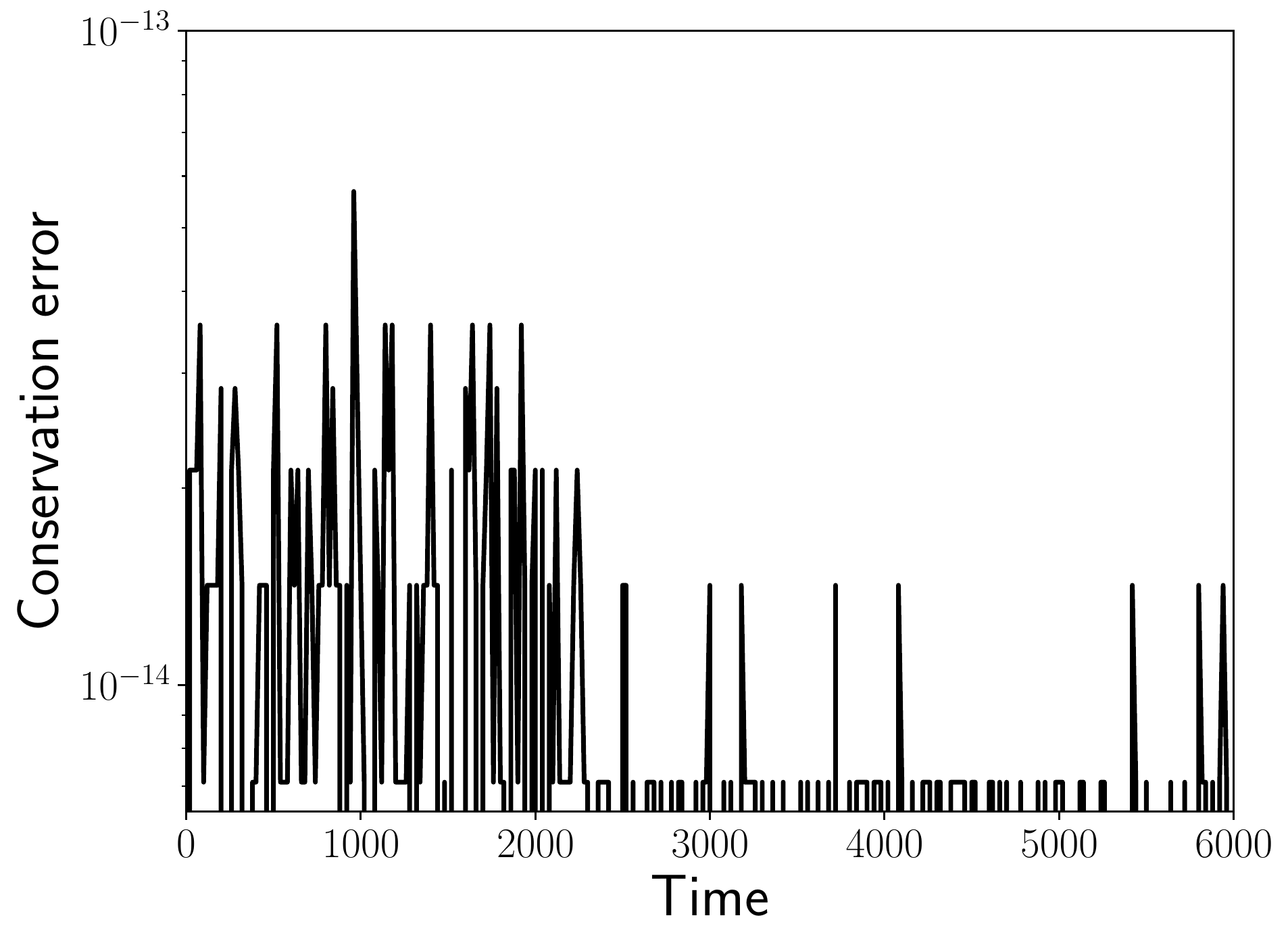}{0.48\textwidth}}\quad
		    		\subfloat[Moving domain.]{%
		    			\label{fig:nexp-cell-wp-wplr-conserv}%
		    			\inc{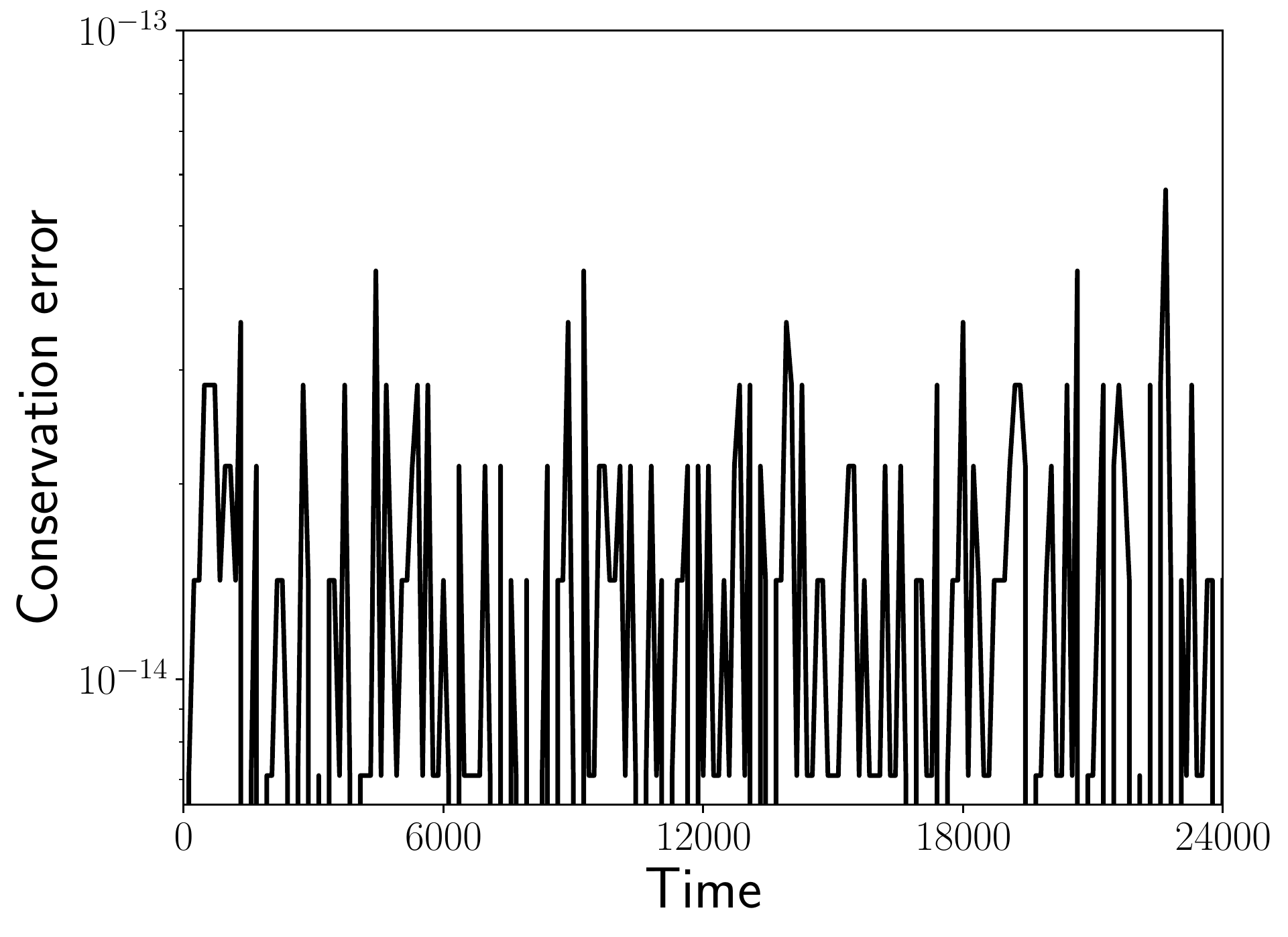}{0.48\textwidth}}
		    		\caption[]{%
	    				Global fully discrete conservation error as a function of time for the bulk-surface wave-pinning model on a \subref{fig:nexp-cell-wp-wpwp-conserv} stationary, circular domain and \subref{fig:nexp-cell-wp-wplr-conserv} a moving domain.}
	    			\label{fig:nexp-cell-wp-wpwp-csv}
		    	\end{figure}
      
	    \subsubsection{Bulk-surface wave-pinning model on a moving domain with extracellular ligand field}
	    \label{sec:bswp-moving}
			In the exterior environment $\Xi(t)$, we assume there exists an external guidance cue in the form of freely diffusing chemoattractant ligand molecules. At the cell membrane $\Gamma(t)$, these molecules bind reversibly to membrane bound receptors. These transmembrane receptors have {\em binding domains} facing both the exterior and interior environments, making them an essential component in the response of a cell to an external guidance cue. Upon binding with an extracellular ligand, the receptor undergoes a conformational change resulting in an active ligand-receptor complex. This active complex then triggers downstream signalling events which lead to cellular migration. Although important for the activation of intracellular signalling pathways, the action of the active ligand-receptor complex on downstream signalling components is that of a catalyst and therefore, is not used up in subsequent reactions. Although the intracellular signalling pathways have received a large amount of attention in the literature, the coupling of these pathways to an external guidance cue, on a moving domain, has received far less attention. Usually, the external guidance cue (or equivalently, ligand-receptor complex concentration) is prescribed in the activation of membrane bound species (see e.g. \cite{meinhardt_1999, levine_2006, maree_2012, elliott_2012}) or the interaction is studied on a stationary domain (see e.g. \cite{levchenko_2002, subramanian_2004, garcia-penarrubia_2014}). Recently, Moure and Gomez \cite{moure_2018} proposed a phase-field model for the simulation of obstacle-mediated chemotaxis. The cytoplasmic mechanics of the acto-myosin network were considered to be dependent on a membrane bound {\em activator}, which in turn depends on the membrane concentration of chemoattractant. This suggests that the {\em activator} approximates the entire signalling pathway from extracellular ligand concentration to the intracellular acto-myosin network. Feng \etal \cite{feng_2018} proposed a lattice Boltzmann-particle method for neutrophil migration, where their model comprised four {\em signalling layers}: extracellular chemoattractant and membrane bound receptor; membrane bound G protein dynamics and specific effectors; cytoplasmic and membrane bound Rho family of GTPases; and their downstream effectors. However, the membrane is unfitted to the underlying computational lattice and therefore, an additional {\em lattice-particle map} was required to redefine the subset of the computational lattice (which defines the membrane) as the cell migrates. Within the context of conservation, this procedure would require care to ensure information is not lost under motion. 
			
			We consider a single exterior bulk species $(N_l = 1)$ corresponding to the extracellular chemoattractant ligand and a single surface species $N_{l_s} = 1$, which corresponds to the active ligand-receptor complex. We assume that the material velocity of the extracellular ligand field is given by ${\bfa u}_{\Xi} = {\bfa 0}$. Following \cite{macdonald_2016}, we employ Dirichlet boundary conditions on the outer boundary $\partial{\cal D}(t)$. The exterior bulk reaction term, and boundary conditions for the outer boundary, will be defined shortly. On the inner boundary (cell membrane) $\Gamma(t)$, the bulk boundary conditions \eqref{eq:rd-ext-bulk-ibcs} and surface reaction kinetics \eqref{eq:rd-ext-surf} is given by
			\[
			  g^{(1)}({\bfa l}) = 0, \quad f_{l_s}^{(1)}({\bfa l}_s) = 0, \quad \mbox{and} \quad \hat{g}^{(1)}({\bfa l}, {\bfa l}_s) = \hat{g}_s^{(1)}({\bfa l}, {\bfa l}_s) = k_{-1}l_s^{(1)} - k_1(R_T - l_s^{(1)})l^{(1)},
			\]
			respectively, where $k_1$ is the ligand association rate, $k_{-1}$ is the disassociation rate and $R_T$ is the total concentration of receptors. Following the previous section \S\ref{sec:bswp-stationary}, in the interior environment $\Omega(t)$ we assume there is a single bulk species $(N_c = 1)$ and a single surface species $(N_{c_s} = 1)$. Furthermore, we assume: that in \eqref{eq:rd-int-bulk} $f_c^{(1)}({\bfa c}) = 0$ and $r^{(1)}({\bfa c}) = 0$; and in \eqref{eq:rd-int-surf} $f_{c_s}^{(1)}({\bfa c}_{s}) = 0$. However, contrary to the previous section, the interaction between the interior bulk and surface species are given by the reaction-kinetics
			\[
			  \hat{r}_{s}^{(1)}({\bfa c}, {\bfa c}_{s}, {\bfa l}_s) = \hat{r}^{(1)}({\bfa c}, {\bfa c}_{s}, {\bfa l}_s) = \omega\left( k_{0} + \frac{\gamma \left(c_s^{(1)}\right)^{2}}{K^{2} + \left(c_s^{(1)}\right)^{2}} \right)c^{(1)} - \delta c_s^{(1)} + \varepsilon\frac{l_s^{(1)}}{R_T}c^{(1)},
			\]
			where $\omega \in \RR$ is a length scaling parameter defined in \S\ref{sec:bswp-stationary}, $\varepsilon$ is the rate of stimulation and $\tilde{R} = l_s^{(1)} / R_T$ defines the local fractional receptor occupancy. We assume (initially) that the intracellular material velocity ${\bfa u}_{\Omega} = {\bfa 0}$. However, due to the dynamic motion of the cell membrane considered in this example, ${\bfa u}_{\Gamma} \neq {\bfa 0}$. In fact, in general the material velocity ${\bfa u}_{\Gamma}$ is not known {\em {a-priori}}. As the shape of the curve is determined solely by the normal motion, we assume ${\bfa u}_{\Gamma}\cdot{\bfa t} = 0$.
	    		
	    		The (dimensional) parameters for this model are detailed in Table \ref{tab:bswp-stationary-params} and Table \ref{tab:bswp-moving-params}. The parameters are the same as used in \cite{mori_2008, cusseddu_2019, macdonald_2016, mackenzie_2016}, with the exception of $\varepsilon$ and $\sigma_l$ which are guessed. Furthermore, so that the wave-pinning model has enough time to establish a polarised state, the diffusion parameters $D_{l_s}^{(1)}$ and $D_l^{(1)}$ have been lowered by an order of magnitude (but maintain the order of magnitude of their ratio) compared with those used in \cite{macdonald_2016, mackenzie_2016}. Initially, the extracellular environment $\Xi(0)$ is assumed to be an annulus where the inner boundary $\Gamma(0)$ is a circle of radius $R_0$ and the outer boundary $\partial{\cal D}(0)$ is also a circle of radius $R_{{\rm outer}}$. Therefore, initially the intracellular environment $\Omega(0)$ is a circle of radius $R_0$. The radius of the outer boundary needs to be sufficiently large so that the Dirichlet boundary condition remains valid. Assuming that there are $N_R / (4\pi R_0^2)$ receptors per unit surface area of a spherical cell (number concentration), then the molar concentration of the total number of receptors is calculated using the following formula:
	   		\[
	   			R_T = \left( \frac{N_R}{4\pi R_0^2} \right)\frac{10^{24}}{N_A},
	   		\]
	   		where $N_A = 6.022\times 10^{23}$ is Avogadro's constant and the factor $10^{24}$ arises from the conversion to molar concentration with nanomolar units (nM) and to a length scale with micron units ($\mu$m). For the extracellular ligand and membrane bound ligand-receptor complex, we initially prescribe a homogeneous fractional receptor occupancy level $(\tilde{R}_0)$ which is given in Table \ref{tab:bswp-moving-params}. The initial condition for the extracellular ligand field and the membrane bound ligand-receptor complex is then calculated from the initial fractional receptor occupancy level and the steady state \cite{neilson_2011}:
	   		\[
	   			l^{(1)}({\bfa x},0) = l_0^{(1)} = \frac{K_d\, \tilde{R}_0}{1 - \tilde{R}_0}, \quad \mbox{and} \quad l_s^{(1)}({\bfa x},0) = l_{s,0}^{(1)} = \frac{R_T\,l_0^{(1)}}{K_d + l_0^{(1)}},
	   		\]
	   		where $K_d = k_{-1} / k_1$. The Dirichlet boundary condition in \eqref{eq:rd-ext-bulk-obcs} is then set: $l_{{\cal D}}^{(1)} = l_0^{(1)}$. The ligand source $f_l^{(1)}$ is defined as a Gaussian approximation of a point source:
	   		\begin{equation}
	   			f_l^{(1)}({\bfa x},t) = 10\exp{\left[ -\left( \frac{\left| {\bfa x} - {\bfa x}_{p} \right|^{2}}{2\sg_l} \right) \right]},
	   		\label{eq:bswp-moving-gaussptsrc}
	   		\end{equation}
	   		where ${\bfa x}_{p}$ is the location of the point source. Initially, ${\bfa x}_{p}(0) = (15, 0)$. The point source then changes location when the cell is within a specified distance from ${\bfa x}_{p}$. The point source is then defined to move in a square with sides of length $15$ $\mu$m. The initial conditions for the intracellular bulk-surface wave-pinning model are given by
	   		\begin{align}
	   			c^{(1)}_{s}({\bfa x},0) &= s_l, \nonumber \\
	   			c^{(1)}({\bfa x}, 0) &= c^{(1)}_{0}, \nonumber
	   		\end{align}
	   		where $s_l$ and $c^{(1)}_0$ are given in Table \ref{tab:bswp-stationary-params}. The membrane $\Gamma(t)$ is evolved according to the procedure summarised in \S\ref{sec:meshgen-sf} where the normal velocity \eqref{eq:meshgen-sf-normalvel} is given by \cite{neilson_2011, macdonald_2016}
	   		\[
	   			{\cal V}({\bfa x}, t) = K_{{\rm prot}}c^{(1)}_{s}({\bfa x},t) - \lambda(t)\kappa, \qquad ({\bfa x}, t) \in Q^{\Gamma}_{T},
	   		\]
	   		where $K_{{\rm prot}}$ is the strength of the protrusive force which is proportional to the concentration of the membrane bound active species and $\lambda(t)$ is a time-dependent cortical tension factor found according to \cite{neilson_2011, macdonald_2016}:
	   		\[
	   			\frac{d\lambda}{dt} = \frac{\lambda_0 \lambda}{A_0(\lambda + \lambda_0)}\left[ \upsilon\left(A - A_0\right) + \frac{dA}{dt} \right] - \beta\lambda,
	   		\]
	   		where $\upsilon, \beta \in \RR^{+}$ are positive parameters given in Table \ref{tab:bswp-moving-params}, $\lambda_0 = R_0 K_{{\rm prot}} c^{(1)}_{s,0}$ and $A_0$ is the initial area contained in the cell. As a full investigation of parameter space is beyond the scope of this article, $K_{{\rm prot}}$, $\upsilon$ and $\beta$ are chosen such that the cellular migration remains stable. We set the final time $T = 24000$ and time step size $\Delta t = 0.1$ (so that $N_T = 240000$). The interior and exterior meshes contain $N_e^{\Omega} = 2091$ and $N_e^{\Xi} = 5123$ triangles, respectively. The boundary of the interior mesh (and inner boundary of the exterior mesh) contains ${\cal N}_s^{\Gamma} = 101$ nodes, whilst the outer boundary of the exterior mesh contains ${\cal N}_s^{\partial {\cal D}} = 68$ nodes. For computational efficiency, we employ a graded mesh in the exterior environment, where the elements become larger the further they are from the inner boundary. However, in the interior environment, we employ a uniform triangular mesh. Although the membrane shape is dynamically changing, we fix the outer boundary of the extracellular environment to be a circle which is merely translated so that the cell remains in the centre. The extracellular and intracellular meshes are then evolved according to the procedure outlined in \S\ref{sec:meshgen-blk}.
	   		
	   		\begin{table}[ht]
	   			\centering
	   			
	   			\renewcommand{\arraystretch}{1.5}
	   			\begin{tabular}{| c | c | c || c | c | c || c | c | c |}
	   				\hline
	   				Parameter & Value & Units & Parameter & Value & Units & Parameter & Value & Units \\ \hline\hline
					$k_1$ & $1/30$ & nM$^{-1}$ $s^{-1}$ & $k_{-1}$ & $1.0$ & $s^{-1}$ & $\varepsilon$ & $1.0$ & $s^{-1}$ \\ \hline
					$D_{l_s}^{(1)}$	& $0.01$ & $\mu$m$^{2}$ $s^{-1}$ & $D_{l}^{(1)}$ & $10.0$ & $\mu$m$^{2}$ $s^{-1}$ & $N_{R}$ & $75000$ & - \\ \hline
					$\tilde{R}_0$ & $0.15$ & - & $R_{{\rm outer}}$ & $50.0$ & $\mu$m & $\beta$ & $0.02$ & $s^{-1}$ \\ \hline
					$K_{{\rm prot}}$ & $10^{-5}$ & $\mu$m$^{3}$ $s^{-1}$ mol$^{-1}$ & $\upsilon$ & $1.0$ & $s^{-1}$ & $\sg_l$ & $0.5$ & - \\ \hline
	   			\end{tabular}
	   			\caption{%
	   				Dimensional bulk-surface wave-pinning and ligand-receptor parameters on a moving domain.}
				\label{tab:bswp-moving-params}
	   		\end{table}
	   		
	   		The concentration of the interior active surface species as a function of scaled arc-length around the cell and the intracellular inactive species are illustrated at times $t = 25, 105, 225, 300$ in Fig. \ref{fig:bswp-moving-nt3e3-a}. In contrast to the stationary example, at $t = 25$ the active state displays a broad peak on the side of the cell which is facing the point source and a broad, depleted region in the cytoplasm (Fig. \ref{fig:bswp-moving-nt3e3-a}\subref{fig:bswp-moving-nt3e3-blk-a}). This broad activation becomes sharper at $t = 105$ due to the expansion of the cell in the direction of the point source (Fig. \ref{fig:bswp-moving-nt3e3-a}\subref{fig:bswp-moving-nt3e3-blk-b}) which in turn, produces a localised depleted region in the cytoplasm analagous to the stationary example (Fig. \ref{fig:nexp-cell-wp-wpwp-a}). Due to the sharp peak in the active state, the cell protrudes local to this region trapping the depleted cytoplasmic species thereby creating a {\em cap} of cytoplasmic inactivity local to the protrusion (Figs. \ref{fig:bswp-moving-nt3e3-a}\subref{fig:bswp-moving-nt3e3-blk-c} and \ref{fig:bswp-moving-nt3e3-a}\subref{fig:bswp-moving-nt3e3-blk-d}). Analagous to the stationary case, the peak in the active surface state flattens as time progresses due to the wave of activation travelling around the membrane, eventually pinning to produce a stable pattern.
			
			\begin{figure}
				\centering
	  			\subfloat{%
	    			\label{fig:bswp-moving-nt3e3-sf-a}%
	    				\inc{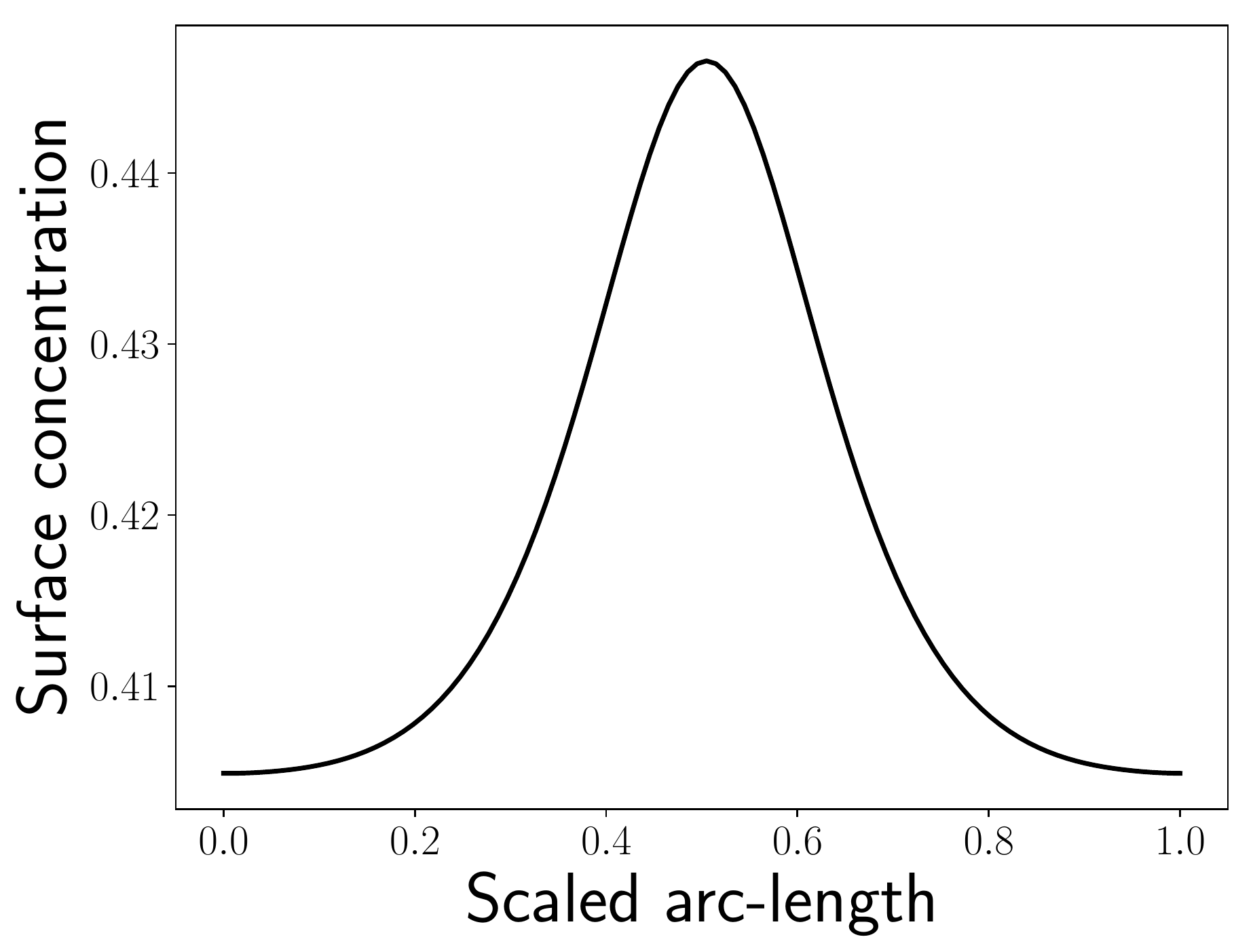}{0.24\textwidth}}
	  			\subfloat{%
	    			\label{fig:bswp-moving-nt3e3-sf-b}%
	    				\inc{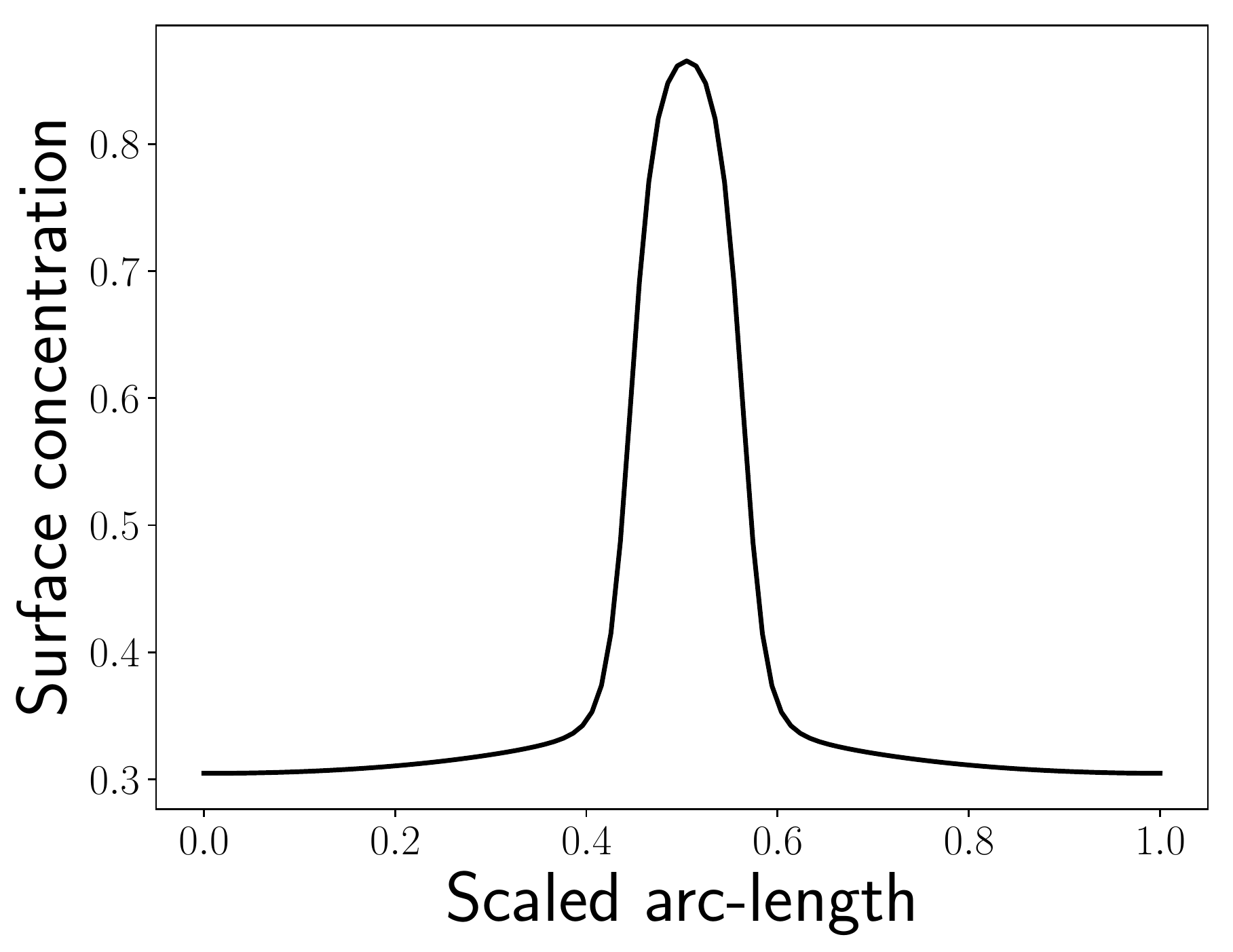}{0.24\textwidth}}
	  			\subfloat{%
	  			\label{fig:bswp-moving-nt3e3-sf-c}%
	    				\inc{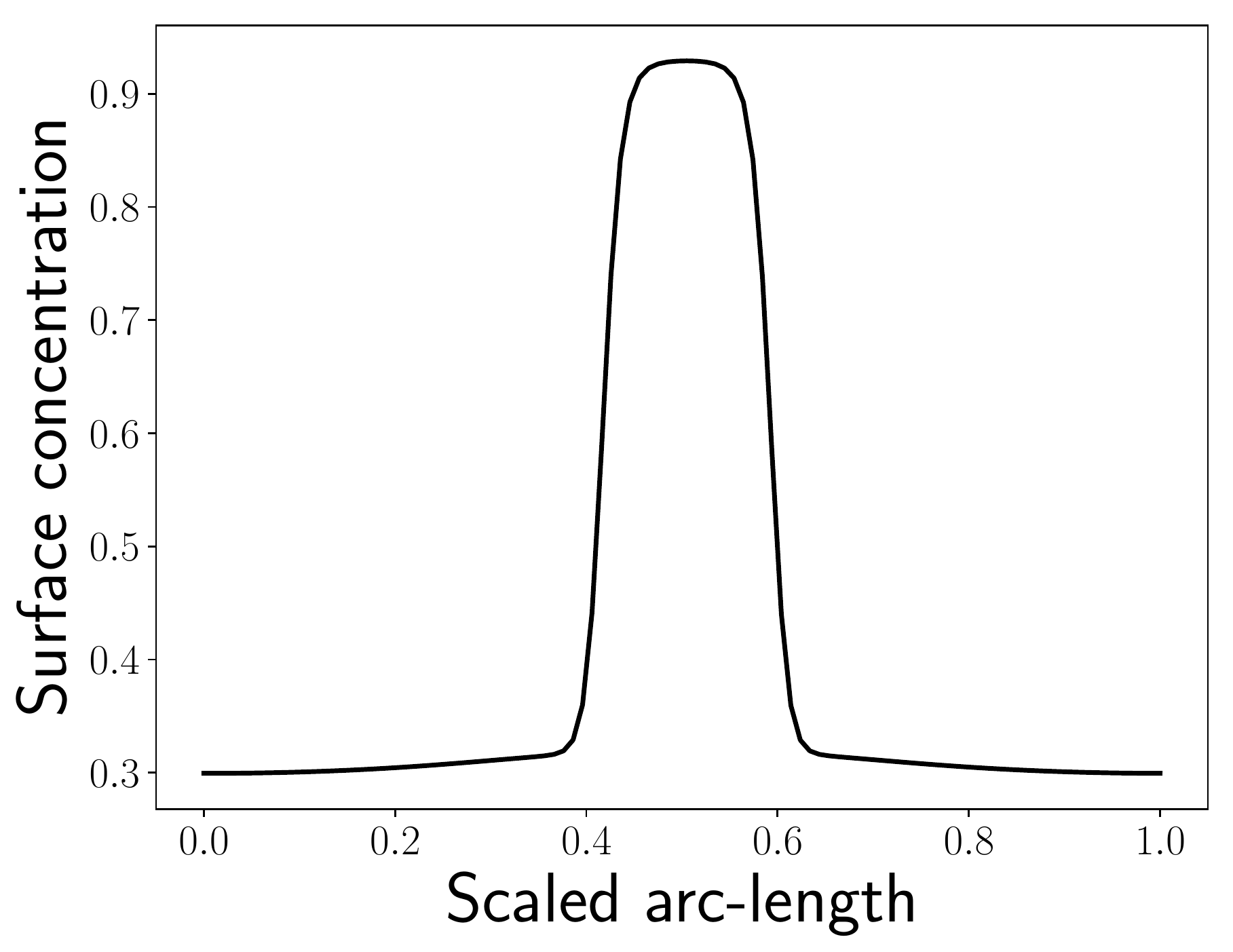}{0.24\textwidth}}
			  	\subfloat{%
	    			\label{fig:bswp-moving-nt3e3-sf-d}%
	    				\inc{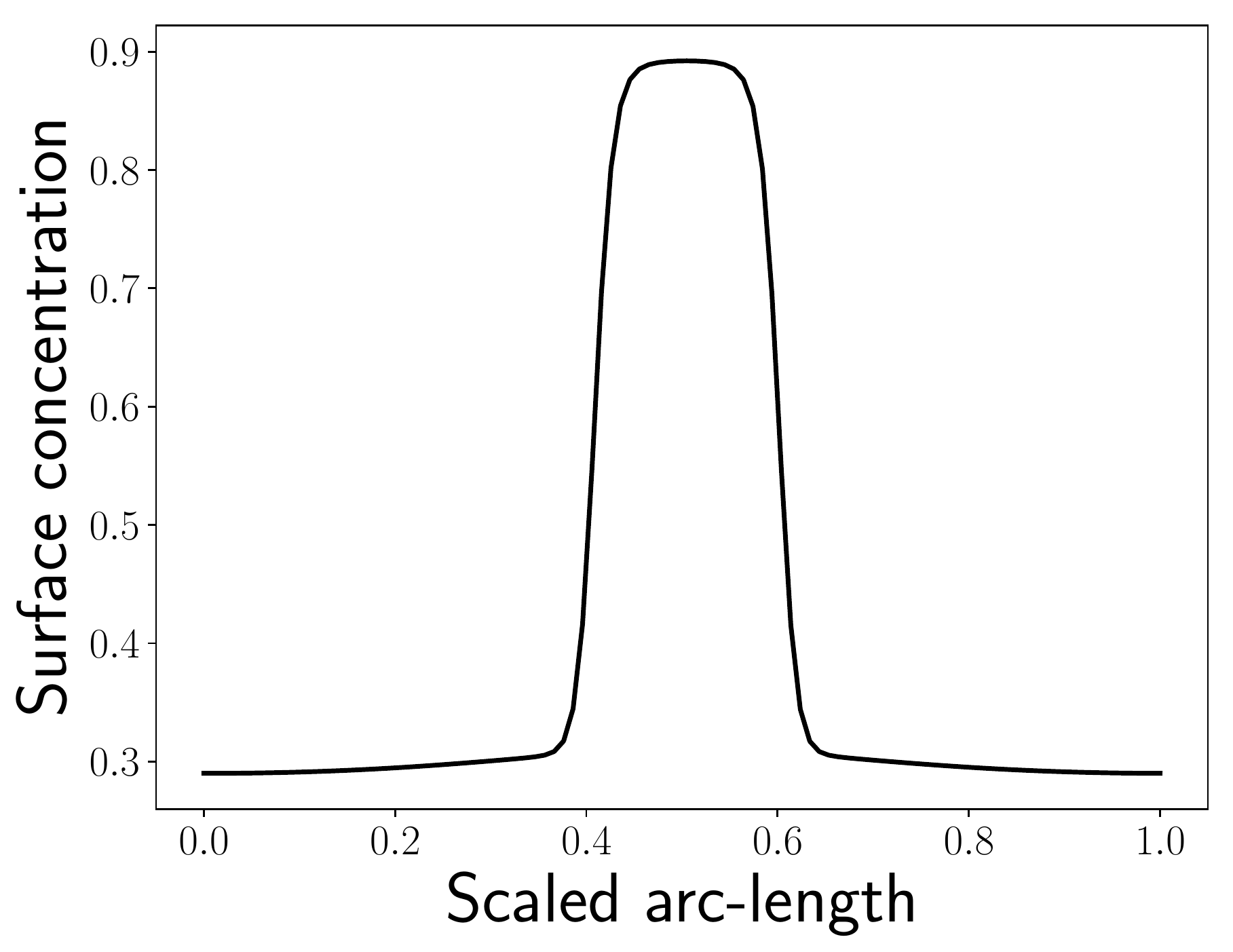}{0.24\textwidth}}\\
	  			\setcounter{subfigure}{0}
	  			\subfloat[Time $= 25$ $(s)$.]{%
	    			\label{fig:bswp-moving-nt3e3-blk-a}%
		    			\inc{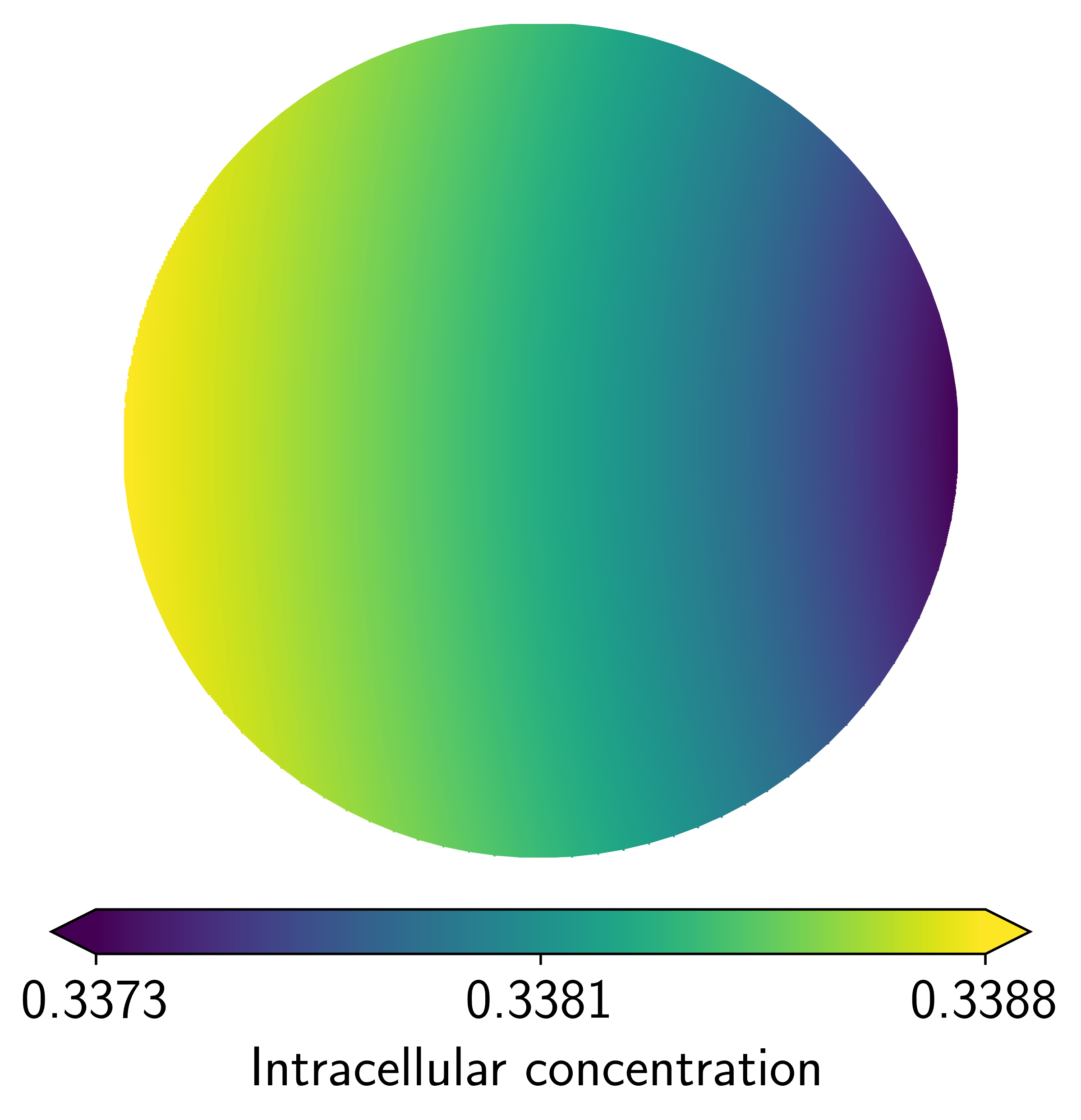}{0.24\textwidth}}
				\subfloat[Time $= 105$ $(s)$.]{%
	    			\label{fig:bswp-moving-nt3e3-blk-b}%
	    				\inc{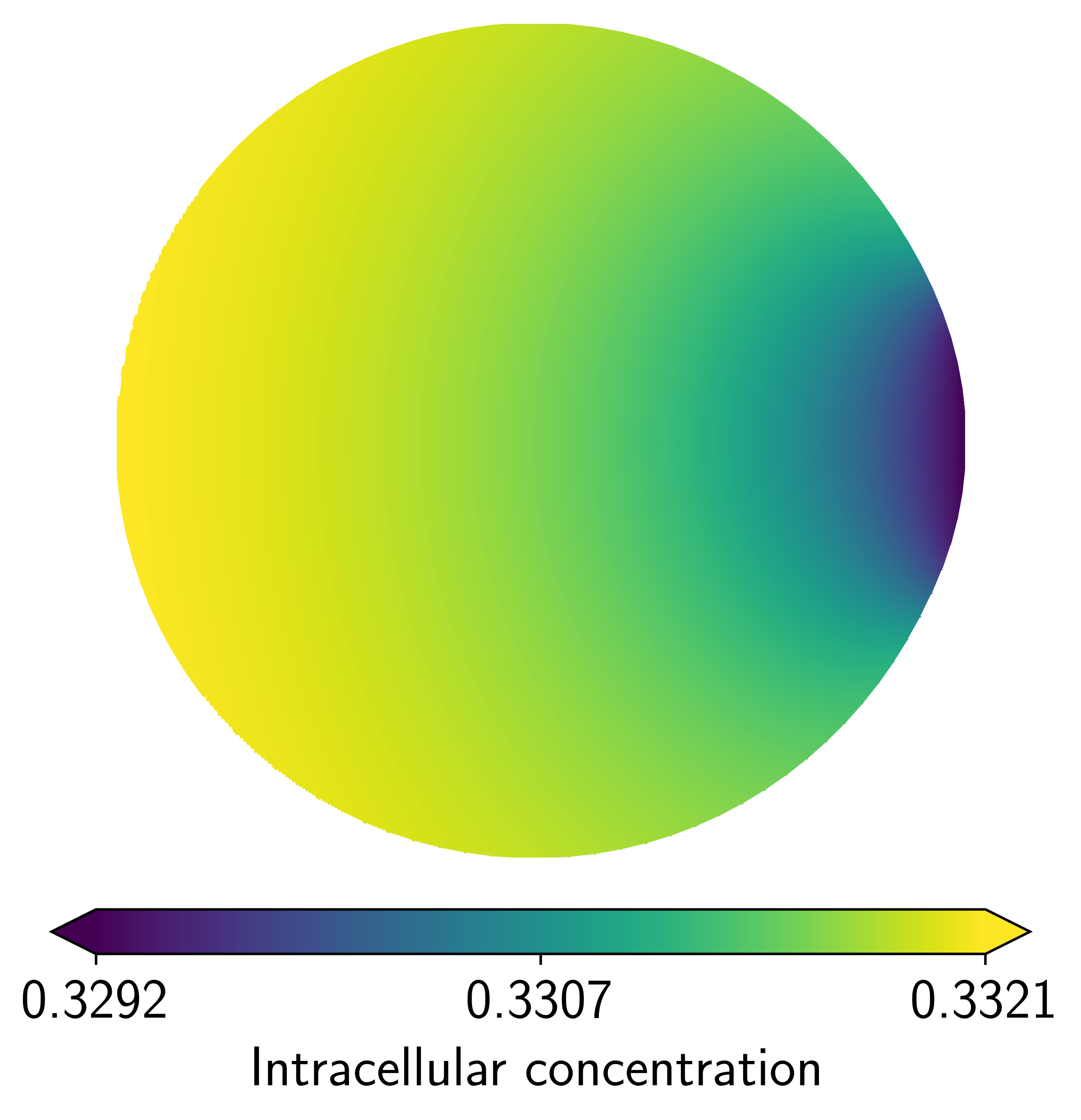}{0.24\textwidth}}
	  			\subfloat[Time $= 225$ $(s)$.]{%
	    			\label{fig:bswp-moving-nt3e3-blk-c}%
	    				\inc{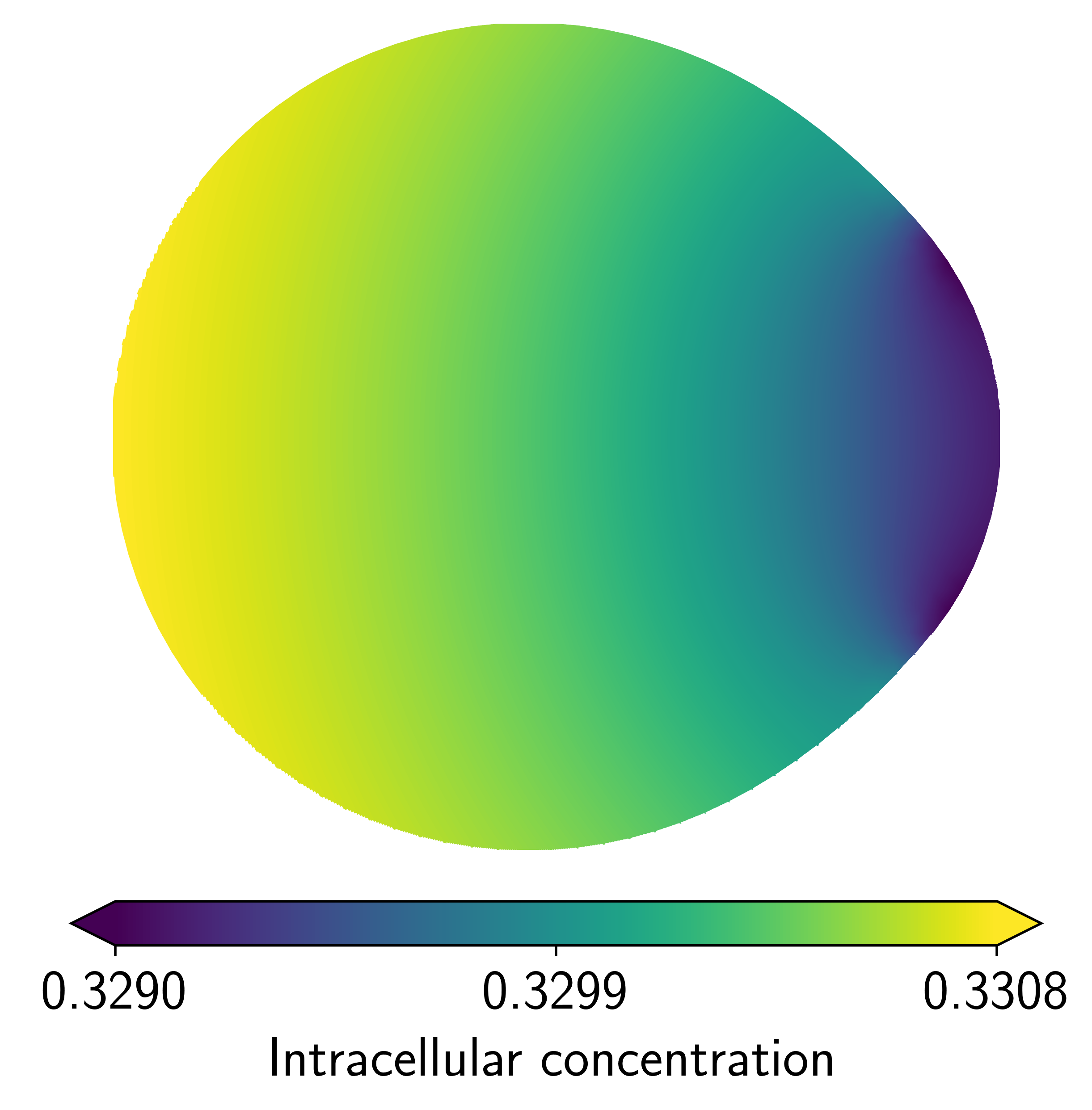}{0.24\textwidth}}
	  			\subfloat[Time $= 300$ $(s)$.]{%
	    			\label{fig:bswp-moving-nt3e3-blk-d}%
	    				\inc{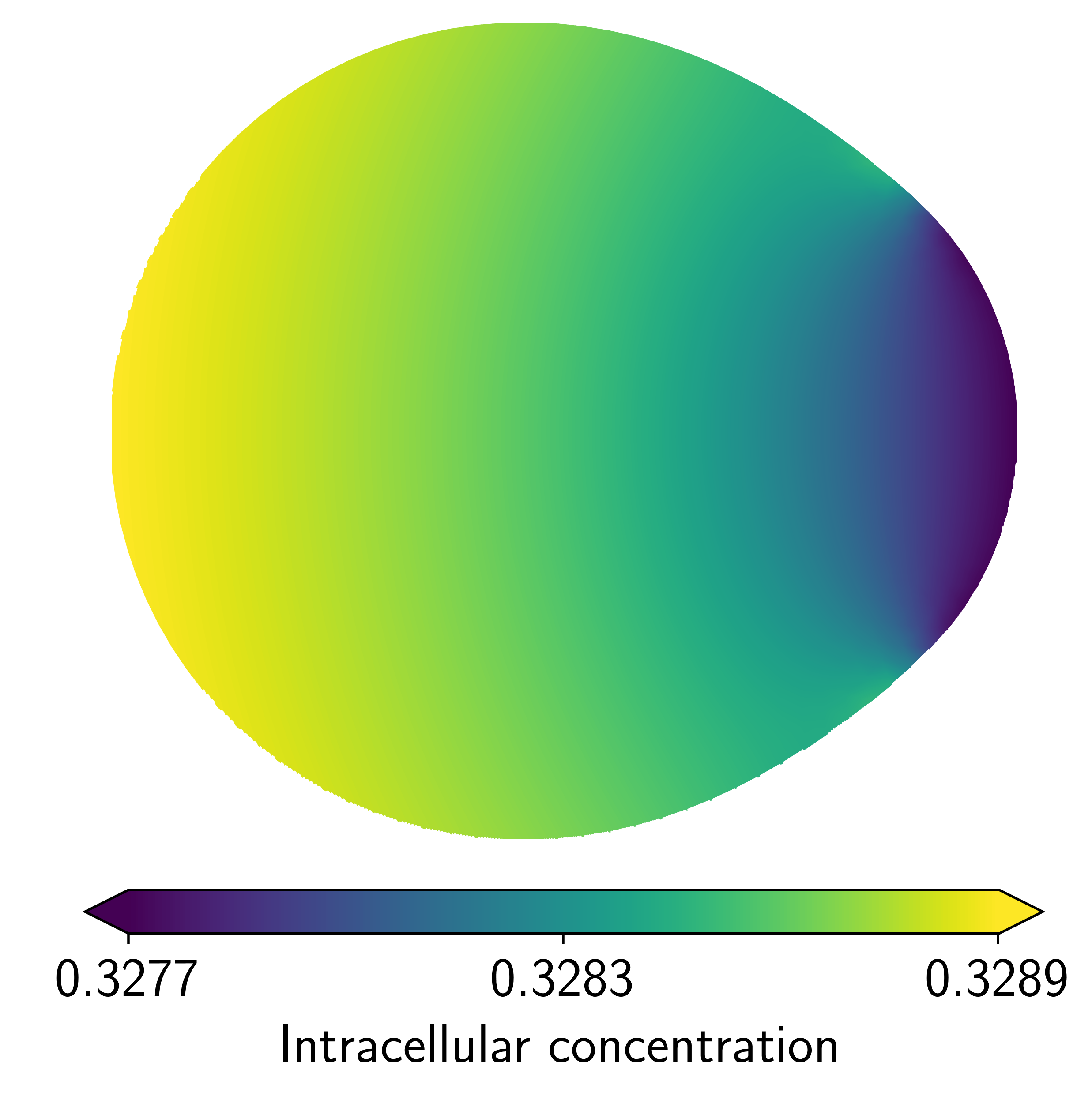}{0.24\textwidth}}
	  			\caption[]{%
	    				{\em (Bulk-surface wave-pinning on a moving domain)} Concentrations of membrane bound active state (top row) and cytoplasmic inactive state (bottom row) on a moving domain at times \subref{fig:bswp-moving-nt3e3-blk-a} $t = 25$, \subref{fig:bswp-moving-nt3e3-blk-b} $t = 105$, \subref{fig:bswp-moving-nt3e3-blk-c} $t = 225$, \subref{fig:bswp-moving-nt3e3-blk-d} $t = 300$.}
	  		\label{fig:bswp-moving-nt3e3-a}
			\end{figure}
			
			Figure \ref{fig:bswp-moving-both-blk} illustrates the extracellular ligand concentration and intracellular inactive concentration at later times than those considered in Fig. \ref{fig:bswp-moving-nt3e3-a}. In Fig. \ref{fig:bswp-moving-both-blk}\subref{fig:bswp-moving-both-blk-a} it is clear that the cell is migrating towards the point source located at ${\bfa x}_{p} = (15,0)$. There is a clear {\em cap} of inactivity in the intracellular environment in the direction of the point source, due to the cell protruding in that direction as was seen previously. Once the cell gets sufficiently close to the point source, the point source is moved and the active surface state migrates around the membrane so that it faces the point source. This causes the cytoplasmic {\em cap} of inactivity to also migrate so that it is facing the point source and this can be seen clearest in Figs. \ref{fig:bswp-moving-both-blk}\subref{fig:bswp-moving-both-blk-b} and \ref{fig:bswp-moving-both-blk}\subref{fig:bswp-moving-both-blk-d}. This type of migration is in contrast with the Meinhardt model \cite{meinhardt_1999,neilson_2011,macdonald_2016}, where one expects a new peak of activation to form in the direction of the point source. However, in the wave-pinning model, the activation remains but migrates round the membrane to face the point source (as demonstrated previously in the absence of a freely diffusing external guidance cue \cite{maree_2006, maree_2012}) which is a testament to the stability of the pattern formed by the wave-pinning mechanism. After turning to face the point source, the cell straightens and migrates persistently towards the point source where the point source is again moved when the cell gets sufficiently near (Figs. \ref{fig:bswp-moving-both-blk}\subref{fig:bswp-moving-both-blk-c}-\subref{fig:bswp-moving-both-blk-f}). Analagous to the previous section \S\ref{sec:bswp-stationary}, the colourbar axes in Figs. \ref{fig:bswp-moving-nt3e3-a} and \ref{fig:bswp-moving-both-blk} change as time progresses to emphasize the local cytoplasmic patterning as the cell migrates. Note that due to the robustness of the MMPDE procedure discussed in \S\ref{sec:meshgen-blk}, our algorithm is able to perform the full loop without the need to re-mesh. An illustration of the exterior and interior meshes during migration is given in Fig. \ref{fig:bswp-moving-misc}\subref{fig:bswp-moving-mesh}. Figure \ref{fig:nexp-cell-wp-wpwp-csv}\subref{fig:nexp-cell-wp-wplr-conserv} confirms that the fully discrete numerical solution is globally conservative to machine zero as proven in Theorem \ref{thm:conserv-fd}.
			
			\begin{figure}
				\centering
	  			\subfloat[Time $= 2400$ $(s)$.]{%
	    			\label{fig:bswp-moving-both-blk-a}%
	    				\inc{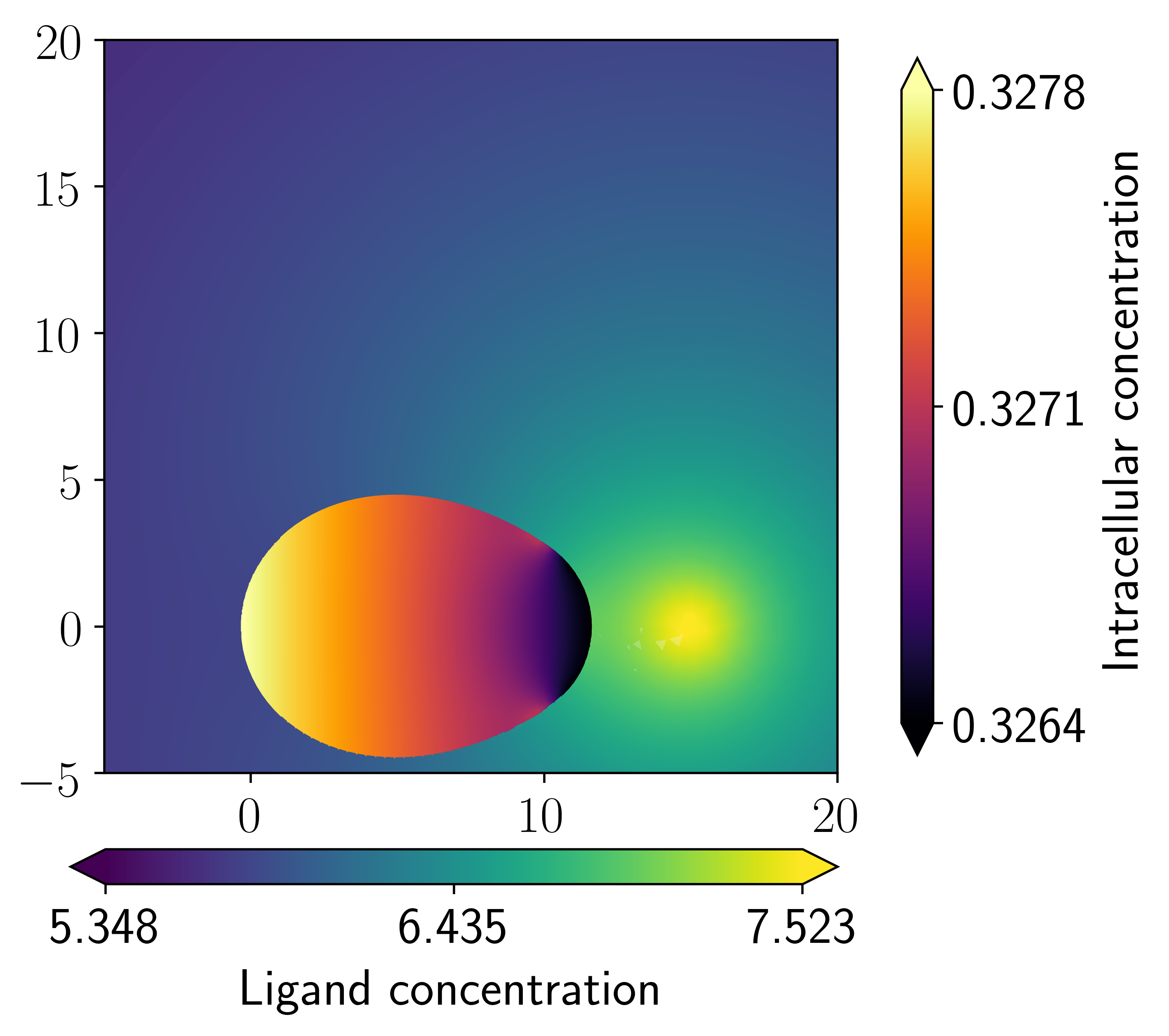}{0.4\textwidth}}\qquad\quad
	  			\subfloat[Time $= 4800$ $(s)$.]{%
	  			\label{fig:bswp-moving-both-blk-b}%
	  				\inc{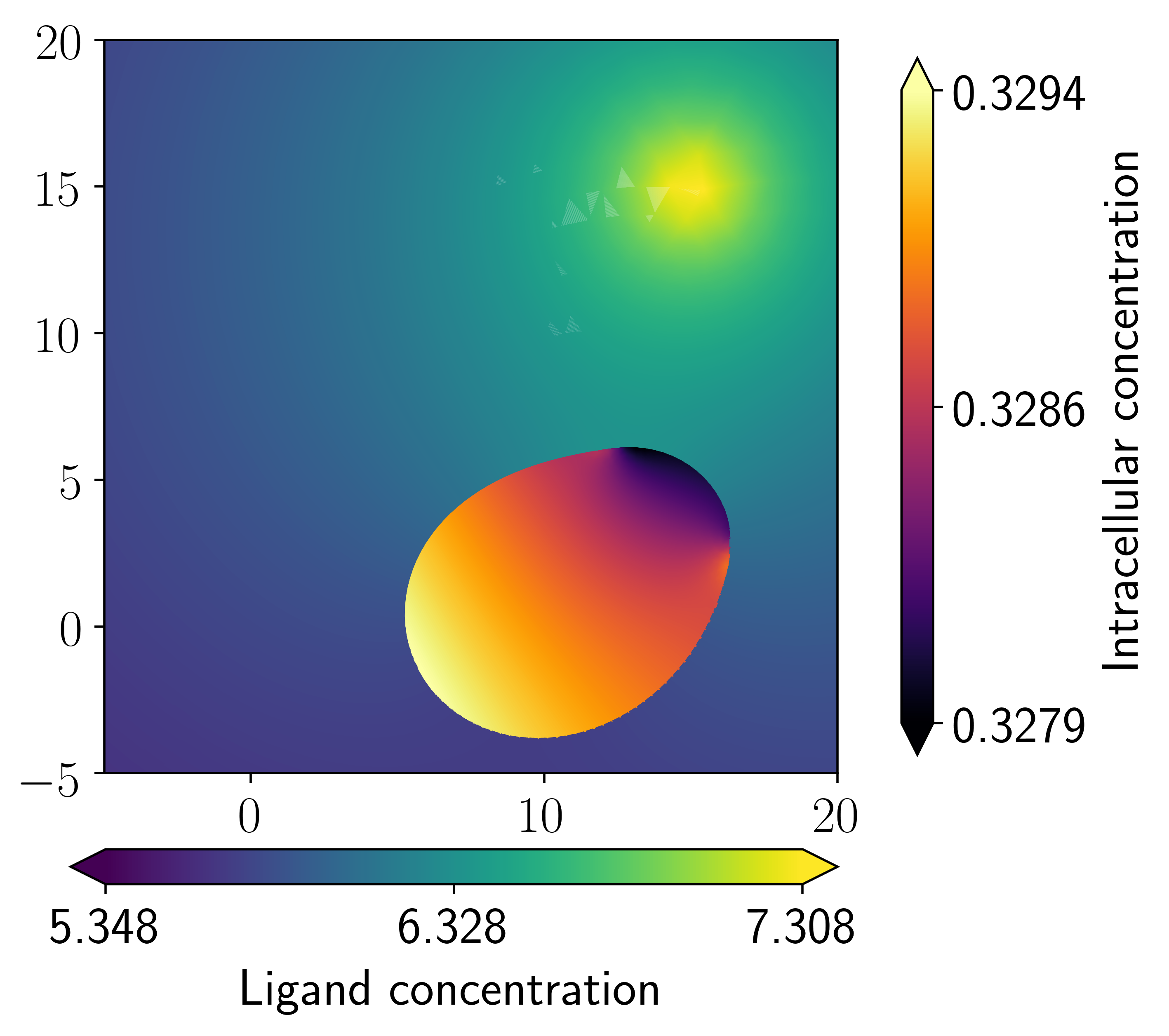}{0.4\textwidth}}\\
			  	\subfloat[Time $= 7200$ $(s)$.]{%
	    			\label{fig:bswp-moving-both-blk-c}%
	    				\inc{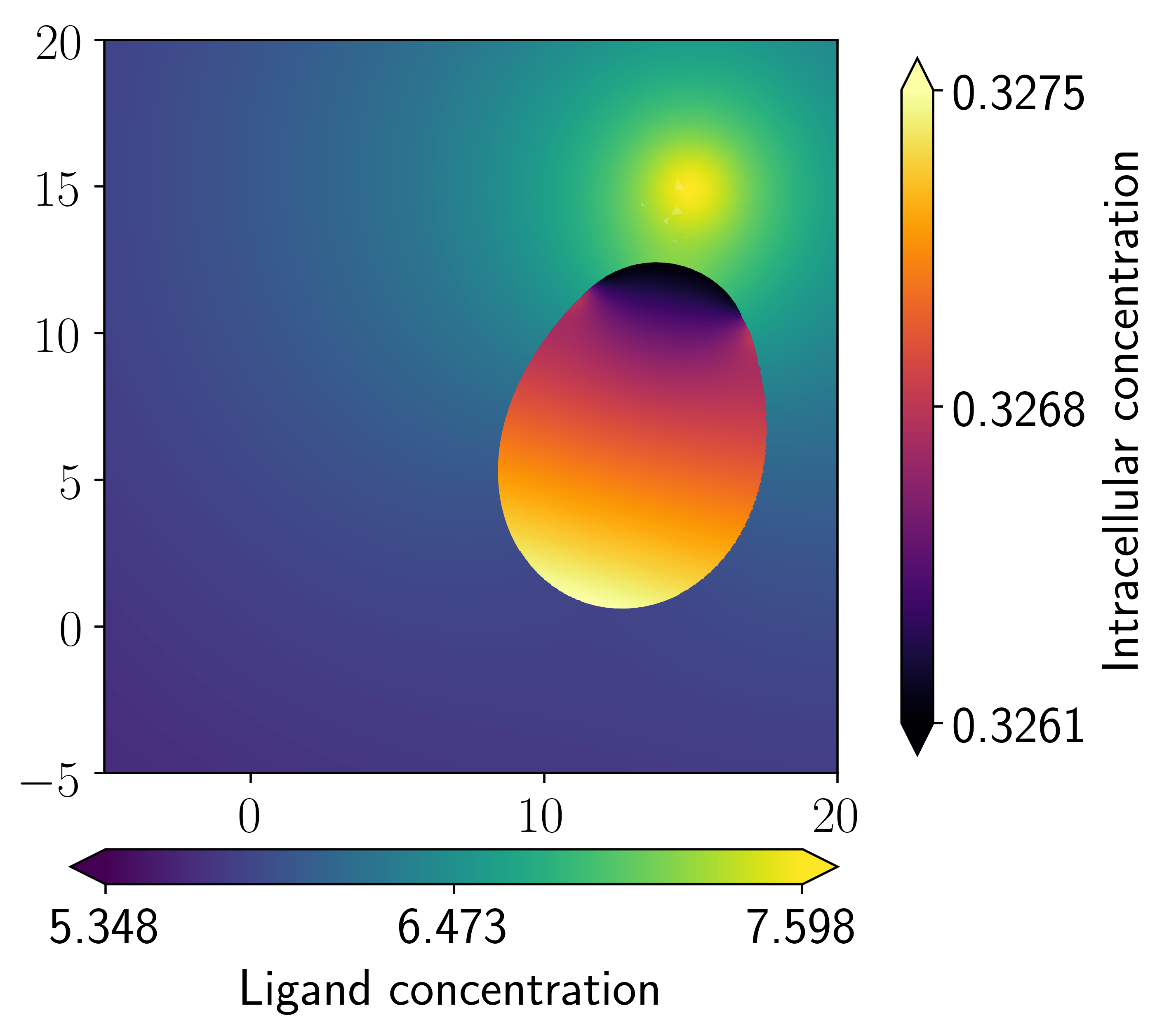}{0.4\textwidth}}\qquad\quad
	  			\subfloat[Time $= 9600$ $(s)$.]{%
	    			\label{fig:bswp-moving-both-blk-d}%
	    				\inc{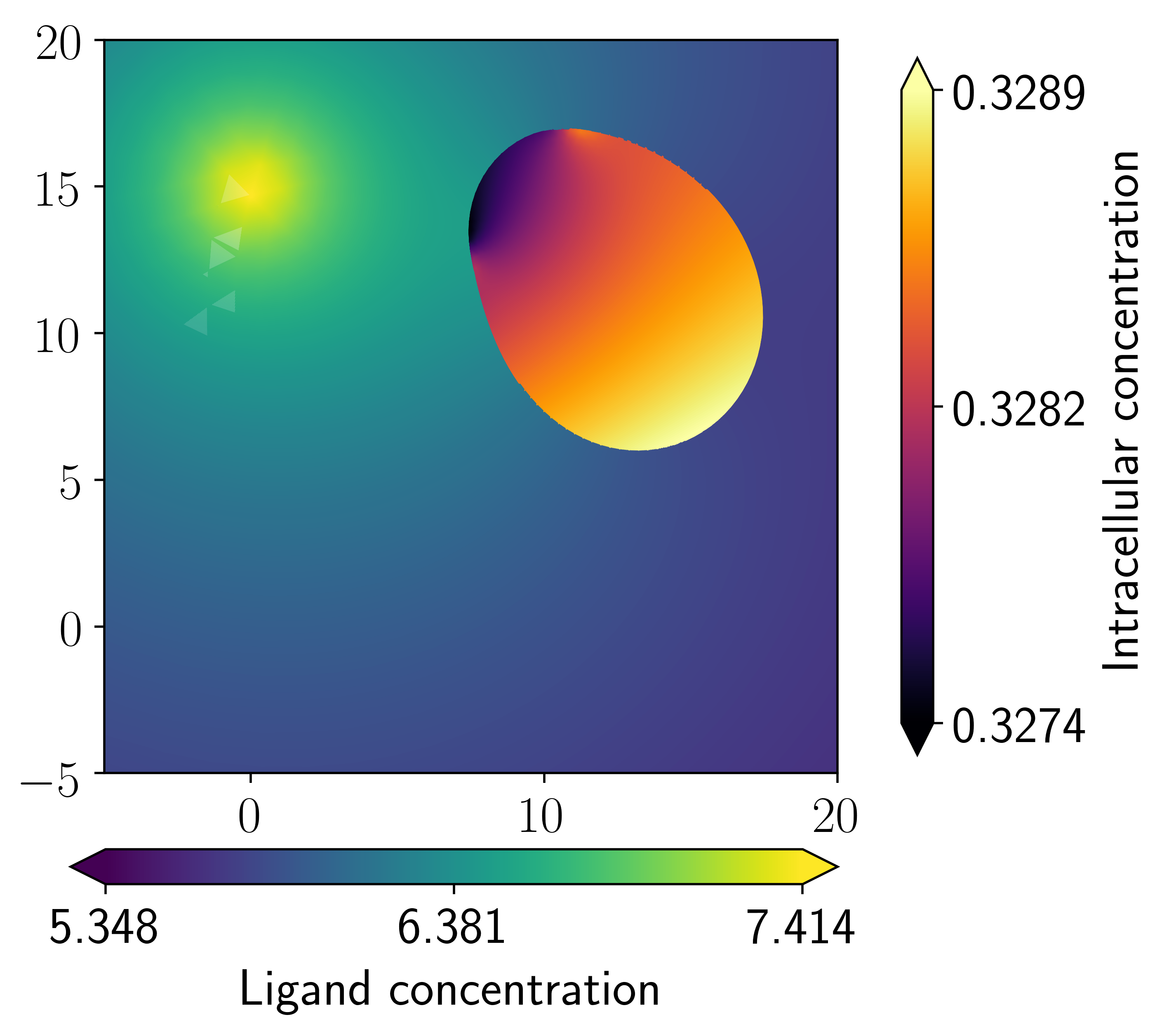}{0.4\textwidth}}\\
				\subfloat[Time $= 12000$ $(s)$.]{%
	    			\label{fig:bswp-moving-both-blk-e}%
		    			\inc{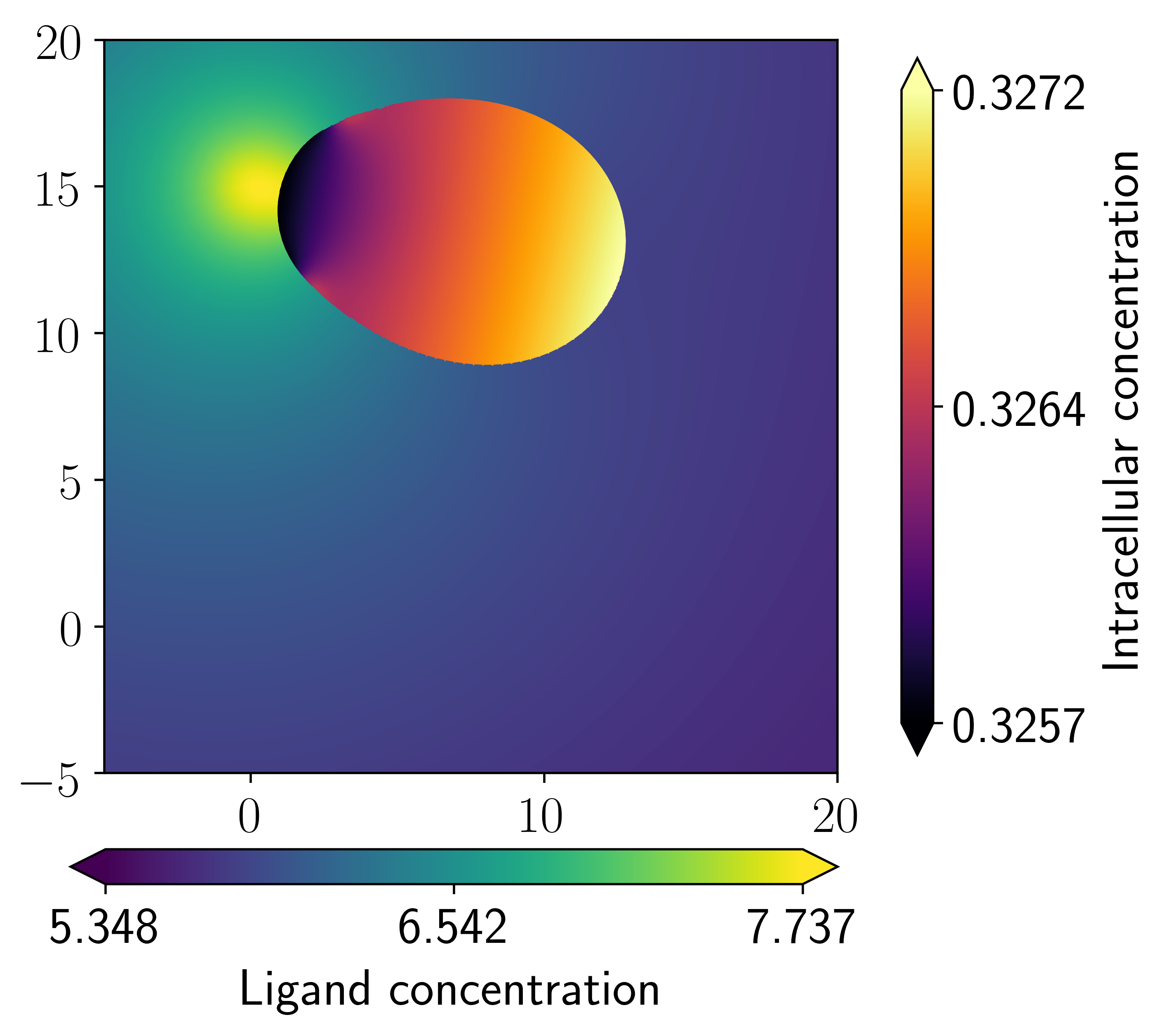}{0.4\textwidth}}\qquad\quad
	  			\subfloat[Time $= 14400$ $(s)$.]{%
	    			\label{fig:bswp-moving-both-blk-f}%
	    				\inc{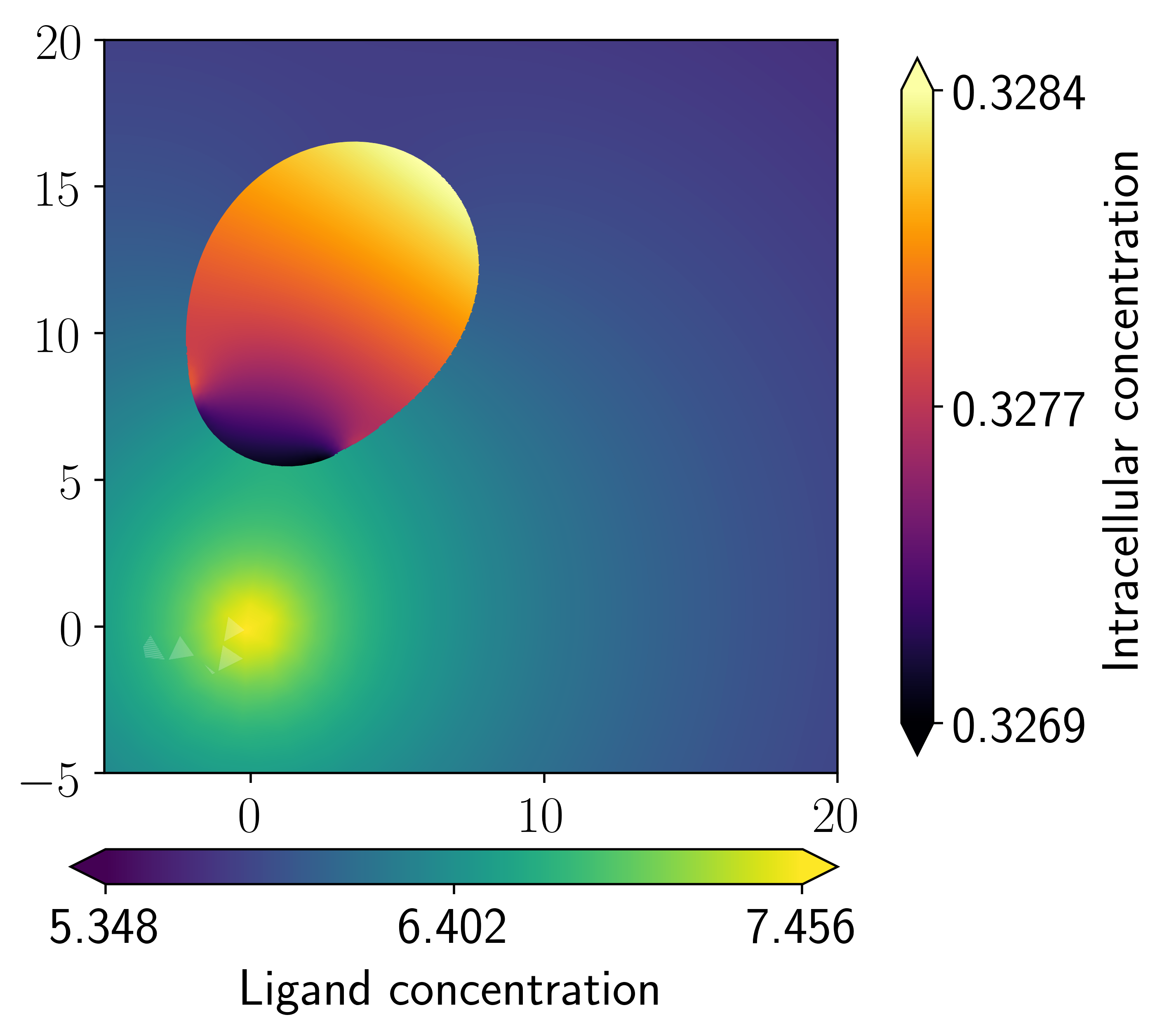}{0.4\textwidth}}\qquad\quad
	  			\caption[]{%
	    				{\em (Bulk-surface wave-pinning on a moving domain)} Concentrations of extracellular ligand and intracellular inactive state as the cell migrates towards a moving point source.}
	  		\label{fig:bswp-moving-both-blk}
			\end{figure}
			
			For the results depicted in Figs. \ref{fig:bswp-moving-nt3e3-a} and \ref{fig:bswp-moving-both-blk}, we assumed that ${\bfa u}_{\Gamma}\cdot {\bfa t} = 0$. This assumption is reasonable as the {\em shape} of an evolving curve is solely determined by its normal motion. However, in general ${\bfa u}_{\Gamma}$ is not known {\em {a-priori}} and therefore, ${\bfa u}_{\Gamma}\cdot {\bfa t} \neq 0$ is entirely possible. In the context of cell migration, the material velocity can be approximated using a mathematical model of the intracellular fluid dynamics and the acto-myosin network; see for example, \cite{camley_2013, camley_2017, moure_2016, moure_2018}. However, such an implementation would warrant significant changes to the algorithms presented in this article and is therefore, not considered here. Instead, we assume that material velocities for the intracellular and membrane environments are the same as the ALE velocity; that is, ${\bfa u}_{\Gamma} = {\bfa w}$ and ${\bfa u}_{\Omega} = {\bfa w}$. Hence, we have assumed that the inside of the cell behaves like an elastic solid. To emphasize the importance of the material velocity, we sharpen the Gaussian point source \eqref{eq:bswp-moving-gaussptsrc} by reducing the value of $\sg_l = 0.1$. A comparison of the migration of the cell when the material velocities are given by: ${\bfa u}_{\Gamma}\cdot{\bfa t} = 0$ and ${\bfa u}_{\Omega} = {\bfa 0}$; and ${\bfa u}_{\Gamma} = {\bfa w}$ and ${\bfa u}_{\Omega} = {\bfa w}$, is illustrated in Fig. \ref{fig:bswp-moving-misc}\subref{fig:bswp-moving-both-blk-matvel}. Both cells successfully migrate towards the initial location of the ligand point source ${\bfa x}_{p} = (15, 0)$ (which is depicted in Fig. \ref{fig:bswp-moving-misc}\subref{fig:bswp-moving-both-blk-matvel} as a yellow circle). Their migration is slower than depicted in Fig. \ref{fig:bswp-moving-both-blk} due to the sharper Gaussian point source. Once close enough, the point source location is moved to ${\bfa x}_{p} = (15, 15)$ (which is illustrated in Fig. \ref{fig:bswp-moving-misc}\subref{fig:bswp-moving-both-blk-matvel} as the extracellular contour plot). The cell with material velocities ${\bfa u}_{\Gamma}\cdot{\bfa t} = 0$ and ${\bfa u}_{\Omega} = {\bfa 0}$, cannot turn sharply enough to be able to migrate towards the new point source location (right-most cell in Fig. \ref{fig:bswp-moving-misc}\subref{fig:bswp-moving-both-blk-matvel}). Therefore, the cell migrates past the location of the point source. The solid black line in Fig. \ref{fig:bswp-moving-misc}\subref{fig:bswp-moving-both-blk-matvel} depicts the evolution of the cell centroid in time. However, the cell with material velocities ${\bfa u}_{\Gamma} = {\bfa w}$ and ${\bfa u}_{\Omega} = {\bfa w}$ (left-most cell in Fig. \ref{fig:bswp-moving-misc}\subref{fig:bswp-moving-both-blk-matvel}) can turn sharply and in fact, gets sufficiently close to the point source that it moves to a new location ${\bfa x}_{p} = (0,15)$ (depicted as a yellow square in Fig. \ref{fig:bswp-moving-misc}\subref{fig:bswp-moving-both-blk-matvel}). The black dashed line in Fig. \ref{fig:bswp-moving-misc}\subref{fig:bswp-moving-both-blk-matvel} depicts the evolution of the cell centroid in time. The material velocity ${\bfa u}_{\Omega} = {\bfa w}$ ensures that the inactive cytoplasmic species is transported with the cell as it migrates. This produces a more uniform distribution on the inside of the cell (this can be seen in Fig. \ref{fig:bswp-moving-misc}\subref{fig:bswp-moving-both-blk-matvel}). Consequently, the inactive species is more readily available to be converted to membrane bound active species, which in turn enables the cell to adapt more efficiently. The depleted regions found in the cytoplasm local to a protrusion (see Figs. \ref{fig:bswp-moving-nt3e3-a} and \ref{fig:bswp-moving-both-blk}) can therefore act as a local inhibitor to migration.
			
			\begin{figure}
				\centering	    				
	  			\subfloat[Illustration of the meshes during migration.]{%
	    			\label{fig:bswp-moving-mesh}%
	    				\inc{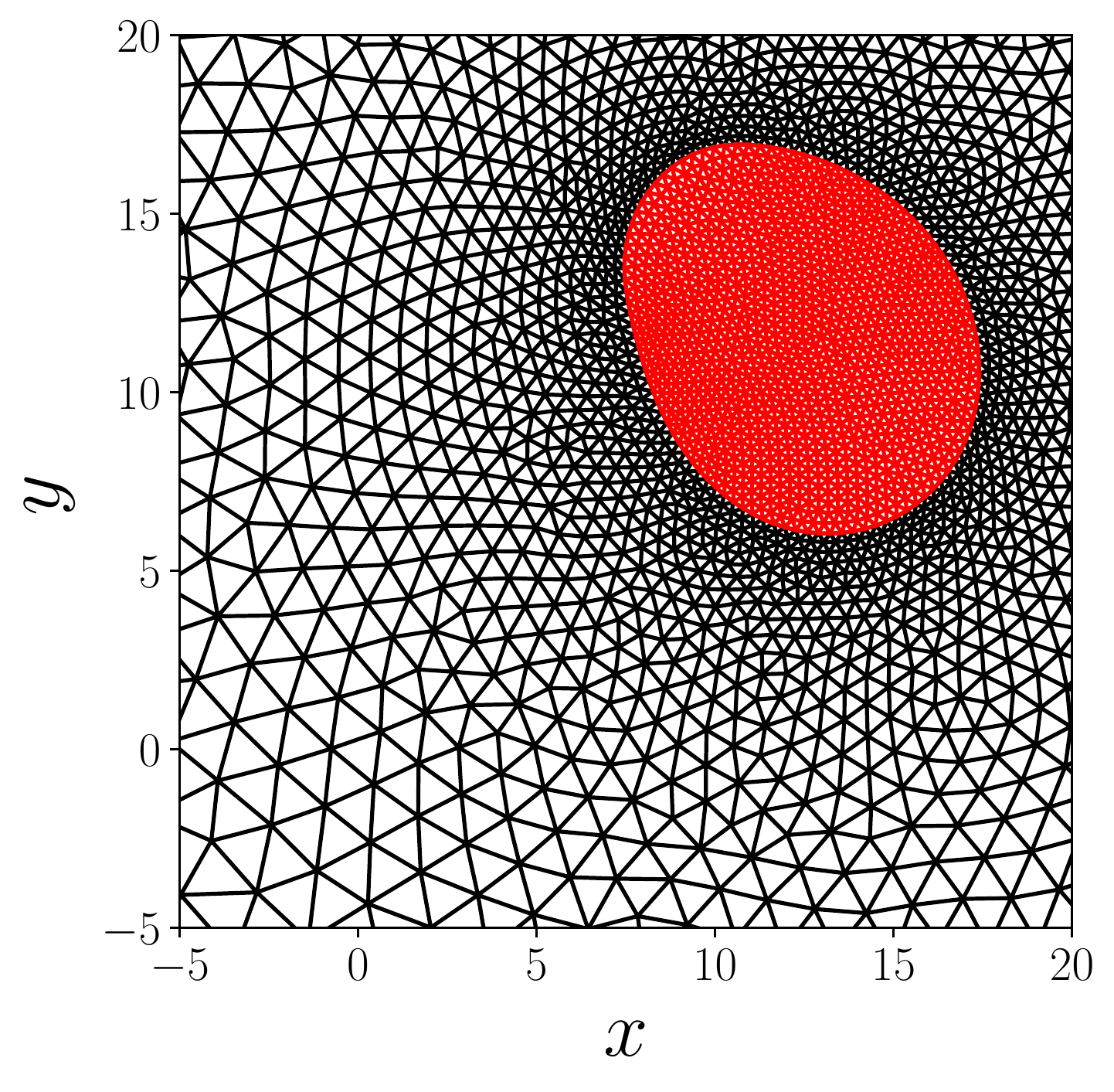}{0.4\textwidth}}\qquad\quad
	  			\subfloat[Migration with differing material velocities.]{%
	  			\label{fig:bswp-moving-both-blk-matvel}%
	    				\inc{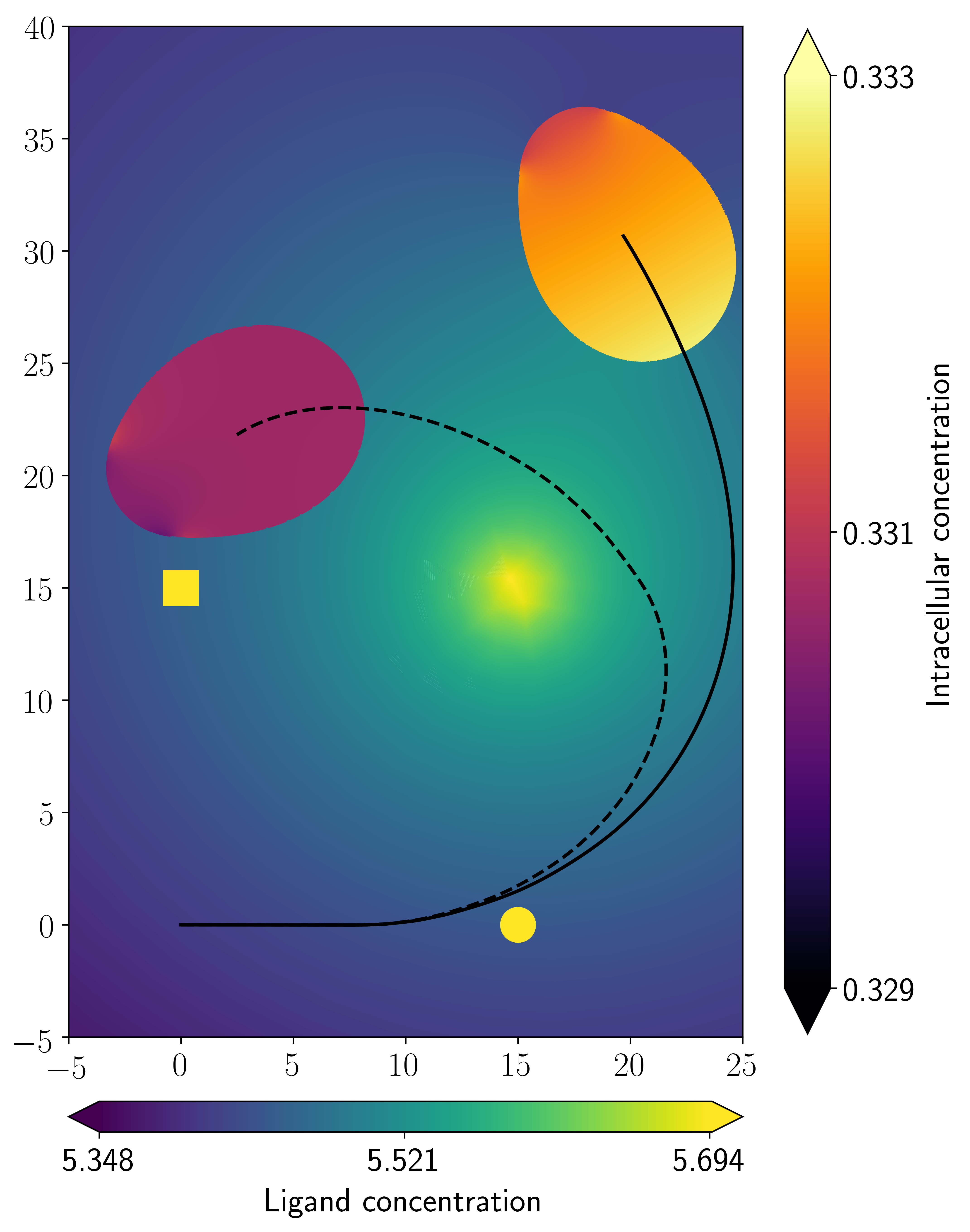}{0.42\textwidth}}
	  			\caption[]{%
	    				\subref{fig:bswp-moving-mesh} Illustration of the exterior and interior meshes during motion. \subref{fig:bswp-moving-both-blk-matvel} Comparison of the migration of two cells with different material velocities; the solid black line depicts the evolution of the cell centroid in time when ${\bfa u}_{\Gamma}\cdot{\bfa t} = 0$ and ${\bfa u}_{\Omega} = 0$ and the dashed black line depicts the evolution of the cell centroid in time when ${\bfa u}_{\Gamma} = {\bfa w}$ and ${\bfa u}_{\Omega} = {\bfa w}$; the yellow circle and square indicate the initial and third positions of the point source, respectively, whilst the extracellular contour plot indicates the second position of the point source; it is clear that when ${\bfa u}_{\Gamma} = {\bfa w}$ and ${\bfa u}_{\Omega} = {\bfa w}$ the cell can turn more effectively.}
	  		\label{fig:bswp-moving-misc}
			\end{figure}

  \section{Conclusions}
  \label{sec:concl}
    Bulk-surface mass-conservative reaction-diffusion systems are a promising framework for modelling GTPase cycling leading to cell polarisation and subsequently cell migration and chemotaxis. We have established theoretically that a recently developed finite element ALE scheme is globally conservative at the fully discrete level and this property holds independently of the chosen time step and of the domain movement. To our knowledge this is the first scheme to be shown to have this desirable property. Furthermore, we have shown that the numerical scheme appears to be second-order accurate when applied to test cases with known exact solutions. The method has been applied to a bulk-surface WP model for cell polarisation in combination with a receptor-ligand model to describe a signalling cascade from external stimulus to downstream cell protrusive activity. The modelled cell has been shown to display efficient chemotaxis even when the source of chemoattractant changes location.  Interestingly, numerical experiments suggest that the accuracy of a model cell to direct its motion can be affected by the bulk material velocity responsible for the transport of cytoplasmic proteins.  

The main theoretical result of this paper concerns the conservation property of the fully discrete ALE finite element scheme. Future work will focus on the numerical stability and convergence of the scheme. Previous studies on the stability of discretisations of reaction-diffusion equations on evolving and growing domains suggest that discretisations of conservative formulations of the governing equations lead to schemes with superior stability characteristics \cite{mackenzie_2007, mackenzie_2011, mackenzie_2012}. While we have used a relatively simple model for ligand-receptor binding to demonstrate the possibilities of the computational framework, the same techniques could be used to investigate the effect of local extracellular ligand depletion by membrane bound enzymes and intracelluar signalling pathways \cite{mackenzie_2016}. The method could also be used to consider the effect of cell shape changes on intracellular signalling from the cell cytoplasm to the cell nucleus \cite{giese_2018, cangiani_2010, sturrock_2012}.  In future work we also plan to couple the GTPase module to a mechanical model for the bulk material velocity which will be driven by the intracellular forces generated by the interaction of actin and myosin and substrate adhesion. Models of this type range in sophistication from purely viscous flow models \cite{shao_2012, george_2013} to models which take into consideration the orientation of actin filaments \cite{kruse_2004, kruse_2005, marth_2015}.

  \section*{Acknowledgements}
    CFR was supported by a Cancer Research UK Multidisciplinary Project Award C22713/A20017. Core support for RHI was provided by Cancer Research UK, grant number 15672.


  \bibliographystyle{abbrv}
  
\end{document}